\documentclass[reqno]{amsart}%
\usepackage[utf8]{inputenc}%
\usepackage[english]{babel}%
\usepackage{amsmath,amssymb,amsthm,amsfonts}%
\usepackage{hyperref}%
\usepackage{enumerate}%
\usepackage{graphicx}
\usepackage{mathrsfs}
\usepackage{bbold}
\usepackage{color}
\usepackage{bbm}
\allowdisplaybreaks%
\numberwithin{equation}{section}%
\newcommand{\la}{\lambda}
\newcommand{\La}{\Lambda}
\newcommand{\te}{\theta}
\newcommand{\al}{\alpha}
\newcommand{\si}{\sigma}
\newcommand{\be}{\beta}
\newcommand{\Z}{\mathbb{Z}}
\newcommand{\R}{\mathbb{R}}
\newcommand{\C}{\mathbb{C}}
\newcommand{\Sb}{\mathbb{S}}
\newcommand{\Yb}{\mathbb{Y}}

\newcommand{\Gb}{\mathbb{G}}
\newcommand{\Tb}{\mathbb{T}}

\newcommand{\PT}{\mathbb{Pt}}
\newcommand{\km}{\boldsymbol\kappa}
\newcommand{\Pl}{\mathsf{Pl}}
\renewcommand{\i}{\mathsf{i}}
\renewcommand{\j}{\mathsf{j}}
\newcommand{\da}{\downarrow}

\newcommand{\ua}{\uparrow}

\newcommand{\slf}{\mathfrak{sl}}
\newcommand{\suf}{\mathfrak{su}}
\newcommand{\Sym}{\mathfrak{S}}
\newcommand{\ellf}{\ell_{\mathrm{fin}}}
\newcommand{\un}[1]{\underline{{#1}}}
\newcommand{\uno}[1]{\underline{{\rule[-1pt]{0pt}{0pt}#1}}}
\newcommand{\unt}[1]{\underline{{\rule[-2pt]{0pt}{0pt}#1}}}
\newcommand{\rep}{\mathcal{R}}
\newcommand{\Uf}{\mathsf{U}}
\newcommand{\Df}{\mathsf{D}}
\newcommand{\Hf}{\mathsf{H}}
\newcommand{\ufunc}{\mathsf{q}}

\newcommand{\zz}{\mathsf{t}}
\newcommand{\n}{{\boldsymbol n}}
\newcommand{\labf}{{\boldsymbol\la}}
\newcommand{\Mf}{\mathfrak{M}}
\newcommand{\meix}{\mathrm{M}}

\newcommand{\Cf}{\mathfrak{C}}
\newcommand{\Lf}{\mathfrak{L}}
\newcommand{\Ff}{\mathfrak{F}}
\newcommand{\Pc}{\mathcal{P}}
\renewcommand{\Mc}{\mathcal{M}}
\newcommand{\Tc}{\mathcal{T}}
\newcommand{\A}{\mathbb{A}}
\newcommand{\qf}[1]{\left\langle#1\right\rangle}
\newcommand{\Ga}{\Gamma}
\newcommand{\ga}{\gamma}
\newcommand{\Prob}{\mathop{\mathrm{Prob}}}
\newcommand{\po}{\mathfrak{p}}
\newcommand{\istry}{\mathcal{I}}
\newcommand{\fun}{\mathop{\mathrm{Fun}}}
\newcommand{\G}{\tilde{\mathsf{G}}}
\newtheorem{proposition}{Proposition}[section]
\newtheorem{lemma}[proposition]{Lemma}
\newtheorem{corollary}[proposition]{Corollary}
\newtheorem{theorem}[proposition]{Theorem}
\theoremstyle{definition}

\newtheorem{definition}[proposition]{Definition}
\newtheorem{remark}[proposition]{Remark}

\newtheorem{remarks}[proposition]{Remarks}
\begin{document}
\title{$\slf(2)$ Operators and Markov Processes on Branching Graphs}
\author{Leonid Petrov}
%\date{\today}
\thanks{The author was partially supported by the RFBR-CNRS grants 10-01-93114 and 11-01-93105, and by the Dynasty foundation fellowship for young scientists}
\address{Department of Mathematics, Northeastern University, 360 Huntington ave., Boston, MA 02115, USA}
\address{Dobrushin Mathematics Laboratory, Kharkevich Institute for Information Transmission Problems, Moscow, Russia}
\email{lenia.petrov@gmail.com}

\begin{abstract}
  We present a unified approach to various examples of Markov dynamics on partitions studied by Borodin, Olshanski, Fulman, and the author. Our technique generalizes the Kerov's operators first appeared in \cite{Okounkov2001a}, and also stems from the study of duality of graded graphs in \cite{fomin1994duality}. 
  
  Our main object is a countable branching graph carrying an $\slf(2,\C)$-module of a special kind. Using this structure, we introduce distinguished probability measures on the floors of the graph, and define two related types of Markov dynamics associated with these measures. We study spectral properties of the dynamics, and our main result is the explicit description of eigenfunctions of the Markov generator of one of the processes.

  For the Young graph our approach reconstructs the $z$-measures on partitions and the associated dynamics studied by Borodin and Olshanski \cite{Borodin2006}, \cite{Borodin2007}. The generator of the dynamics of \cite{Borodin2006} is diagonal in the basis of the Meixner symmetric functions introduced recently in \cite{Olshanski2010LaguerreMeixner}, \cite{Olshanski2011Meixner}. We give new proofs to some of the results of these two papers. Other graphs to which our technique is applicable include the Pascal triangle, the Kingman graph (with the two-parameter Poisson-Dirichlet measures), the Schur graph and the general Young graph with Jack edge multiplicities.
\end{abstract}

\maketitle
\setcounter{tocdepth}{1}
\tableofcontents
\setcounter{tocdepth}{2}

\section{Introduction} % (fold)
\label{sec:introduction}

The study of \emph{branching graphs} is interesting from the point of view of algebraic combinatorics, representation theory and probability. The basic example of a branching graph is the celebrated Young graph $\Yb=\bigsqcup_{n=0}^\infty\Yb_n$ which is the lattice of integer partitions\footnote{We always identify partitions with corresponding Young diagrams as in \cite[Ch. I, \S1]{Macdonald1995}.} ordered by inclusion (here $\Yb_n$ is the set of all partitions of $n$). The general definition of a branching graph is given below in \S\ref{sec:ideal_branching_graphs},
%\footnote{In fact, we deal only with more special \emph{ideal branching graphs} which are lattices of ideals of a poset (partially ordered set). For the Young graph the underlying poset would be $\Z_{>0}^2$.} 
and in the Introduction for simplicity we only speak about the Young graph. Let us briefly indicate main steps in understanding the Young graph, and recall constructions of distinguished measures on it and related stochastic dynamics. This will necessary motivations. We describe our results in \S \ref{sub:results}.

\subsection{Young graph}

The Young graph originated in the study of representations of finite symmetric groups $\Sym(n)$. The irreducible representations of $\Sym(n)$ are parametrized by $\Yb_n$, and the edges in the Young graph come from the classical Young branching rule: two Young diagrams are connected by an edge iff one is obtained from the other by removing one box.

\subsubsection{Boundary of the Young graph}
\label{ssub:Boundary of the Young graph}

Thoma in \cite{Thoma1964} described normalized indecomposable characters of the infinite symmetric group $\Sym(\infty)$. Earlier this result was independently obtained in a different form by Aissen, Edrei, Schoenberg, and Whitney \cite{AESW51} (see also \cite{Edrei1952}). It was shown by Vershik and Kerov  \cite{VK1981Characters}, \cite{VK81AsymptoticTheory} that Edrei--Thoma theorem is equivalent to describing all coherent systems of probability measures $\{M_n\}$ on floors of the Young graph (see \S\ref{sub:coherent_systems} for definition). The set of all such systems is convex, and its extreme points (corresponding to indecomposable characters) are identified with points of the infinite-dimensional \emph{Thoma simplex}
\begin{equation}\label{Thoma_simplex}
  \Omega:=
  \Big\{(\al;\be)\in\R^{\infty+\infty}\colon \al_1\ge\al_2\ge
  \dots\ge0,\; 
  \be_1\ge\be_2\ge
  \dots\ge0,\ 
  \sum\nolimits_{i=1}^\infty
  \al_i+\be_i\le1\Big\}.
\end{equation}
Any coherent system of measures on floors of $\Yb$ is uniquely expressed as a convex combination of the extreme ones, thus yielding a Borel probability measure on $\Omega$. Vice versa, any Borel probability measure on $\Omega$ leads to a coherent system $\{M_n\}$ on $\Yb$. It is said that $\Omega=\partial\Yb$ is the \emph{boundary} of the Young graph. For more details about the notion of the boundary of a branching graph see \cite{Kerov1998}.

\subsubsection{Distinguished coherent systems}
\label{ssub:Distinguished harmonic functions}
The next step in the representation theory of the infinite symmetric group was to find suitable analogues of the (two-sided) regular representation\footnote{It turns out that the two-sided regular representation of $\Sym(\infty)$ itself is irreducible.} and decompose them into irreducibles (the problem of harmonic analysis on $\Sym(\infty)$). This was done by Kerov, Olshanski and Vershik \cite{Kerov1993}, \cite{Kerov2004}. They introduced generalized irreducible representations of $\Sym(\infty)$ which depend on a complex parameter $z$. These representations give rise to a one-parameter family of probability measures on $\Omega$, and thus (via the boundary correspondence) to a distinguished family of coherent systems on $\Yb$. The set of parameters for these measures can be extended, and in this way one gets the remarkable $z$-measures on partitions which depend on two complex parameters $z,z'$ subject to certain constraints (e.g., see \cite[\S1]{Borodin2000a}). The generalized regular representations correspond to the case $z'=\bar z$. The $z$-measures on partitions are denoted by $M_n^{zz'}$.

\subsubsection{Study of $z$-measures on partitions}
\label{ssub:Study of $z$-measures}
The $z$-measures were studied extensively by Borodin, Olshanski, Okounkov, and other authors. It was shown that these measures give rise to determinantal random point processes on the discrete lattice $\Z':=\Z+\frac12$ and on the punctured real line $\R\setminus\{0\}$ \cite{Borodin2000a}, and their correlation kernels were explicitly computed. These point processes can also be viewed as discrete analogues of the $\beta=2$ random matrix ensembles \cite{Borodin1999RSK}. See \cite{Borodin1998}, \cite{borodin2000harmonic}, \cite{Okounkov2001a}, \cite{Borodin2005}, \cite{borodin2006meixner} for more details.

\subsubsection{Stochastic dynamics associated with the $z$-measures}
\label{ssub:Stochastic dynamics associated with the z-measures}

In the present paper we consider two types of Markov dynamics which generalize known models on the Young graph we are about to describe.

The first model is a sequence of Markov chains on the floors $\Yb_n$ of $\Yb$ (i.e., the $n$th chain lives on Young diagrams with $n$ boxes) which preserve the $z$-measures $M_n^{zz'}$. One step of the $n$ up/down chain consists of relocating one box in the Young diagram from one place to another. For the Plancherel measures on partitions\footnote{Which arise in the $z,z'\to\infty$ limit from the $z$-measures.} an almost indistinguishable model of down/up Markov chains first appeared in \cite{Fulman2005}. For the $z$-measures the up/down Markov chains were the main object of \cite{Borodin2007}, where their asymptotic behavior was analyzed. This led to a family of (infinite-dimensional) diffusion processes on the simplex $\Omega$. In \cite{Fulman2007} spectral properties of the down/up Markov chains for the $z$-measures (along with up/down chains for other measures on other branching graphs) were studied.

The second model which we generalize in the present paper is a continuous-time Markov jump dynamics on the whole Young graph $\Yb$. One jump of this process consists of adding one box to the Young diagram or deleting one box from it. This process was introduced in \cite{Borodin2006}. It preserves a certain mixture $M^{zz'}$ of the $z$-measures over the index $n$, see \cite[\S1]{Borodin2000a}, and \S\ref{sub:mixing_of_measures}. Denote this process (starting from the invariant distribution) by $\labf^{zz'}(t)$. It was shown in \cite{Borodin2006} that this dynamics is (space-time) determinantal, and the dynamical correlation kernel was written out explicitly.

\subsubsection{Laguerre and Meixner symmetric functions}
\label{ssub:Laguerre and Meixner symmetric functions}

Olshanski \cite{Olshanski2010LaguerreMeixner}, \cite{Olshanski2011Meixner} constructed an orthogonal basis in the Hilbert space $\ell^2(\Yb,M^{zz'})$ of square integrable functions with respect to the measure $M^{zz'}$ consisting of eigenfunctions of the Markov generator of the process $\labf^{zz'}(t)$. This basis is indexed by all partitions. These functions are related to the classical Meixner orthogonal polynomials (e.g., see \cite[\S1.9]{Koekoek1996} for definition), and belong to the algebra of symmetric functions \cite[Ch. I]{Macdonald1995} suitably realized as a subspace of $\ell^2(\Yb,M^{zz'})$. They are called the \emph{Meixner symmetric functions}.

In a certain limit transition the dynamics $\labf^{zz'}(t)$ becomes a diffusion process living on a cone over the Thoma simplex $\Omega$. The basis of eigenfunctions of the generator of the limiting dynamics is formed by the \emph{Laguerre symmetric functions} \cite{Olshanski2010LaguerreMeixner}, \cite{Olshanski2011Meixner}.

\subsection{Methods for the Young graph}

Let us briefly indicate existing approaches to the results cited in \S\ref{ssub:Distinguished harmonic functions}, \S\ref{ssub:Stochastic dynamics associated with the z-measures}, and \S\ref{ssub:Laguerre and Meixner symmetric functions}. Then we explain the technique we use is this paper which helps to understand these results from a different point of view, and generalize them to other branching graphs. 

The $z$-measures first appeared from a representation-theoretic construction of generalized regular representations of $\Sym(\infty)$. Then a couple of algebraic/com\-bi\-na\-torial characterizations of these measures were suggested in \cite{rozhkovskaya1997multiplicative} and \cite{Borodin2000}. In the latter paper another graphs were also studied. For one of the graphs considered in \cite{Borodin2000} (namely, for the Schur graph, see \S \ref{sub:schur_graph}) a construction similar to \cite{rozhkovskaya1997multiplicative} allowed to define an analogue of the $z$-measures, see \cite{Borodin1997}.

The first approaches to (correlation functions of) the $z$-measures (in, e.g., \cite{Borodin1998}, \cite{Borodin2000a}) involved rather direct methods which included inverting certain infinite-dimensional matrices (the so-called ``$L$-$K$ correspondence'', see \cite{Borodin2000a} or \cite[Appendix~2]{Borodin2000b}).

The study of the $z$-measures on partitions and the dynamics $\labf^{zz'}(t)$ on $\Yb$ in \cite{borodin2006meixner}, \cite{Borodin2006} was done through analytic continuation. It appears that a certain degeneration of the $z$-measures leads to known and tractable $N$-particle point processes (namely, the Meixner orthogonal polynomial ensembles) and associated dynamics of noncolliding particles \cite{Konig2005}. Then an analytic continuation in the dimension $N$ is performed to obtain properties of the $z$-measures themselves. This analytic continuation approach is also employed in Olshanski's construction of the orthogonal basis of eigenfunctions of the generator of $\labf^{zz'}(t)$ in \cite{Olshanski2010LaguerreMeixner}, \cite{Olshanski2011Meixner}.

Asymptotic behavior of the up/down Markov chains for the $z$-measures in \cite{Borodin2007} was analyzed with the use of the algebra of symmetric functions suitably realized on the Young graph and on its boundary $\Omega$. A crucial role was played by an explicit formula for the action of the Markov transition operator on the Frobenius-Schur symmetric functions on Young diagrams \cite[Lemma 5.2]{Borodin2007}.

\subsection{$\slf(2)$ approach} 
Okounkov \cite{Okounkov2001a} presented another approach to the $z$-measures. He used the \emph{Kerov's operators} which span a certain $\slf(2,\C)$-module. Consider the Hilbert space $\ell^2(\Yb)$ of square integrable functions on $\Yb$ with the inner product $(f,g):=\sum_{\la\in\Yb}f(\la)\overline {g(\la)}$. A standard basis for this space is $\{\un\la\}_{\la\in\Yb}$, where $\un\la(\mu):=\delta_{\la,\mu}$. The Kerov's operators look as \cite[\S2.2]{Okounkov2001a}\footnote{These operators act in $\ell^2(\Yb)$. Note that they are unbounded.}
\begin{equation}\label{UDYoung_Ok}
  U\un\la=\sum_{\nu\colon\nu=\la+\square}
  (z+\j-\i)\un\nu,\qquad
  D\un\la=\sum_{\mu\colon\mu=\la-\square}
  (z'+\j-\i)\un\mu.
\end{equation}
The sums are over all Young diagrams which are obtained from $\la$ by adding or deleting a box $\square=(\i,\j)$, respectively. Here $\i$ and $\j$ are the row and column coordinates of $\square$. The crucial fact first observed by Kerov is that the commutator $H:=[D,U]$ acts diagonally: $H\un\la=(2|\la|+zz')\un\la$, where $|\la|$ is the number of boxes in $\la$. Thus, the operators $(U,D,H)$ satisfy the $\slf(2,\C)$ commutation relations:
\begin{equation}\label{sl2_relations}
  \left[ H,U \right]=2U,\qquad 
  \left[ H,D \right]=-2D,\qquad
  \left[ D,U \right]=H.
\end{equation}
The $z$-measures then have the following form \cite[\S2.2]{Okounkov2001a}:
\begin{equation}\label{z_meas_UD}
  M_n^{zz'}(\la)=\frac1{Z_n(z,z')}
  (U^n\uno\varnothing,\uno\la)
  (D^n\uno\la,\uno\varnothing),
\end{equation}
where $Z_n$ are normalizing constants and $(\cdot,\cdot)$ is the inner product in $\ell^2(\Yb)$.

The Kerov's operators are a powerful tool in dealing with the Young graph and the $z$-measures. In \cite{Okounkov2001a} they were used (together with a remarkable fermionic Fock space structure of partitions) to compute the correlation kernel of the $z$-measures. Later Olshanski \cite{Olshanski-fockone} used these operators and the Fock space formalism to obtain the dynamical determinantal kernel of the process $\labf^{zz'}(t)$ (this approach was suggested in \cite{Borodin2006}). A similar approach in the case of the Schur graph (together with its own fermionic Fock space picture) was carried out in \cite{Petrov2010}, \cite{Petrov2010Pfaffian}. In \cite{Borodin2007}, \cite{Petrov2007}, \cite{petrov2009eng} these or similar operators were used to explicitly compute the generators of the up/down Markov chains on various branching graphs, and obtain infinite-dimensional diffusions as their limits.
\begin{remark}\label{rmk:no_det}
  It must be noted that the fermionic Fock space methods for the Young and Schur graphs make much use of their fine structure. In the general picture described in the present paper there is no hope of obtaining determinantal or Pfaffian structure as in \cite{Okounkov2001a}, \cite{Olshanski-fockone}, \cite{Petrov2010}, \cite{Petrov2010Pfaffian}.
\end{remark}

\subsection{Results}
\label{sub:results} 
The notion of the Kerov's operators on the Young graph is central for the present paper, and we aim to generalize it to other branching graphs. We start with a branching graph $\Gb=\bigsqcup_{n=0}^\infty\Gb_n$ (each $\Gb_n$ is finite, $\Gb_0:=\{\varnothing\}$) whose set of vertices is a lattice of order ideals in some poset, and a triplet of operators $(U,D,H)$ in $\ell^2(\Gb)$ satisfying the $\slf(2,\C)$ commutation relations (\ref{sl2_relations}). For them we retain the name ``Kerov's operators''. See \S\ref{sec:ideal_branching_graphs} and \S\ref{sec:kerov_s_operators} for a detailed description of our model. We obtain the following results:

$\bullet$ Any triplet $(U,D,H)$ of Kerov's operators gives rise to a coherent system (see \S\ref{sub:coherent_systems} for the definition) of (possibly complex-valued) probability measures $\{M_n\}$ on the floors of the graph via (\ref{z_meas_UD}). This system is multiplicative in the sense of \cite{Borodin1997}, \cite{rozhkovskaya1997multiplicative}. Let us assume that each $M_n$ is a positive probability measure.

$\bullet$ We find all triplets of Kerov's operators for various important examples of branching graphs. For the Young graph we reconstruct operators (\ref{UDYoung_Ok}) corresponding to the $z$-measures. We also get all the degenerations of the $z$-measures in a unified way. Other examples include the Pascal triangle (with measures related to Bernoulli trials with Beta prior), the Kingman graph (with the Ewens-Pitman's random partitions corresponding to the two-parameter Poisson-Dirichlet distributions \cite{Pitman1997}), the Schur graph (with the measures introduced in \cite{Borodin1997}), and the general Young graph with Jack edge multiplicities (with the Jack deformation of the $z$-measures \cite{Borodin2005b}). We also briefly discuss the $\sl(2,\C)$ structure of Plancherel measures on rooted unlabeled trees recently studied by Fulman \cite{Fulman2009Trees}. Thus, we provide a unified approach to several known models. 

$\bullet$ The transition operators of up/down Markov chains on the floors $\Gb_n$ of the graph are expressed in terms of the Kerov's operators. Their eigenstructure is explicitly described. These results are parallel to the work of Fulman \cite{Fulman2007}, but in addition we write down an explicit expression for the action of the transition operators of up/down chains. For the Young graph this is precisely the formula of \cite[Lemma 5.2]{Borodin2007} which helps to study the asymptotic behavior of up/down Markov chains.

$\bullet$ Let $M_{\xi}$ be the mixture of the measures $\{M_n\}$ by means of the negative binomial distribution $\{(1-\xi)^{c}\frac{(c)_n}{n!}\xi^n\}_{n=0,1,\ldots}$ over the index $n$ (\S\ref{sub:mixing_of_measures}). We define Markov jump dynamics on the whole graph $\Gb$ preserving this measure $M_{\xi}$. The generator ${\mathsf{A}}_{\xi}$ of this dynamics is expressed in terms of the Kerov's operators.

$\bullet$ Our main result is the explicit construction of an orthogonal basis $\{\Mf_\la\}_{\la\in\Gb}$ in $\ell^2(\Gb,M_\xi)$ which provides a diagonalization of ${\mathsf{A}}_\xi$. For the Young graph this basis consists of the Meixner symmetric functions introduced in  \cite{Olshanski2010LaguerreMeixner}, \cite{Olshanski2011Meixner}. We give new proofs to some of the results of these papers. In particular, we manage to establish orthogonality of $\Mf_\la$ using a construction related to the Kerov's operators.

$\bullet$ The functions $\Mf_\la$ on the graph $\Gb$ arise as matrix elements of the $\sl(2,\C)$-module spanned by the Kerov's operators. This description is parallel to (and, in, fact, a generalization of) the known fact that the Meixner polynomials can be viewed as matrix elements of irreducible lowest weight $\slf(2,\C)$-modules \cite{Koornwinder1982Krawtchouk}, \cite{Vilenkin-Klimyk-DAN_UKR_1988}, \cite{Vilenkin-Klimyk-ITOGI1995-en}. 

$\bullet$ In the case of the Young graph with Jack edge multiplicities one gets the functions $\Mf_\la^{zz'\te\xi}$ which are discrete analogues of the Laguerre symmetric functions with the Jack  parameter constructed in \cite{Desrosiers2011Laguerre}. These functions $\Mf_\la^{zz'\te\xi}$ live on Young diagrams and cannot be viewed as symmetric functions. In a scaling limit they converge to the functions of \cite{Desrosiers2011Laguerre}.

In should be noted that on the abstract level of the present paper some of the results cited in \S\ref{ssub:Distinguished harmonic functions}, \S\ref{ssub:Stochastic dynamics associated with the z-measures}, and \S\ref{ssub:Laguerre and Meixner symmetric functions} are inaccessible, and these problems have to be figured out for each graph separately (see also Remark \ref{rmk:no_det}).

\subsection{Similar operators}
\label{sub:similar_operators}
Let us recall other work involving operators on posets and on branching or, more generally, graded graphs which are similar to Kerov's operators. First example of operators on posets satisfying $\slf(2,\C)$ commutation relations appeared in \cite{proctor1982solution} (see also the end of \cite{stanley1990variations}). In fact, these operators are precisely the ones corresponding to double degenerate $z$-measures which live on Young diagrams inside a rectangular shape (see \cite[\S4]{borodin2006meixner}). Stanley \cite{stanley1988differential} and Fomin \cite{fomin1979thesis}, \cite{fomin1994duality} independently considered differential posets for which the up and down operators $U^{\circ}$ and $D^{\circ}$ satisfy $[D^{\circ},U^{\circ}]=\mathbf{1}$ (throughout the paper $\mathbf{1}$ denotes the identity operator). For the Young graph these operators look as
\begin{equation}\label{U0D0Young}
  U^\circ\un\la=\sum_{\nu\colon\nu=\la+\square}\un\nu,\qquad
  D^\circ\un\la=\sum_{\mu\colon\mu=\la-\square}\un\mu.
\end{equation}
Fomin \cite{fomin1994duality} (see also \cite{stanley1990variations}) generalized these operators to other graded graphs (in fact, the paper \cite{fomin1994duality} deals with pairs of dual graded graphs, see \S\ref{sub:dual_multiplicity_functions}) by inserting graph's edge multiplicities as coefficients in (\ref{U0D0Young}). In \cite{fomin1994duality} other forms of commutation relations were also considered. The commutation relations in Stanley's and Fomin's papers were used to obtain enumeration results for various posets in a unified manner.

In \cite{Fulman2007} the commutation relations similar to the ones in \cite{fomin1994duality} were studied, but the up and down operators were probabilistic in nature: certain weights were inserted as coefficients in (\ref{U0D0Young}). Fulman used these operators to study spectral properties of down/up Markov chains on various branching graphs including Young, Schur, and Kingman graphs. The operators in \cite{Fulman2007} are very similar to the ones we consider in the present paper, and our analysis of the eigenstructure of up/down Markov chains could also be deduced from Fulman's results. However, as in \cite{Okounkov2001a}, we also make much use of a representation-theoretic intuition which comes from the $\slf(2,\C)$ structure.  

\subsection{Organization of paper}

In \S\ref{sec:ideal_branching_graphs} we describe our main abstract object in the present paper --- a branching graph whose set of vertices can be identified with the distributive lattice of finite ideals in a poset. 
In \S\ref{sec:kerov_s_operators} we discuss Kerov's operators on such branching graphs. It turns out that their existence implies that the edge multiplicities of the graph must be self-dual in a sense similar to \cite{fomin1994duality}. 
Using Kerov's operators, in \S\ref{sec:coherent_systems_of_measures} we define distinguished coherent systems of probability measures on floors of the graph as in (\ref{z_meas_UD}). 
In \S\ref{sec:relative_dimension_functions} we define relative dimension functions on the branching graph on which the Kerov's operators act in a particularly nice way. This is a necessary step towards the next two sections in which the main results of the paper are given. 
In \S\ref{sec:up_down_markov_chains} we recall the definition of the up/down Markov chains from \cite{Borodin2007} and use the Kerov's operators and the relative dimension functions to analyze the eigenstructure of these chains. 
In \S\ref{sec:markov_jump_dynamics} we turn to Markov jump dynamics on the whole graph and explicitly construct the orthogonal basis of eigenfunctions of the generator of the jump dynamics. 
In \S \ref{sec:heisenberg_operators} we discuss the degeneration of our picture when the $\slf(2,\C)$ structure is replaced by the (3-dimensional) Heisenberg algebra structure. This allows to define and diagonalize Markov processes on self-dual branching graphs (in the sense of \cite{fomin1994duality}).
In \S\ref{sec:examples} we briefly discuss how our general technique is applied to various branching graphs including the general Young graph with Jack edge multiplicities (whose particular cases are the Young and the Kingman graphs), the Schur graph, and several other models.

\smallskip
I am grateful to Grigori Olshanski for fruitful discussions and to Sergey Fomin for useful remarks.

% section introduction (end)

\section{Lattices of ideals as branching graphs} % (fold)
\label{sec:ideal_branching_graphs}

\subsection{Branching graphs} % (fold)
\label{sub:branching_graphs}

By a \emph{branching graph} we mean an abstract graph $\Gb$ under the following assumptions and conventions (e.g., see \cite{Borodin2000}):

$\bullet$ We identify the graph $\Gb$ with its set of vertices.

$\bullet$ The graph $\Gb$ is $\Z_{\ge0}$-graded, that is, $\Gb=\bigsqcup_{n=0}^{\infty}\Gb_n$, and the endpoints of any edge lie on consecutive levels. Typically, graphs will be infinite, but some examples involve finite graphs for which $\Gb_n$'s are empty sets for large enough $n$.
 
$\bullet$ All levels $\Gb_n$ are finite,\footnote{This excludes from our considerations the remarkable Gelfand-Tsetlin graph which describes the branching of irreducible representations of unitary groups, see Voiculescu \cite{Voiculescu1976}, Vershik and Kerov \cite{VK82CharactersU}, Boyer \cite{Boyer1983}, Okounkov and Olshanski \cite{OkOl1998}.} and the lowest level $\Gb_0$ consists of a single vertex denoted by $\varnothing$.
  
$\bullet$ We denote by $|\la|$ the level of the vertex $\la$, i.e., $\la\in\Gb_{|\la|}$.

$\bullet$ We assume that the edges of $\Gb$ are oriented from $\Gb_n$ to $\Gb_{n+1}$ for each $n$. If $(\mu,\la)$ is an edge and $|\la|=|\mu|+1$, then we write $\mu\nearrow\la$ or, equivalently, $\la\searrow\mu$.

$\bullet$ We assume that for any vertex $\mu$ there exists at least one vertex $\la\searrow\mu$, and for any vertex $\varkappa\ne\varnothing$ there exists at least one vertex $\nu\nearrow\varkappa$. This implies that $\Gb$ is connected.

$\bullet$ We are given an \emph{edge multiplicity function} which assigns to every edge $\mu\nearrow\la$  a {strictly positive} number $\km(\mu,\la)$. It should be emphasized that these numbers are not necessarily integers.
  
$\bullet$ Let $\mu=\la^{(1)}\nearrow\dots\nearrow\la^{(m)}=\la$ be an oriented path in $\Gb$. Its \emph{weight} is, by definition, equal to the product $\km(\la^{(1)},\la^{(2)})\km(\la^{(2)},\la^{(3)}) \dots\km(\la^{(m-1)},\la^{(m)})$ of edge multiplicities along the path. By $\dim(\mu,\la)$ denote the sum of weights of all distinct oriented paths in $\Gb$ from $\mu$ to $\la$ (all these paths have length $|\la|-|\mu|$). If $\dim(\mu,\la)\ne0$, we write $\mu\subseteq\la$. Also set $\dim\la:=\dim(\varnothing,\la)$; this number is always nonzero. The quantity $\dim\la$ can be called the \emph{(combinatorial) dimension} of the vertex $\la\in\Gb$. (For the Young graph $\dim\la$ actually is the dimension of the irreducible representation of the symmetric group $\Sym(|\la|)$ corresponding to the Young diagram $\la$.) The number $\dim(\mu,\la)$ is then the \emph{relative (combinatorial) dimension}.

% subsection branching_graphs (end)

\subsection{Young graph and Pascal triangle} % (fold)
\label{sub:pascal_triangle}

Our basic example of a branching graph is the celebrated Young graph $\Yb$ whose vertices are Young diagrams (we identify integer partitions and Young diagrams as in \cite[Ch. I, \S1]{Macdonald1995}). The property $\mu\nearrow\la$ for Young diagrams means that $\la$ is obtained from $\mu$ by adding one box. The edges in the Young graph are simple, i.e., the multiplicity function is always equal to $1$. However, as the running example used throughout the paper we choose a simpler graph, namely, the Pascal triangle. Other examples are considered in \S\ref{sec:examples}.

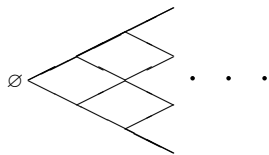
\begin{figure}[htpb]
  \begin{center}
    \setlength{\unitlength}{2000sp}%
    \begingroup\makeatletter\ifx\SetFigFont\undefined%
    \gdef\SetFigFont#1#2#3#4#5{%
      \reset@font\fontsize{#1}{#2pt}%
      \fontfamily{#3}\fontseries{#4}\fontshape{#5}%
      \selectfont}%
    \fi\endgroup%
    \begin{picture}(2280,1424)(-3314,-5073)
    \thinlines
    {\color[rgb]{0,0,0}\put(-2999,-4561){\line( 2, 1){600}}
    \put(-2399,-4261){\line( 2, 1){1200}}
    \put(-1199,-3661){\line(-2,-1){1800}}
    \put(-2999,-4561){\line( 2,-1){1800}}
    \put(-1199,-5461){\line(-2, 1){600}}
    \put(-1799,-5161){\line( 2, 1){600}}
    \put(-1199,-4861){\line(-2, 1){600}}
    \put(-1799,-4561){\line( 2, 1){600}}
    \put(-1199,-4261){\line(-2, 1){600}}
    \put(-1799,-3961){\line(-2,-1){600}}
    \put(-2399,-4261){\line( 2,-1){600}}
    \put(-1799,-4561){\line(-2,-1){600}}
    }%
    \put(-1049,-4561){\makebox(0,0)[lb]{\smash{{\SetFigFont{20}{24.0}{\rmdefault}{\mddefault}{\updefault}{\color[rgb]{0,0,0}. . .}%
    }}}}
    \put(-3269,-4640){\makebox(0,0)[lb]{\smash{{\SetFigFont{12}{14.4}{\rmdefault}{\mddefault}{\updefault}{\color[rgb]{0,0,0}$\scriptstyle\varnothing$}%
    }}}}
    \end{picture}%

    \caption{Pascal triangle}
    \label{fig:pascal}
  \end{center}  
\end{figure}

The classical Pascal triangle $\PT$ (Figure~\ref{fig:pascal}) is a branching graph whose $n$th level $\PT_n$ has the form $\{(k,l)\in\Z_{\ge0}^2\colon k+l=n\}$ and consists of $n+1$ vertices (on the zeroth level we identify $(0,0)\equiv\varnothing$). For any vertex $(k,l)$ there are exactly two vertices above it, namely, $(k+1,l)\searrow(k,l)$ and $(k,l+1)\searrow(k,l)$. The edges are simple. The dimension function for $\PT$ is clearly 
\begin{equation}\label{Pascal_dim}
  \dim(k,l)={k+l\choose k},\qquad (k,l)\in\PT.
\end{equation}

The description of the boundary of the Pascal triangle is equivalent to the classical de Finetti's theorem which classifies exchangeable random binary sequences. The boundary of $\PT$ is identified with the segment $[0,1]$ which in our terms means that coherent systems of probability measures on floors of $\PT$ (\S\ref{sub:coherent_systems}) are in one-to-one correspondence with Borel probability measures on $[0,1]$. See, e.g., \cite{Kerov1998} for more details about the notion of a branching graph's boundary.

% subsection pascal_triangle (end)

\subsection{Lattices of ideals} % (fold)
\label{sub:ideal_branching_graphs}

We restrict our attention to branching graphs which satisfy an additional property:
\begin{definition}\label{def:ideal_branching_graph}
  We say that a branching graph $\Gb$ is a \emph{graph of ideals}, if its set of vertices can be identified with the lattice $J(P)$ of \emph{finite} order ideals of some poset (partially ordered set) $(P,\preceq)$ which is called the \emph{underlying poset} of $\Gb$. 
\end{definition}

See Chapter 3 of Stanley's book \cite{Stanley1997} for a background on posets. Let us give necessary comments to Definition \ref{def:ideal_branching_graph} and introduce more notation:

$\bullet$ A subset $I\subseteq P$ is called an \emph{order ideal} if with any $x\in I$ it also contains all elements $y\in P$ such that $y\preceq x$.

$\bullet$ Finite order ideals form a distributive lattice with respect to the set-theoretic operations of union and intersection. The minimal ideal is the empty ideal $\varnothing$. It is identified with the initial vertex $\varnothing\in\Gb_{0}$.

$\bullet$ The level $|\la|$ of an ideal $\la\subset P$ is the number of elements in $\la$.  

$\bullet$ The edge $\mu\nearrow\la$ between two ideals means that $\mu$ is contained inside $\la$ and $|\la|=|\mu|+1$.\footnote{In fact, one sees that the general notation $\mu\subseteq\la$ introduced in \S\ref{sub:branching_graphs} for abstract branching graphs in the case of Definition \ref{def:ideal_branching_graph} means exactly that $\mu$ is contained inside $\la$.} By $\la/\mu$ denote the set difference of the ideals $\la$ and $\mu$; for $\mu\nearrow\la$ we identify the one-element set $\la/\mu$ with the corresponding element of $P$.

$\bullet$ Elements of the poset $P$ will be sometimes called \emph{boxes} and will be denoted by $\square$. This notation is chosen by analogy with the Young graph.

$\bullet$ The underlying poset is not uniquely defined. In the sequel for simplicity we will always assume that $P$ coincides with the union of all its finite order ideals (\emph{minimality}). In particular, this implies that $P$ is locally finite (i.e., any interval of the form $\{z\colon x\preceq z\preceq y\}$ is finite).

$\bullet$ A branching graph $\Gb$ corresponding to a poset $P$ and having an edge multiplicity function $\km$ will be denoted by $\Gb=(J(P),\km)$. Note that the lattice of ideals $J(P)$ itself can be equipped with a structure of a branching graph; in this case all edges will be simple. We call the corresponding multiplicity function $\km\equiv 1$ \emph{trivial}. Any other multiplicity function $\km\not\equiv1$ will be called \emph{nontrivial}.

\smallskip

Clearly, the Pascal triangle is represented as $\PT=(J(\Z_{>0}\sqcup\Z_{>0}),1)$, where $\Z_{>0}\sqcup\Z_{>0}$ is the disjoint union (e.g., see \cite[Ch.\;3]{Stanley1997}) of two copies of the standard half-lattice $(\Z_{>0},\le)$.  Any finite ideal in $\Z_{>0}\sqcup\Z_{>0}$ is characterized by a pair of nonnegative numbers $(k,l)$ and has the form $\{1,\dots,k\}\sqcup\{1,\dots,l\}$ (where, by agreement, $\{1,\dots,m\}=\varnothing$ for $m=0$). %Note that every finite ideal in $\Z_{>0}\sqcup\Z_{>0}$ is \emph{principal}, i.e., has the form $\{z\colon z\le x\}$ for a fixed $x$.

The Young graph $\Yb$ is also a graph of ideals (having a trivial edge multiplicity function). Its underlying poset $P=\Z_{>0}^2$ is the direct product of two copies of $(\Z_{>0},\le)$. For $(x_1,x_2), (y_1,y_2)\in P$ one has $(x_1,x_2)\le (y_1,y_2)$ iff $x_i\le y_i$, $i=1,2$.

% subsection ideal_branching_graphs (end)

% section ideal_branching_graphs (end)

\section{Kerov's operators} % (fold)
\label{sec:kerov_s_operators}

Here we give the general definition of Kerov's operators on a branching graph of ideals. Such operators for the Young graph were invented by S. Kerov, and appeared in \cite{Okounkov2001a}, \cite{Borodin2007}. Moreover, in \cite{Petrov2007} (the arXiv ``v1'' version), \cite{petrov2009eng}, and \cite{Petrov2010Pfaffian} similar operators for the Kingman and Schur graphs were considered, and in \cite{Petrov2010Pfaffian} (the arXiv version) an axiomatic approach to Kerov's operators on the Schur graph was suggested. The abstract definition of Kerov's operators generalizes all these situations. 

\subsection{Definition} % (fold)
\label{sub:definition_KO}

Consider the following basis in the Lie algebra $\slf(2,\C)$:
\begin{equation}\label{UDH_matrices}
  U:=\left[
  \begin{array}{cc}
    0&1\\
    0&0
  \end{array}
  \right],\qquad
  D:=\left[
  \begin{array}{cc}
    0&0\\
    -1&0
  \end{array}
  \right],\qquad
  H:=\left[
  \begin{array}{cc}
    1&0\\
    0&-1
  \end{array}
  \right].
\end{equation}
The commutation relations for this basis are
\begin{equation}\label{comm_rel}
  \left[ H,U \right]=2U,\qquad \left[ H,D \right]=-2D,\qquad
  \left[ D,U \right]=H.
\end{equation}

Assume that a graph of ideals $\Gb=(J(P),\km)$ is fixed. By $\ellf^2(\Gb)$ denote the (complex) pre-Hilbert space of all finitely supported functions on $\Gb$ with the standard inner product
\begin{equation*}
  (f,g):=\sum_{\la\in\Gb}f(\la)\overline{g(\la)}.
\end{equation*}
The Hilbert completion of $\ellf^2(\Gb)$ is the usual Hilbert space $\ell^2(\Gb)$. The standard orthonormal basis in $\ell^2(\Gb)$ is denoted by $\{\un\la \}_{\la\in\Gb}$, that is,
\begin{equation*}
  \un\la(\mu):=\left\{
  \begin{array}{ll}
    1,&\mbox{if $\mu=\la$};\\
    0,&\mbox{otherwise}.
  \end{array}
  \right.
\end{equation*}

\begin{definition}\label{def:Kerov's_rep}
  Let $\rep$ be a representation of the Lie algebra $\slf(2,\C)$ in $\ellf^2(\Gb)$. We call this representation a \textit{Kerov's representation} iff

  $\bullet$ Each vector $\un\la$, $\la\in\Gb$, is an eigenvector of $H$, and the eigenvalue of $\un\la$ depends only on the level $|\la|$.

  $\bullet$ The operators $\rep(U)$ and $\rep(D)$ act as
  \begin{align*}
    \rep(U)\un\la=
    \sum_{\nu\colon\nu\searrow\la}
    \km(\la,\nu)u(\nu/\la)\cdot\un\nu;\qquad
    \rep(D)\un\la=
    \sum_{\mu\colon\mu\nearrow\la}
    \km(\mu,\la)d(\la/\mu)\cdot\un\mu,
  \end{align*}
  for some functions $u$ and $d$ on the underlying poset $P$. We assume that if $u(\square)\ne0$, then $d(\square)\ne0$.
\end{definition}

The next fact is a straightforward corollary of the commutation relations (\ref{comm_rel}):
\begin{proposition}\label{prop:action_H}
  The operator $\rep(H)$ in any Kerov's representation acts as
  \begin{equation*}
    \rep(H)\un\la=\big(2|\la|+\zz\big)\un\la,
    \qquad\mbox{for all $\la\in\Gb$},
  \end{equation*}
  where $\zz:=(\rep(H)\un\varnothing,\un\varnothing)$ is the eigenvalue of $\un\varnothing$, which is some complex parameter depending on the representation.
\end{proposition}

In fact, the parameter $\zz$ can be expressed through the functions $u$ and $d$ on $P$, which follows from the relation $[D,U]=H$ (see also (\ref{q_condition_0})). Thus, we see that a Kerov's representation is completely defined by $u$ and $d$.

Observe that for each Kerov's representation $\rep$ with functions $u$ and $d$, the representation $\rep'$ with the functions 
\begin{equation*}
  u'(\square):=
  f(\square)u(\square),\qquad
  d'(\square):=
  \frac1{f(\square)}d(\square),
  \qquad\mbox{for all $\square\in P$},
\end{equation*}
where $f$ is an arbitrary nonvanishing function on $P$, is also a Kerov's representation. One can say that the representation $\rep'$ is obtained from $\rep$ by a \emph{gauge transformation}: we just multiply each basis vector $\un\la$ by $\prod_{\square\in\la}(f(\square))^{-1}$.

We will not distinguish gauge equivalent Kerov's representations, and among all equivalent representations we set apart one with $u\equiv d$. Denote $\ufunc(\square):=\sqrt{u(\square)d(\square)}$, and set
\begin{align}\label{Kerov_UfDf_def}
  \Uf\un\la:=
  \sum_{\nu\colon\nu\searrow\la}
  \km(\la,\nu)\ufunc(\nu/\la)\cdot\un\nu;
  \qquad
  \Df\un\la:=
  \sum_{\mu\colon\mu\nearrow\la}
  \km(\mu,\la)\ufunc(\la/\mu)\cdot\un\mu.
\end{align}
Also denote $\rep(H)$ by $\Hf$. The diagonal operator $\Hf$ does not change under a gauge transformaion. Note that the operators $\Uf$ and $\Df$ are adjoint to each other in $\ellf^2(\Gb)$.

\begin{definition}
  We call the triplet $(\Uf,\Df,\Hf)$ (i.e., with $\Uf$ and $\Df$ adjoint) the \textit{Kerov's operators}.
\end{definition} 

\begin{remark}\label{rmk:KO_root_q}
  Above we have set $\ufunc(\square)=\sqrt{u(\square)d(\square)}$. By agreement, this means that we choose some branch of the square root for each box $\square\in P$. In fact, the whole picture does not depend on this choice, and it is the function $\ufunc(\square)^2$ which essentially defines the Kerov's operators.
\end{remark}

\subsection{Lowest weight $\slf(2,\C)$-modules} % (fold)
\label{sub:lowest_weight_slf_2_c_modules}
Kerov's representations can be viewed as generalizations of classical irreducible lowest weight $\slf(2,\C)$-modules. 

Namely, let $P=\{1,2,\dots\}$ with the usual order, then $J(P)$ is identified with $\Z_{\ge0}$ (where $0\equiv\varnothing$). The standard basis in $\ellf^2(\Z_{\ge0})$ will be denoted by $\{\un n\}_{n=0,1,\dots}$. Due to the chain nature of $P$, it can be readily seen that any multiplicity function on $J(P)$ is equivalent (in the sense of \cite{Kerov1989}, see Definition \ref{def:equivalent_multilpicities}) to the trivial one. We will assume that $\Gb=(J(\Z_{>0}),1)$.

\begin{proposition}\label{prop:sl2_irrep}
  All triplets of Kerov's operators on $\Gb=(J(\Z_{>0}),1)$ are pa\-ra\-met\-rized by one complex parameter $\zz\in\C$ and have the form
  \begin{equation*}
    \Uf\un n=\sqrt{(n+1)(n+\zz)}\un{(n+1)},\quad
    \Df\un n=\sqrt{n(n-1+\zz)}\un{(n-1)},\quad
    \Hf\un n=(2n+\zz)\un n.
  \end{equation*}
  This corresponds to $\ufunc(n)=\sqrt{n(n-1+\zz)}$ ($n=1,2,\dots$).
\end{proposition}
\begin{proof}
  This is a classical computation involved in the description of $\slf(2,\C)$-modules, e.g., see \cite[Part I]{Pukanszky_SL2_1964}.
\end{proof}
All $\slf(2,\C)$-modules obtained in this proposition have the lowest weight vector $\un0=\un\varnothing$. For $\zz=0,-1,-2,\dots$ one recovers the $(-\zz+1)$-dimensional irreducible representation.

Thus, one sees that the general Kerov's representations on branching graphs arise as natural generalizations of this picture when the underlying poset is not a chain. Of course, for a general graph there is no hope of getting rid of an edge multiplicity function.

% subsection lowest_weight_slf_2_c_modules (end)

\subsection{Kerov's operators for Pascal triangle} % (fold)
\label{sub:kerov_s_operators_for_pascal_triangle}

Now let us describe all triplets of Kerov's operators for the Pascal triangle $\PT$ (\S\ref{sub:pascal_triangle}). The underlying poset here $P=\Z_{>0}\sqcup\Z_{>0}$ is a disjoint union of the posets considered in the previous subsection. It turns out that the Kerov's operators behave nicely under such an operation on underlying posets. 

Let us first consider a general situation. Let $\Gb^i=(J(P_i),\km_i)$, $i=1,2$, be two branching graphs of ideals, and let $\Gb:=(J(P_1\sqcup P_2),\km)$. Here the multiplicity function $\km$ is constructed from $\km_1$ and $\km_2$ in an obvious way. Functions $\ufunc$ on $P_1\sqcup P_2$ are best understood as pairs $\ufunc=(\ufunc_1,\ufunc_2)$, where $\ufunc_i$ is a function on $P_i$.

\begin{proposition}\label{prop:KO_disjoint}
  A function $\ufunc=(\ufunc_1,\ufunc_2)$ on $P=P_1\sqcup P_2$ defines a triplet of Kerov's operators on $\Gb$ iff for $i=1,2$ the function $\ufunc_i$ defines Kerov's operators on $\Gb^i$. Moreover, the eigenvalues of the initial vectors are related as
  \begin{equation*}
    (\Hf\un\varnothing,\un\varnothing)
    =
    (\Hf\un\varnothing^{1},\un\varnothing^{1})+
    (\Hf\un\varnothing^{2},\un\varnothing^{2}),\qquad
    \varnothing\in\Gb_0,\qquad\varnothing^{i}\in\Gb_0^{i},\quad
    i=1,2.
  \end{equation*}
\end{proposition}
\begin{proof}
  Observe that $J(P)=J(P_1)\times J(P_2)$. To establish the claim it suffices to see that the identity (\ref{q_condition}) below holds for any $\la=(\la^{1},\la^{2})$ if and only if it holds for special $\la$'s in $\Gb$ of the form $\la=(\la^{1},\varnothing^{2})$ and $\la=(\varnothing^{1},\la^{2})$.
\end{proof}

\begin{corollary}
  All triplets of Kerov's operators on $\PT$ are parametrized by two complex parameters $\zz_1,\zz_2\in\C$. The corresponding functions $\ufunc=(\ufunc_1,\ufunc_2)$ on $\Z_{>0}\sqcup\Z_{>0}$ look as $\ufunc_i(n)=\sqrt{n(n-1+\zz_i)}$, $i=1,2$, $n=1,2,\ldots,$ and $\zz=(\Hf\un\varnothing,\un\varnothing)=\zz_1+\zz_2$.
\end{corollary}

Everything that is said in the present paper about the Pascal triangle $\PT$ readily carries out to its $d$-dimensional analogue $\PT^{d}=(J(\Z_{>0}\sqcup \ldots \sqcup \Z_{>0}),1)$ ($d$ times, where $d=2,3,\ldots$). We classify Kerov's operators for other branching graphs of ideals in \S\ref{sec:examples}.

% subsection kerov_s_operators_for_pascal_triangle (end)

\subsection{Kerov's operators and UD-self-duality} % (fold)
\label{sub:kerov_s_operators_and_self_dual_edge_multiplicities}

In this and the next subsection we work out special features arising for graphs of ideals with nontrivial edge multiplicity functions.

Assume that we have a graph $\Gb=(J(P),\km)$ and a triplet of Kerov's operators $(\Uf,\Df,\Hf)$ on it. It turns out that the commutation relation $[\Df,\Uf]=\Hf$ is very restrictive and imposes a certain condition on the multiplicity function on $\Gb$. Namely, fix $\la\in \Gb$ and consider
\begin{equation}\label{[D,U]}
  \begin{array}{l}
    \left[ \Df,\Uf \right]\un\la=\displaystyle
    \sum_{\nu\searrow\la}\sum_{\rho\nearrow\nu}
    \km(\la,\nu)\km(\rho,\nu)
    \ufunc(\nu/\la)\ufunc(\nu/\rho)\un\rho
    \\
    \qquad\qquad
    \qquad \qquad\displaystyle-
    \sum_{\mu\nearrow\la}
    \sum_{\rho\searrow\mu}
    \km(\mu,\la)
    \km(\mu,\rho)
    \ufunc(\la/\mu)
    \ufunc(\rho/\mu)
    \un\rho.
  \end{array}
\end{equation}
This is a linear combination of some vectors $\un\rho$, where $|\rho|=|\la|$, which should be equal to $(2|\la|+\zz)\un\la$. There are two cases: either $\rho\ne\la$, or $\rho=\la$.
  
\textbf{(1)} In the first case there exist $\square_1, \square_2 \in P$ such that $\rho=(\la\cup\square_1)\setminus\square_2$. In (\ref{[D,U]}) this corresponds to only one $\nu=\la\cup\square_1=\rho\cup\square_2$ and only one $\mu=\la\setminus\square_2=\rho\setminus\square_1$. Thus, the coefficient by $\un\rho$ in $\left[ \Df,\Uf \right]\un\la$ is equal to
\begin{equation*}
  \km(\la,\nu)\km(\rho,\nu)
  \ufunc(\square_1)\ufunc(\square_2)-
  \km(\mu,\la)\km(\mu,\rho)\ufunc(\square_2)\ufunc(\square_1). 
\end{equation*}
On the other hand, since $\rho\ne\la$, this coefficient must be zero. Thus, we see that for the existence of Kerov's operators on the graph $\Gb$ it is necessary that for any \emph{quadrangle} $\mu\rho\nu\la$, that is, for any four vertices $\mu$, $\rho$, $\nu$, $\la$ of $\Gb$ related as
\begin{equation}\label{quadrangle}
  \begin{array}{ccccc}
    &&\rho&&\\
    &\nearrow&&\searrow\\
    \mu&&&&\nu\\
    &\searrow&&\nearrow\\
    &&\la
  \end{array}
\end{equation}
(where $\rho\ne\la$ and arrows mean the $\nearrow$ relation for vertices in $\Gb$), we have
\begin{equation}\label{UD-self-duality}
  \km(\mu,\la)\km(\mu,\rho)=\km(\la,\nu)\km(\rho,\nu).
\end{equation}

\begin{definition}\label{def:self-dual}
  A multiplicity function $\km$ on $\Gb$ satisfying property (\ref{UD-self-duality}) for any quadrangle $\mu\rho\nu\la$ is called \textit{UD-self-dual} (in contrast to more special duality properties considered in \cite{fomin1994duality}, see \S\ref{sub:dual_multiplicity_functions} below).
\end{definition}

Of course, the trivial multiplicity function is UD-self-dual.

\textbf{(2)} Considering $\rho=\la$ in (\ref{[D,U]}) we see that the function $\ufunc(\cdot)$ on $P$ must satisfy
\begin{align}\label{q_condition}
  \sum_{\nu\colon\nu\searrow\la}
  \km(\la,\nu)^2\ufunc(\nu/\la)^2-
  \sum_{\mu\colon\mu\nearrow\la}
  \km(\mu,\la)^2
  \ufunc(\la/\mu)^2=2|\la|+\zz\qquad\forall\la\in\Gb.
\end{align}
In particular, for $\la=\varnothing$, one has
\begin{equation}\label{q_condition_0}
  \zz=
  \sum_{\nu\searrow\varnothing}
  \km(\varnothing,\nu)^2\ufunc(\nu/\varnothing)^2,
\end{equation}
i.e., the parameter $\zz=(\Hf\un\varnothing,\un\varnothing)\in\C$ is determined by $\ufunc$.

Thus, we have proved the following:
\begin{proposition}
  \begin{enumerate}[\bf1.]
    \item 
    Kerov's operators can exist only on a graph $\Gb=(J(P),\km)$ with an UD-self-dual multiplicity function $\km$. 
    \item
    Triplets of Kerov's operators $(\Uf,\Df,\Hf)$ are in one-to-one correspondence with functions $\ufunc(\cdot)^2$ (see Remark \ref{rmk:KO_root_q}) on $P$ satisfying (\ref{q_condition}), where $\zz$ is given by (\ref{q_condition_0}).
  \end{enumerate}
\end{proposition}

This proposition gives a practical tool for finding all Kerov's operators by solving (\ref{q_condition})--(\ref{q_condition_0}). For many known examples this leads to a new characterization of distinguished coherent systems of measures, most notably, the $z$-measures and the two-parameter Ewens-Pitman's partition structures (\S\ref{sec:examples}). The general construction of measures is explained in \S\ref{sub:coherent_systems} below.

% subsection kerov_s_operators_and_self_dual_edge_multiplicities (end)

\subsection{UD-dual multiplicity functions} % (fold)
\label{sub:dual_multiplicity_functions}

Most examples of branching graphs with nontrivial edge multiplicities (\S\ref{sec:examples}) come with multiplicity functions which are not UD-self-dual. Before finding Kerov's operators on these graphs one must find an ``equivalent'' UD-self-dual multiplicity function to work with. Here the notion related to the duality of graded graphs \cite{fomin1994duality} comes to play.

Assume that the underlying poset $P$ is fixed, and we are dealing with various multiplicity functions on $J(P)$.\footnote{Recall that the lattice of ideals $J(P)$ has edges defined regardless of a multiplicity function. Any multiplicity function on $J(P)$ must be positive on every edge.}

\begin{definition}[{\cite{Kerov1989}}]
\label{def:equivalent_multilpicities}
  Two multiplicity functions $\km$ and $\km'$ on the edges of $J(P)$ are called \emph{gauge equivalent} iff they are related as
  \begin{equation*}
    \km(\mu,\la)=\frac{g(\mu)}{g(\la)}\km'(\mu,\la)
    \qquad\mbox{for any edge $\mu\nearrow\la$},
  \end{equation*}
  where $g$ is some nonvanishing function on $J(P)$.
\end{definition}
For gauge equivalent $\km$ and $\km'$ the branching graphs $(J(P),\km)$ and $(J(P),\km')$ have the same down transition function and thus the same supply of coherent systems of measures (see \S\ref{sub:coherent_systems}), so for us the difference between them is inessential.

% To construct an UD-self-dual multiplicity function $\tilde\km$ which is gauge equivalent to some existing multiplicity function $\km$ on $J(P)$ one may first find another function $\km'$ which is UD-dual to $\km$, see Corollary \ref{cor:UD-self-dual} below.
\begin{definition}\label{def:UD-dual}
  Two multiplicity functions $\km$ and $\km'$ on $J(P)$ are called \textit{UD-dual} iff
  $\km(\la,\nu)\km'(\rho,\nu)=\km'(\mu,\la)\km(\mu,\rho)$ for any quadrangle $\mu\rho\nu\la$ (\ref{quadrangle}).
\end{definition}

\begin{proposition}\label{prop:duality_and_gauge}
  UD-dual multiplicity functions are gauge equivalent.
\end{proposition}
\begin{proof}
  As $g(\la)$ (for any $\la\in\Gb$) in Definition \ref{def:equivalent_multilpicities} take the product $\prod_i\frac{\km(\la^{(i)},\la^{(i+1)})} {\km'(\la^{(i)},\la^{(i+1)})}$ along a path from $\varnothing$ to $\la$. It is readily seen from the duality property that this product does not depend on the path (because one path can be transformed into another by a sequence of flips $\mu\rho\nu\mapsto\mu\la\nu$ along quadrangles (\ref{quadrangle})). Since $\km'(\mu,\la)=\frac{g(\mu)}{g(\la)}\km(\mu,\la)$ for any edge $\mu\nearrow\la$, this concludes the proof.
\end{proof}

A pair of UD-dual multiplicity functions $\km$ and $\km'$ gives rise to an UD-self-dual multiplicity function $\tilde \km(\mu,\la)=\sqrt{\km(\mu,\la)\km'(\mu,\la)}$. Thus, to study Kerov's operators on a graph $\Gb=(J(P),\km)$ with a multiplicity function which is not UD-self-dual, one could first find a function $\km'$ which is UD-dual to $\km$, and then consider the UD-self-dual $\tilde\km=\sqrt{\km\km'}$.

% We note one more property of UD-self-dual multiplicity functions, and then discuss how Definition \ref{def:UD-dual} is related to the duality studied in \cite{fomin1994duality}.

\begin{remark}\label{rmk:Plancherel_operators}
  The notion of UD-(self)-duality generalizes the duality of \cite{fomin1994duality}. Indeed, let $\km$ be UD-self-dual, and consider
  \begin{equation*}
    \Uf^\circ\un\la:=\sum_{\nu\colon\nu\searrow\la}
    \km(\la,\nu)\un\nu,\qquad
    \Df^\circ\un\la:=\sum_{\mu\colon\mu\nearrow\la}
    \km(\mu,\la)
    \un\mu.
  \end{equation*}
  Then by Definition \ref{def:self-dual} the commutator $[\Df^\circ,\Uf^\circ]$ acts diagonally: $[\Df^\circ,\Uf^\circ]\un\la=h(\la)\un\la$ for some function $h(\cdot)$ on $\Gb$. The $\mathbf{r}$-duality \cite[(1.3.4)]{fomin1994duality} (and equivalent sequential differentiality of \cite{stanley1990variations}) is more special as it requires the function $h(\la)$ to depend on $\la$ only through its level $|\la|$ (cf. the branching of plane partitions, see \S \ref{sub:plane_partitions}).

  Note that our notion of UD-(self)-duality does not require branching graphs to be ideal, i.e., this notion falls into the more general framework of \cite{fomin1994duality}.
\end{remark}

% subsection dual_multiplicity_functions (end)

% section kerov_s_operators (end)

\section{Coherent systems of measures} % (fold)
\label{sec:coherent_systems_of_measures}

Here we explain in detail how Kerov's operators on a graph of ideals give rise to distinguished probability measures on floors of the graph. 

\subsection{Coherent systems} % (fold)
\label{sub:coherent_systems}

Let us recall the general definition of coherent systems (e.g., see \cite{Borodin2007}). Let $\Gb$ be a branching graph (not necessary a lattice of ideals) with the multiplicity function $\km$. Recall that the multiplicity function gives rise to the dimension function $\dim\la$, $\la\in\Gb$ (see the end of \S\ref{sub:branching_graphs}). 

\begin{definition}
  \emph{Down transition probabilities} on the graph $\Gb$ are defined as
  \begin{equation*}
    p^{\da}_{n,n-1}(\la,\mu):=\left\{
    \begin{array}{ll}
      \frac{\km(\mu,\la)\dim\mu}{\dim\la},
      &\mbox{if $\mu\nearrow\la$},\\
      0,&\mbox{otherwise},
    \end{array}
    \right.
  \end{equation*}
  where $\la\in\Gb_{n}$ and $\mu\in\Gb_{n-1}$. For every $n\ge1$, $p^\da_{n,n-1}$ is a \emph{Markov transition kernel} from $\Gb_n$ to $\Gb_{n-1}$ in the sense that $p^\da_{n,n-1}(\la,\mu)\ge0$ for any $\la\in\Gb_n,\mu\in\Gb_{n-1}$, and $\sum_{\mu\in\Gb_{n-1}}p^\da_{n,n-1}(\la,\mu)=1$ for any $\la\in\Gb_n$.
\end{definition}

\begin{definition}
  A family $\{M_n\}$, where each $M_n$ is a probability measure on $\Gb_n$, is called \emph{coherent}, if these measures are compatible with the Markov kernels $p^\da_{n+1,n}$, that is, $M_{n+1}\circ p^\da_{n+1,n}=M_n$. In more detail, this means that
  \begin{equation}\label{coherency}
    \sum_{\nu\colon\nu\searrow\la}
    M_{n+1}(\nu)
    p^{\da}(\nu,\la)=M_n(\la)
    \qquad\mbox{for all $n$ and $\la\in\Gb_n$}.
  \end{equation}
\end{definition}
Observe that the condition (\ref{coherency}) is purely algebraic. Though we will typically consider only positive probability measures, we can also speak about \emph{complex-valued coherent systems} $\{M_n\}$ which satisfy (\ref{coherency}) and for which $\sum_{\la\in\Gb_n}M_n(\la)=1$ (i.e, we drop the condition $M_n(\la)\ge0$ for all $\la\in\Gb_n$).

We say that a family of measures $\{M_n\}$ on $\Gb_n$ is \emph{nondegenerate} if $M_n(\la)\ne0$ for all $\la\in\Gb_n$ and all $n$. We will only deal with nondegenerate case, see Remark \ref{rmk:nondegenerate}.1 below.

Later we will need one more object related to coherent systems of measures:
\begin{definition}
  Let $\{M_n\}$ be a nondegenerate coherent system on $\Gb$. \emph{Up transition probabilities} for $\{M_n\}$ are defined as
  \begin{equation*}
    p^\ua_{n,n+1}(\la,\nu):=\frac{M_{n+1}(\nu)}{M_n(\la)}p^\da_{n+1,n}(\nu,\la),
    \qquad
    \la\in\Gb_n,\ \nu\in\Gb_{n+1},\ \nu\searrow\la.
  \end{equation*}
  From (\ref{coherency}) we see that for every $n$, $p^\ua_{n,n+1}$ is a Markov transition kernel from $\Gb_n$ to $\Gb_{n+1}$. The measures $M_n$ are compatible with these up kernels: $M_n\circ p^\ua_{n,n+1}=M_{n+1}$.
\end{definition}

% subsection coherent_systems (end)

\subsection{Distinguished coherent systems obtained from Kerov's operators} % (fold)
\label{sub:distinguished_coherent_systems_obtained_from_kerov_s_operators}

Let us fix a triplet $(\Uf,\Df,\Hf)$ of Kerov's operators on a graph of ideals $\Gb=(J(P),\km)$ with an UD-self-dual multiplicity function $\km$. We use (\ref{z_meas_UD}) as a prompt to define:
\begin{equation}\label{Kerov_coherent}
  M_n(\la):=\frac{1}{Z_n}(\Uf^n\uno\varnothing,\uno\la)
  (\Df^n\uno\la,\uno\varnothing),\qquad
  \mbox{for all $n=0,1,\dots$ and all $\la\in\Gb_n$}.
\end{equation}
Throughout the paper, $M_n(\la)$ denotes the measure of the singleton $\{\la\}$. In (\ref{Kerov_coherent}), $Z_n$ is the normalizing constant such that $\sum_{\la\in\Gb_n}(\Uf^n\uno\varnothing,\uno\la)(\Df^n\uno\la,\uno\varnothing)=1$, and $(\cdot,\cdot)$ is the inner product in $\ell^2(\Gb)$ (\S\ref{sub:definition_KO}). Let us look closer at the constants $Z_n$, $n=0,1,\ldots$:

\begin{lemma}[{cf. \cite[\S4.1]{Petrov2010}}]\label{lemma:Z_n}
  We have
  \begin{equation*}
    \sum_{\la\in\Gb_n}
    (\Uf^n\uno\varnothing,\uno\la)(\Df^n\uno\la,\uno\varnothing)
    =n!(\zz)_n,
  \end{equation*}
  where $(a_k):=a(a+1)\ldots(a+k-1)$ is the Pochhammer symbol and $\zz=(\Hf\un\varnothing,\un\varnothing)$.
\end{lemma}
\begin{proof}
  From the commutation relations (\ref{comm_rel}) one gets
  \begin{equation}\label{DU^n}
    \Df\Uf^n=\Uf^n\Df+\sum_{k=0}^{n-1}\Uf^{n-k-1}\Hf\Uf^k.
  \end{equation}
  We obtain
  \begin{align*}
    Z_n&=\sum_{\la\in\Gb_n}
    (\Uf^n\uno\varnothing,\uno\la)(\Df^n\uno\la,\uno\varnothing)
    =
    \Big( \Df^n\sum_{\la\in\Sb_n}
    (\Uf^n\uno\varnothing,\uno\la)\cdot\uno\la,\uno\varnothing\Big)
    =
    (\Df^n\Uf^n\un\varnothing,\un\varnothing)
    \\&=
    \sum_{k=0}^{n-1}(\Df^{n-1}\Uf^{n-k-1}\Hf\Uf^k\un\varnothing,\un\varnothing)
    =Z_{n-1}
    \sum_{k=0}^{n-1}\left(2k+\zz\right)=
    n\left( n-1+\zz \right)Z_{n-1}
  \end{align*}
  (we have used the fact that $\Df\un\varnothing=0$). Taking into account the initial value $Z_0=(\Uf^0\Df^0\un\varnothing,\un\varnothing)=1$, we see that $Z_n=n!(\zz)_n$.
\end{proof}
This means that for any $\zz\in\C$, $\zz\ne 0,-1,-2,\dots$, the normalizing constants $Z_n$ do not vanish for all $n$. Thus, for such $\zz$ we can consider the (possibly complex-valued) measures $M_n$ for all $n\in\Z_{\ge0}$. We have
\begin{equation}\label{mult_m}
  M_n(\la)=\frac{(\dim\la)^2}{n!(\zz)_n}\prod_{\square\in\la}
  \ufunc(\square)^2,\qquad \la\in\Gb_n.
\end{equation}

\begin{theorem}[{cf. \cite[\S4.1]{Petrov2010}}]
\label{thm:Kerov_coherency}
  Let us assume that $\zz\in\C\setminus\{0,-1,-2,\ldots\}$. The measures $\left\{ M_n \right\}_{n=0}^{\infty}$ given by (\ref{Kerov_coherent}) (or (\ref{mult_m})) form a coherent system on the graph $\Gb=(J(P),\km)$.
\end{theorem}
\begin{proof}
  We must establish (\ref{coherency}). Let us use another Kerov's representation (Definition \ref{def:Kerov's_rep}) where the up and down operators look as
  \begin{equation*}
    \hat \Uf\uno\la:=
    \sum_{\nu\colon\nu\searrow\la}
    \km(\la,\nu)\ufunc(\nu/\la)^2\uno\nu,
    \qquad
    \hat \Df\uno\la:=
    \sum_{\mu\colon\mu\nearrow\la}
    \km(\mu,\la)\uno\mu.
  \end{equation*}
  Clearly, from (\ref{q_condition}) we have $[\hat\Df,\hat\Uf]=\Hf$. Also note that
  \begin{equation*}
    M_n(\la)=\frac1{Z_n}
    (\hat\Uf^n\uno\varnothing,\uno\la)
    (\hat\Df^n\uno\la,\uno\varnothing)=
    \frac{\dim\la}{Z_n}
    (\hat\Uf^n\uno\varnothing,\uno\la).
  \end{equation*}
  Therefore, we can write
  \begin{align*}
    \sum_{\nu\colon\nu\searrow\la}
    M_{n+1}(\nu)p^\da_{n+1,n}(\nu,\la)&=
    \frac{\dim\la}{Z_{n+1}}
    \sum_{\nu\colon\nu\searrow\la}
    \km(\la,\nu)(\hat\Uf^{n+1}\un\varnothing,\un\nu)\\&=
    \frac{\dim\la}{Z_{n+1}}
    (\hat\Uf^{n+1}\un\varnothing,\hat\Df^*\un\la)=
    \frac{\dim\la}{Z_{n+1}}
    (\hat\Df\hat\Uf^{n+1}\un\varnothing,\un\la).
  \end{align*}
  Using identity (\ref{DU^n}) (which of course also holds for $(\hat\Uf,\hat\Df,\Hf)$), we compute
  \begin{align*}
    \frac{\dim\la}{Z_{n+1}}
    (\hat\Df\hat\Uf^{n+1}\un\varnothing,\un\la)
    =\frac{\dim\la}{Z_{n+1}}
    \sum_{k=0}^{n-1}
    (\hat\Uf^{n-k-1}\Hf\hat\Uf^{k}\un\varnothing,\un\la)
    =\frac{\dim\la}{Z_n}(\hat\Uf^{n}\un\varnothing,\un\la)=
    M_n(\la).
  \end{align*}
  This concludes the proof.
\end{proof}

We see that the measures $M_n$ are nonnegative (i.e., they are true probability measures) precisely when one can take all complex numbers $\{\ufunc(\square)^2\}_{\square\in P}$ to be nonnegative. In fact, this implies (see (\ref{q_condition_0})) that $\zz\ge0$. 

\begin{remarks}\label{rmk:nondegenerate}
  \textbf{1.}
  Assume that the coherent system $\{M_n\}$ is degenerate. This is equivalent to having $\ufunc(v)=0$ for some $v\in P$. Set $I_v:=\{u\in P\colon u\succeq v\}$. Clearly, $M_n(\nu)=0$ for any $\nu\in\Gb$ for which $\nu\cap I_v\ne\varnothing$. This means that $\{M_n\}$ lives on the lattice of ideals of the subposet $(P\setminus I_v)\subset P$ (cf. how finite-dimensional modules arise in \S\ref{sub:lowest_weight_slf_2_c_modules}). This observation allows to deal only with nondegenerate coherent systems, and in the sequel we assume that $\{M_n\}$ (\ref{Kerov_coherent}) is a nondegenerate coherent system of (real-valued) probability measures on floors of the graph $\Gb$.

  \textbf{2.}
  One sees that the measures $M_n(\la)$ given by (\ref{Kerov_coherent})--(\ref{mult_m}) are multiplicative in the sense of \cite{Borodin1997}, \cite{rozhkovskaya1997multiplicative}. Those papers gave classifications of multiplicative coherent systems on Schur and Young graphs, respectively. In fact (see \S\S \ref{sub:measures_with_jack_parameter}, \ref{sub:schur_graph}), the supply of coherent systems obtained through Kerov's operators for these two graphs is the same as described in \cite{Borodin1997}, \cite{rozhkovskaya1997multiplicative}. However, our approach with Kerov's operators is more unified, is applicable to more examples of branching graphs and seems simpler. It also provides other constructions and results.

  \textbf{3.}
  It is also possible to consider operators spanning an $\slf(2,\C)$-module for branching graphs whose set of vertices is not a lattice of ideals. Such an example is considered in \cite{Fulman2009Trees}, we discuss it in \S \ref{sub:rooted_trees}. In that case many of the results of the present paper also hold. However, we stick to the formalism of graphs of ideals because it allows to give a unified characterization of several interesting families of measures on various graphs.
\end{remarks}

% subsection distinguished_coherent_systems_obtained_from_kerov_s_operators (end)

\subsection{Mixing of measures} % (fold)
\label{sub:mixing_of_measures}

Let $\xi\in(0,1)$ and $c>0$ be parameters. The \emph{negative binomial distribution} on $\Z_{\ge0}$ is the probability measure defined as
\begin{equation}\label{NegBinom}
  \pi_{c,\xi}(n):=(1-\xi)^{c}\frac{(c)_n}{n!}\xi^n,\qquad
  n=0,1,\ldots.
\end{equation}
% For certain negative values of $\zz$, namely, for $(-\zz)\in\Z_{\ge0}$, the measure $\pi_{c,\xi}$ should be replaced by the ordinary binomial distribution. In more detail about this situation see \S4 in \cite{borodin2006meixner} where the second degenerate series of the $z$-measures is considered.

\begin{definition}
    Assume that the Kerov's operators $(\Uf,\Df,\Hf)$ on our graph $\Gb=(J(P),\km)$ are such that the parameter $\zz=(\Hf\un\varnothing,\un\varnothing)$ is positive. Then one can define the \emph{mixing} of the measures $\{M_n\}_{n=0}^{\infty}$ (\ref{Kerov_coherent}), (\ref{mult_m}) by means of the negative binomial distribution $\pi_{\zz,\xi}$ (where $\xi\in(0,1)$ is the new parameter) as
  \begin{equation*}
    M_\xi(\la):=\pi_{\zz,\xi}(|\la|)\cdot M_{|\la|}(\la)=
    (1-\xi)^{\zz}\xi^{|\la|}
    \left(\frac{\dim\la}{|\la|!}\right)^2
    \prod_{\square\in\la}
    \ufunc(\square)^2,
  \end{equation*}
  where $\la$ runs over all graph $\Gb$. The measure $M_n$ is reconstructed from $M_\xi$ by conditioning on $\{|\la|=n\}$.
\end{definition}

The idea of mixing of the measures $M_n$ comes from the case of the $z$-measures \cite{Borodin2000a}. A similar construction is the poissonization of the Plancherel measures on partitions \cite{baik1999distribution}, \cite{Borodin2000b} (see also \S\ref{sec:heisenberg_operators}). The mixing of measures also can be viewed as a passage to the grand canonical ensemble \cite{Vershik1996StatMech}. The recent paper \cite{BorodinOlsh2011Bouquet} introduces an object related to the Young graph --- the Young bouquet. In that context the appearance of mixed measures looks very natural.

The mixed measures $M_\xi$ are preserved by the Markov jump process on $\Gb$ we construct below in \S\ref{sec:markov_jump_dynamics}.

% subsection mixing_of_measures (end)

\subsection{Application to Pascal triangle} % (fold)
\label{sub:application_to_pascal_triangle}

Let us use our description of Kerov's operators for the Pascal triangle (\S\ref{sub:kerov_s_operators_for_pascal_triangle}) to see what measures arise for $\PT$. From (\ref{Pascal_dim}) and (\ref{Kerov_coherent}) one has the following (complex-valued) measures coming from Kerov's operators:
\begin{equation}\label{Kerov_coherent_Pascal}
  M_n^{\zz_1\zz_2}(k,l)=\frac{n!}{k!l!}
  \frac{(\zz_1)_k(\zz_2)_l}{(\zz_1+\zz_2)_{n}},
\end{equation}
where $(k,l)\in\PT_n$ (i.e., $k+l=n$). Assume that the coherent system $\{M_n^{\zz_1\zz_2}\}$ is nondegenerate, that is, $\zz_1,\zz_2>0$. Let us explain the probabilistic nature of the measures $M_n^{\zz_1\zz_2}$. Consider the sequence of Bernoulli trials with the probability of success $p\in[0,1]$ having the $\mathrm{Beta}(\zz_1,\zz_2)$ distribution. That is, $p$ is a random variable with density $\frac{\Ga(\zz_1+\zz_2)}{\Ga(\zz_1)\Ga(\zz_2)}p^{\zz_1-1}(1-p)^{\zz_2-1}dp$. In other words, one first samples $p\in[0,1]$ from $\mathrm{Beta}(\zz_1,\zz_2)$, and then samples a binary sequence with probability of success $p$. Then one readily checks that
\begin{align*}&
  M_n^{\zz_1\zz_2}(k,l)=
  \Prob\{\text{in a sequence of length $n$ 
  there are $k$ 1's and $l=n-k$ 0's}\}.
\end{align*}
Thus, the measures $M_n^{\zz_1\zz_2}$ describe Bernoulli trials with $\mathrm{Beta}(\zz_1,\zz_2)$ prior. The $\mathrm{Beta}(\zz_1,\zz_2)$ distribution on the boundary $[0,1]=\partial\PT$ is exactly the one corresponding to the coherent system $\{M_n^{\zz_1\zz_2}\}$, see \S\ref{sub:pascal_triangle}.

\begin{remark}\label{rmk:Pascal_extreme_coherent_limits}
  Extreme coherent systems on branching graphs (i.e., the ones not expressible as nontrivial convex combinations of other coherent systems) correspond to delta measures on the boundary (see \cite{Kerov1998}). For $\PT$, the extreme coherent system corresponding to $p\in[0,1]$ looks as $M_{n}^{\text{ex},p}(k,l)=\frac{n!}{k!l!}p^k(1-p)^{l}$, and can be (pointwise) approximated by the systems $\{M_n^{\zz_1\zz_2}\}$ coming from the Kerov's operators (take $\zz_1/(\zz_1+\zz_2)\to p$). This property distinguishes the ``small'' branching graph $\PT$ (and its multidimensional generalization) from ``big'' branching graphs such as the Young graph, for which only very particular extreme coherent systems can be obtained as limits of the ones coming from the Kerov's operators.% (see also \S\ref{sec:remarks_about_plancherel_type_measures}).
\end{remark}

The mixing of the distinguished measures $M_n^{\zz_1\zz_2}$ (\ref{Kerov_coherent_Pascal}) by means of the negative binomial distribution $\pi_{\zz_1+\zz_2,\xi}$ simplifies the measures significantly ($\zz_1,\zz_2>0$):
\begin{equation*}
  M_{\xi}^{\zz_1\zz_2}(k,l)=\pi_{\zz_1+\zz_2,\xi}(k+l)\cdot M_{k+l}^{\zz_1\zz_2}(k,l)=
  \pi_{\zz_1,\xi}(k)\cdot
  \pi_{\zz_2,\xi}(l),\qquad (k,l)\in\PT.
\end{equation*}
That is, the components $k$ and $l$ of the vertex $(k,l)\in\PT$ become independent, and we see how the effect of Proposition \ref{prop:KO_disjoint} comes into play. For other graphs the mixing procedure also simplifies the measures. For example, the determinantal/Pfaffian structure of measures and dynamics on Young and Schur graphs is present only for mixed measures and not at the level of the coherent systems $\{M_n\}$ \cite{Borodin2000a}, \cite{Borodin2006}, \cite{Petrov2010Pfaffian}.

% subsection application_to_pascal_triangle (end)

% section coherent_systems_of_measures (end)

\section{Relative dimension functions} % (fold)
\label{sec:relative_dimension_functions}

In this section we discuss relative dimension functions on the graph $\Gb=(J(P),\km)$ with an UD-self-dual edge multiplicity function. It turns out that the action of the Kerov's operators on these functions has a particularly nice form.

\subsection{Definition of relative dimension functions} % (fold)
\label{sub:definition_of_relative_dimension_functions}

Let $a^{\da k}:=a(a-1)\ldots(a-k+1)$ denote the falling factorial power. The \emph{relative dimension functions} $\{\Pc_{\mu}^*\}_{\mu\in\Gb}$ are indexed by vertices of the graph, and are defined as
\begin{equation}\label{relative_dimension_functions}
  \Pc_{\mu}^*(\la):=|\la|^{\da|\mu|}\frac{\dim(\mu,\la)}{\dim\la},
  \qquad\la\in\Gb.
\end{equation}
Here the relative dimension $\dim(\cdot,\cdot)$ is defined in the end of \S\ref{sub:branching_graphs}. The notation $\Pc_{\mu}^*$ comes from \cite{Borodin2000} where it is assumed that these functions form an algebra. The ``$*$'' symbol here does not mean any duality: in concrete examples the ``$*$'' notation marks \emph{shifted} (or \emph{factorial}) functions. In particular, for the Young graph the functions $\Pc_{\mu}^*$ are the shifted Schur functions $s_\mu^*$ which are ``deformed'' versions of the ordinary Schur functions $s_\mu$  \cite{OkounkovOlshanski1996ShiftSchur}. On our abstract level we do not have any non-shifted functions $\Pc_\mu$, and deal only with $\Pc_\mu^*$'s.

The relative dimension functions $\Pc_\mu^*$ have the following properties which are readily checked:
\begin{enumerate}[{\bf1.\/}]
  \item\label{Pc_prop_varnothing}
  $\Pc_\varnothing^*\equiv1$.
  \item\label{Pc_prop_interpolation}
  (Interpolation property; cf. \cite[Thm. 3.2]{OkounkovOlshanski1996ShiftSchur}, \cite{Okounkov1996quantumImm})
  $\Pc_\mu^*(\la)=0$ unless $\mu\subseteq\la$. For $|\mu|=|\la|$, $\Pc_\mu^*(\la)$ is nonzero only for $\mu=\la$, and in this case $\Pc_\mu^*(\mu)={|\mu|!}/{\dim\mu}$.
  \item\label{Pc_prop_estimate}
  $0\le\Pc_\mu^*(\la)\le|\la|^{\da|\mu|}\le|\la|^{|\mu|}$ for all $\mu,\la\in\Gb$.
  \item\label{Pc_prop_p1}
  Let $\po$ denote the function which is equal to the level of a vertex: $\po(\la):=|\la|$. It is equal to
  \begin{equation*}
    \po=\sum_{\varkappa\colon\varkappa\searrow\varnothing}
    \km(\varnothing,\varkappa)\Pc^*_{\varkappa}.
  \end{equation*}

  \item\label{Pc_prop_Pieri}
  The functions $\Pc_\mu^*$ satisfy the following recurrence: 
  \begin{equation*}
    \po\cdot
    \Pc_\mu^*=|\mu|\Pc_\mu^*+
    \sum_{\rho\colon\rho\searrow\mu}
    \km(\mu,\rho)\Pc_\rho^* \qquad
    \mbox{for all $\mu\in\Gb$}.
  \end{equation*}
  \item\label{Pc_prop_la-mu}
  There is also another recurrence for the values of the functions $\Pc_\mu^*$:
  \begin{align*}
    (|\la|-|\mu|)\Pc_\mu^*(\la)&=
    |\la|\sum_{\nu\colon\nu\nearrow\la}
    p^\da_{n,n-1}(\la,\nu)
    \Pc_{\mu}^*(\nu),\qquad
    \mbox{for all $\mu,\la\in\Gb$, $|\la|=n$.}
  \end{align*}
\end{enumerate}

% subsection definition_of_relative_dimension_functions (end)

\subsection{Linear space $\A$} % (fold)
\label{sub:space_A}

Let $\A$ denote the linear span of the functions $\{\Pc_\mu^*\}_{\mu\in\Gb}$ (i.e., the space of \emph{finite} linear combinations of $\Pc_\mu^*$'s). This is a linear space, and we do not require $\A$ to have a structure of an algebra under pointwise multiplication (see also \S\ref{ssub:te-regular}). Let $\A_m$ be the span of $\{\Pc_\mu^*\}_{\mu\in\Gb,\,|\mu|\le m}$, $m=0,1,\ldots$. These are finite-dimensional linear spaces, and they provide an ascending filtration of $\A$:
\begin{equation*}
  \A_0\subset\A_1\subset\A_2\subset\dots\subset\A.
\end{equation*}

Until the end of this section we fix a triplet of Kerov's operators on $\Gb$. Let $\{M_n\}$ be the corresponding coherent system (\S\ref{sub:distinguished_coherent_systems_obtained_from_kerov_s_operators}), and let $M_\xi$ be the mixed measure (\S\ref{sub:mixing_of_measures}). We assume that $M_\xi(\la)>0$ for any $\la\in\Gb$, that is, the coherent system is nondegenerate, and also that $\zz=(\Hf\un\varnothing,\un\varnothing)>0$. By $\ell^2(\Gb,M_\xi)$ denote the Hilbert space of functions on $\Gb$ which are square integrable with respect to the measure $M_\xi$, with the inner product $(f,g)_{M_\xi}:=\sum_{\la\in\Gb}f(\la)\overline{g(\la)}M_{\xi}(\la)$.

\begin{proposition}\label{prop:completeness}
  The space $\A$ with the linear basis $\{\Pc_{\mu}^{*}\}_{\mu\in\Gb}$ is a dense subspace of $\ell^2(\Gb,M_\xi)$.
\end{proposition}
\begin{proof}
  The proof goes by steps.

  {\bf1.\/} We first show that the functions $\Pc_\mu^*$ are linearly independent. Consider a linear relation $\sum c_\nu\Pc_\nu^*\equiv 0$. Let $n$ be the minimal $|\nu|$ appearing in this relation. For any $\la\in\Gb_n$, we see by the interpolation property that in the sum $\sum c_\nu\Pc_\nu^*(\la)$ only the term with $\nu=\la$ survives. This implies that $c_\la=0$, and thus all $c_\nu$'s are zero by induction.

  {\bf2.\/} The functions $\Pc_\mu^*$ separate points of $\Gb$ because for any two vertices $\rho\ne\varkappa$, $|\rho|\le|\varkappa|$, the function $\Pc^*_{\varkappa}$ vanishes at $\rho$ and is nonzero at $\varkappa$.

  {\bf3.\/} Now let us show that each $\Pc_\mu^*$ is square integrable with respect to $M_\xi$. We have
  \begin{align*}
    (\Pc_\mu^*,&\Pc_\mu^*)_{M_\xi}=
    (1-\xi)^{\zz}
    \sum\nolimits_{\la\in\Gb}\Pc_\mu^*(\la)^2
    \frac{(\zz)_{|\la|}}{|\la|!}\xi^{|\la|}M_{|\la|}(\la)
    \\&\le
    \mathrm{const}
    \sum\nolimits_{\la\in\Gb}|\la|^{2|\mu|}
    \frac{(\zz)_{|\la|}}{|\la|!}\xi^{|\la|}M_{|\la|}(\la)
    \\&=
    \mathrm{const}
    \sum\nolimits_{n=0}^{\infty}n^{2|\mu|}
    \frac{(\zz)_{n}}{n!}\xi^{n}\sum\nolimits_{\la\in\Gb_n}M_{n}(\la)
    =\mathrm{const}
    \sum\nolimits_{n=0}^{\infty}n^{2|\mu|}
    \frac{(\zz)_{n}}{n!}\xi^{n}<\infty
  \end{align*}
  because $\frac{(\zz)_n}{n!}=\frac{\Gamma(n+\zz)}{\Gamma(\zz)\Gamma(n+1)}\sim \mathrm{const}\cdot n^{\zz-1}$ as $n\to\infty$, see \cite[(1.18.5)]{Erdelyi1953}.

  {\bf4.\/} Now we can show that the space $\A$ is dense in $\ell^2(\Gb,M_\xi)$. It suffices to approximate any function $F$ on $\Gb$ with finite support by functions from $\A$. Take $N$ so large that $\mathop{\mathrm{supp}}(F)$ lies inside the set $\Gb_{\le N}:=\{\la\in\Gb\colon|\la|\le N\}$. Let $\chi_N$ denote the characteristic function of $\Gb_{\le N}$. Since (by step 2) one can find $f\in\A$ with any prescribed values on $\Gb_{\le N}$, it suffices to approximate (by elements of $\A$) any function of the form $f\chi_N$, where $f\in\A$ and $N=0,1,2,\dots$.

  We will approximate $f\chi_N$ by elements of the form $f \cdot h(\po)$, where $h$ is some polynomial. Observe that by properties in \S\ref{sub:definition_of_relative_dimension_functions}, $h(\po)\in\A$. Let $m$ be the degree of $f$ under the filtration of $\A$, then
  \begin{equation*}
    |f(\la)|^2\le \mathrm{const}(1+|\la|^{2m}).
  \end{equation*}
  Therefore, similarly to the previous step, we can write
  \begin{align}\label{Pc_density_proof}
    \sum_{\la\in\Gb}
    \left|f(\la)\big(\chi_N(\la)-h(\po(\la))\big)\right|^2\le
    \mathrm{const}
    \sum_{n=0}^{\infty}
    (1+n^{2m})
    |\bar\chi_{N}-h(n)|^{2}\pi_{\zz,\xi}(n),
  \end{align}
  where $\bar\chi_{N}$ is the characteristic function of the subset $\{0,\ldots,N\}\subset\Z_{\ge0}$ and $\pi_{\zz,\xi}$ is the negative binomial distribution (\ref{NegBinom}). Since $n^{2m}\pi_{\zz,\xi}(n)\sim \mathrm{const}\,\pi_{\zz+2m,\xi}(n)$ as $n\to\infty$ by \cite[(1.18.5)]{Erdelyi1953}, (\ref{Pc_density_proof}) is bounded by 
  \begin{equation*}
    \sum\nolimits_{n=0}^{\infty}|\bar\chi_{N}(n)-h(n)|^{2}
    \pi_{\zz+2m,\xi}(n).
  \end{equation*}
  This sum be made arbitrarily small because the polynomials are dense in the Hilbert space $\ell^2(\Z_{\ge0},\pi_{\zz+2m,\xi})$, see \cite[Thm.\,2.3.3]{Akhiezer1965Moment}, \cite{riesz1923probleme}.
\end{proof}

% subsection space_A (end)

\subsection{Action of Kerov's operators on relative dimension functions} % (fold)
\label{sub:action_of_Kerov_s_operators}

The functions $\Pc_\mu^*$ belong to the weighted Hilbert space $\ell^2(\Gb,M_\xi)$, and the Kerov's operators act in another space $\ell^2(\Gb)$. We use the isometry of these spaces:
\begin{equation}\label{istry}
  \istry_{\xi}\colon \ell^2(\Gb,M_\xi)\to\ell^2(\Gb),\qquad
  (\istry_{\xi} f)(\la):=f(\la)\cdot(M_{\xi}(\la))^{\frac12},
  \qquad\la\in\Gb.
\end{equation}
\begin{theorem}\label{thm:action_of_KO}
  The action of the Kerov's operators $\Uf$ and $\Df$ on the functions $\istry_\xi\Pc_{\mu}^*\in\ell^2(\Gb)$ (where $\mu$ runs over all vertices of $\Gb$) is given by
  \begin{align*}
    \Uf(\istry_\xi\Pc_\mu^*)&=
    \xi^{-\frac12}({\po-|\mu|})(\istry_\xi\Pc_\mu^*),
    \\
    \Df(\istry_\xi\Pc_\mu^*)&=
    \xi^{\frac12}\big(\po+|\mu|+\zz\big)(\istry_\xi\Pc_{\mu}^*)
    +\xi^{\frac12}
    \sum\nolimits_{\rho\colon\rho\nearrow\mu}\km(\rho,\mu)\ufunc(\mu/\rho)^2
    (\istry_\xi\Pc_\rho^*).
  \end{align*}
\end{theorem}

First, we need the following combinatorial lemma:
\begin{lemma}[{cf. \cite[proof of Thm.\,4.1(2)]{Borodin2007}}]
\label{lemma:combinatorial_identity}
  For all $\la,\mu\in\Gb$ with $|\la|=n$ and $|\mu|=m$ the following identity holds:
  \begin{equation}
    \label{comb_id}
    \begin{array}{r}
      \displaystyle
      \sum\limits_{\nu\colon\nu\searrow\la} 
      \km(\la,\nu)\ufunc(\nu/\la)^2\dim(\mu,\nu)=
      \sum_{\rho\colon\rho\nearrow\mu}
      \km(\rho,\mu)\ufunc(\mu/\rho)^2\dim(\rho,\la)
      \quad\qquad\\
      +
      (n+m+\zz)(n-m+1)\dim(\mu,\la).
    \end{array}
  \end{equation}
\end{lemma}
\begin{proof}
  We may assume that $n\ge m-1$, otherwise all relative dimensions above vanish. For $\mu=\varnothing$, the statement of the lemma is equivalent to the coherency property of the measures $\{M_n\}$ (Theorem \ref{thm:Kerov_coherency}). Thus, $m\ge1$.

  Let us consider two cases. Assume first that $\mu$ is not inside $\la$. Then (for identity (\ref{comb_id}) to be nontrivial) the set difference $\mu\setminus\la$ must consist of one point of $P$, and in the LHS of (\ref{comb_id}) there is only one $\nu$ such that $\dim(\mu,\nu)$ does not vanish. Also, in the RHS of (\ref{comb_id}) in the sum there is only one suitable $\rho$. Note that $\nu/\la=\mu/\rho$. Thus, our identity becomes
  \begin{equation*}
    \km(\la,\nu)\ufunc(\nu/\la)^2\dim(\mu,\nu)
    =
    \km(\rho,\mu)\ufunc(\mu/\rho)^2\dim(\rho,\la)
  \end{equation*}
  (for those $\nu$ and $\rho$). Let us choose any path from $\mu$ to $\nu$; $\dim(\mu,\nu)$ is the sum of weights of all such paths. Assume that on the level $n$ (recall that $|\nu|=n+1$) this path goes through a vertex $\nu_\bullet$. Set $\la_\bullet:=\nu_\bullet\cap\la$. By the UD-self-duality of the multiplicity function,
  \begin{equation*}
    \km(\la,\nu)\km(\nu_\bullet,\nu)=
    \km(\la_\bullet,\la)\km(\la_\bullet,\nu_\bullet).
  \end{equation*}
  Clearly, $\km(\nu_\bullet,\nu)$ is counted in $\dim(\mu,\nu)$, and $\km(\la_\bullet,\la)$ is counted in $\dim(\rho,\la)$. Thus, we have reduced the problem from the pair $(\la,\nu)$ to the lower pair $(\la_\bullet,\nu_\bullet)$, and the rest of the proof goes by induction because there is a bijection between paths from $\mu$ to $\nu$ and paths from $\rho$ to $\la$ (in $J(P)$). Thus, we have established our identity for $\mu\not\subseteq\la$.

  Now assume that $\mu\subseteq\la$. Multiplying (\ref{comb_id}) by $\prod_{\square\in\la/\mu}\ufunc(\square)^2$, we have (see the proof of Theorem \ref{thm:Kerov_coherency}):
  \begin{align*}
    \sum\nolimits_{\nu\colon\nu\searrow\la} 
    \km(\la,\nu)&
    \left(\prod\nolimits_{\square\in\nu/\mu}
    \ufunc(\square)^2\right)\dim(\mu,\nu)\\&=
    \sum\nolimits_{\nu\colon\nu\searrow\la} 
    \km(\la,\nu)
    (\hat\Uf^{n+1-m}\unt\mu,\unt\nu)\\&=
    (\hat\Uf^{n+1-m}\unt\mu,\hat \Df^*\unt\la)
    =
    (\hat \Df \hat\Uf^{n+1-m}\unt\mu,\unt\la)
    \\&=
    (\hat\Uf^{n+1-m}\hat\Df\unt\mu,\unt\la)+
    \sum\nolimits_{i=0}^{n-m}
    (\hat\Uf^{n-m-i}\Hf\hat\Uf^{i}\unt\mu,\unt\la)
    \\&=
    \sum\nolimits_{\rho\colon\rho\nearrow\mu}\km(\rho,\mu)
    \left(\prod\nolimits_{\square\in\la/\rho}\ufunc(\square)^2\right)
    \dim(\rho,\la)\\&\qquad+
    (n+k+\zz)(n-k+1)
    (\hat\Uf^{n-m}\unt\mu,\unt\la)
    \\&=
    \sum\nolimits_{\rho\colon\rho\nearrow\mu}\km(\rho,\mu)
    \left(\prod\nolimits_{\square\in\la/\rho}\ufunc(\square)^2\right)
    \dim(\rho,\la)\\&\qquad+
    (n+k+\zz)(n-k+1)
    \left(\prod\nolimits_{\square\in\la/\mu}\ufunc(\square)^2\right)
    \dim(\mu,\la).
  \end{align*}
  Dividing again by $\prod_{\square\in\la/\mu}\ufunc(\square)^2$, we see that (\ref{comb_id}) holds.
\end{proof}

\smallskip

\par\noindent
\emph{Proof of Theorem \ref{thm:action_of_KO}.}
  Now it is not hard to establish the theorem. Let us first consider the operator $\Uf$ (this case does not require Lemma \ref{lemma:combinatorial_identity}). We have (let $|\la|=n$, $|\mu|=m$)
  \begin{align*}
    \big(\Uf&(\istry_\xi\Pc_\mu^*)\big)(\la)=
    \sum\nolimits_{\varkappa\colon\varkappa\nearrow\la}
    \km(\varkappa,\la)\ufunc(\la/\varkappa)
    (M_\xi(\varkappa))^{\frac12}
    \Pc_\mu^*(\varkappa)\\&=
    (n-1)^{\da m}
    (M_\xi(\la))^{\frac12}
    \sum\nolimits_{\varkappa\colon\varkappa\nearrow\la}
    \km(\varkappa,\la)\ufunc(\la/\varkappa)
    \left(\frac{M_{\xi}(\varkappa)}{M_{\xi}(\la)}\right)^{\frac12}
    \frac{\dim(\mu,\varkappa)}{\dim\varkappa}\\&=
    (n-1)^{\da m}\frac{n}{\sqrt\xi\cdot\dim\la}
    (M_\xi(\la))^{\frac12}
    \sum\nolimits_{\varkappa\colon\varkappa\nearrow\la}
    \km(\varkappa,\la)\dim(\mu,\varkappa)\\&=
    \frac{n-m}{\sqrt\xi}
    (M_\xi(\la))^{\frac12}
    n^{\da m}\frac{\dim(\mu,\la)}{\dim\la}
    =\xi^{-\frac12}({\po(\la)-|\mu|})
    \big(\istry_\xi\Pc_\mu^*\big)(\la).
  \end{align*}

  Now let us deal with the operator $\Df$. Here we will use Lemma \ref{lemma:combinatorial_identity}. We have ($|\la|=n$, $|\mu|=m$):
  \begin{align*}
    \big(\Df&(\istry_\xi\Pc_\mu^*)\big)(\la)=
    \sum\nolimits_{\nu\colon\nu\searrow\la}
    \km(\la,\nu)\ufunc(\nu/\la)
    (M_\xi(\nu))^{\frac12}
    \Pc_\mu^*(\nu)\\&=
    (n+1)^{\da m}
    (M_\xi(\la))^{\frac12}
    \sum\nolimits_{\nu\colon\nu\searrow\la}
    \km(\la,\nu)\ufunc(\nu/\la)
    \left(\frac{M_\xi(\nu)}{M_{\xi}(\la)}\right)^{\frac12}
    \frac{\dim(\mu,\nu)}{\dim\nu}
    \\&=
    (n+1)^{\da m}\frac{\sqrt\xi}{(n+1)\dim\la}
    (M_\xi(\la))^{\frac12}
    \sum\nolimits_{\nu\colon\nu\searrow\la}
    \km(\la,\nu)\ufunc(\nu/\la)^2
    \dim(\mu,\nu)\\&=
    n^{\da (m-1)}\frac{\sqrt\xi}{\dim\la}
    (M_\xi(\la))^{\frac12}
    \Big(
    \sum\nolimits_{\rho\colon\rho\nearrow\mu}
    \km(\rho,\mu)\ufunc(\mu/\rho)^2\dim(\rho,\la)
    \\&\qquad \qquad \qquad \qquad \qquad \qquad \qquad+
    (n+m+\zz)(n-m+1)\dim(\mu,\la)
    \Big)\\&=
    \sqrt\xi\big(\po(\la)+|\mu|+\zz\big)
    (\istry_\xi\Pc_{\mu}^*)(\la)
    +\sqrt\xi
    \sum\nolimits_{\rho\colon\rho\nearrow\mu}
    \km(\rho,\mu)\ufunc(\mu/\rho)^2
    (\istry_\xi\Pc_\rho^*)(\la).
  \end{align*}
  This concludes the proof.
\qed

\smallskip

\begin{remark}
  The papers \cite{Borodin2007}, \cite{Petrov2007}, \cite{Olshanski2009}, \cite{petrov2009eng} also deal with the operators acting similarly to $\Uf$ and $\Df$ in Theorem \ref{thm:action_of_KO}. However, in those papers the ``up'' and ``down'' notation was switched. In the present paper we follow the notation of Okounkov \cite{Okounkov2001a} for the Kerov's operators.
\end{remark}

% subsection action_of_Kerov_s_operators (end)

\subsection{The case of Pascal triangle} % (fold)
\label{sub:the_case_of_pascal_triangle}

For the Pascal triangle (\S\ref{sub:pascal_triangle}) one has for two vertices $\mu=(k,l),\la=(x,y)\in\PT$:
\begin{equation*}
  \frac{\dim(\mu,\la)}{\dim\la}
  =\frac{\binom{x+y-k-l}{x-k}}{\binom{x+y}{x}}=
  \frac{x^{\da k}y^{\da l}}{(x+y)^{\da(k+l)}},
\end{equation*}
so the relative dimension functions have the form 
\begin{equation*}
  \Pc_{(k,l)}^{*}(x,y)=x^{\da k}y^{\da l},\qquad
  (k,l)\in\PT.
\end{equation*}
We see that these functions factorize in accordance with Proposition \ref{prop:KO_disjoint}. Next, observe that $(\istry_\xi \Pc_{(k,l)}^{*})(x,y)=(1-\xi)^{\frac{\zz_1+\zz_2}2}\xi^{\frac{x+y}2}x^{\da k}y^{\da l}\sqrt{\frac{(\zz_1)_x(\zz_2)_y}{x!y!}}$. The action of the Kerov's operators $\Uf$ and $\Df$ (\S\ref{sub:kerov_s_operators_for_pascal_triangle}) on relative dimension functions given by Theorem \ref{thm:action_of_KO} is equivalent to the following identities which are readily verified:
\begin{align*}
  \mbox{Action of $\Uf$}:&\qquad x(x-1)^{\da k}=(x-k)x^{\da k};\\
  \mbox{Action of $\Df$}:&\qquad (x+\zz)(x+1)^{\da k}=
  (x+k+\zz)x^{\da k}+k(k-1+\zz)x^{\da (k-1)}.
\end{align*}

% subsection the_case_of_pascal_triangle (end)

% section relative_dimension_functions (end)

\section{Up/down Markov chains and their spectral structure} % (fold)
\label{sec:up_down_markov_chains}

Assume that $\Gb=(J(P),\km)$ is a graph of ideals with a fixed triplet ($\Uf,\Df,\Hf$) of Kerov's operators. Let $\{M_n\}$ be the coherent system of measures on the floors of $\Gb$ corresponding to this triplet (\S\ref{sub:distinguished_coherent_systems_obtained_from_kerov_s_operators}). As always, we assume that $\{M_n\}$ is nondegenerate, and that $\zz>0$.

\subsection{Up/down Markov chains} % (fold)
\label{sub:up_down_markov_chains}

In \S\ref{sub:coherent_systems} we have defined two families of Markov transition kernels --- $p^{\da}_{n,n-1}$ from $\Gb_{n}$ to $\Gb_{n-1}$ (which does not depend on a choice of a coherent system), and $p^{\ua}_{n,n+1}$ from $\Gb_{n}$ to $\Gb_{n+1}$ (which depends on $\{M_n\}$). Thus, it is possible to define a family of Markov chains on $\Gb_n$ by composing ``up'' and ``down'' steps: first from $\Gb_{n}$ to $\Gb_{n+1}$, and then back to $\Gb_{n}$. The transition operator of the $n$th chain is given by
\begin{equation}\label{T_n_def}
  T_n(\la,\tilde\la):=\sum_{\nu\colon|\nu|=n+1}
  p^{\ua}_{n,n+1}(\la,\nu)p^{\da}_{n+1,n}(\nu,\tilde\la),\qquad
  \la,\tilde\la\in\Gb_n.
\end{equation}
Using the definitions of \S\ref{sec:coherent_systems_of_measures}, it can be readily seen that the chain $T_n$ preserves the measure $M_n$ and is reversible with respect to it. About up/down Markov chains and similar objects see also \cite{Fulman2005}, \cite[\S1]{Borodin2007}, \cite{Fulman2007}.

Observe that $p^{\ua}_{n,n+1}(\la,\nu)p^{\da}_{n+1,n}(\nu,\tilde\la)$ vanishes unless either $\tilde\la=\la$, or $\tilde\la$ is obtained from $\la$ by a minimal possible transformation, i.e., $\tilde\la=(\la\cup\square_1)\setminus\square_2$ for some elements $\square_1,\square_2$ of the underlying poset $P$. Thus, each nontrivial step of the up/down chain consists in such relocation of a box.

% subsection up_down_markov_chains (end)

\subsection{Up/down Markov chains and Kerov's operators} % (fold)
\label{sub:up_down_markov_chains_and_kerov_s_operators}

Each Markov transition operator $T_n$ is acting in the finite-dimensional space $\fun(\Gb_n)$ of all functions on $\Gb_n$. Since the space $\A$ (\S\ref{sub:space_A}) separates points of $\Gb$, the restrictions of functions $f\in\A$ to $\Gb_n$ (denoted by $f_n$) exhaust the space $\fun(\Gb_n)$. Therefore, the action of $T_n$ in $\fun(\Gb_n)$ is completely determined by its action on $(\Pc_\mu^*)_{n}$. The latter action is given by
\begin{proposition}\label{prop:T_n_action}
  For every $\mu\in\Gb$ the following identity holds:
  \begin{align}
    \label{T_n_action}
    \big(T_n-\mathbf{1}\big)(\Pc_\mu^*)_n&=
    -\frac{|\mu|(|\mu|-1+\zz)}{(n+1)(n+\zz)}(\Pc_\mu^*)_n
    \\&\qquad \quad\qquad+\frac{n+1-|\mu|}{(n+1)(n+\zz)}
    \sum_{\rho\colon\rho\nearrow\mu}
    \km(\rho,\mu)\ufunc(\mu/\rho)^{2}(\Pc_\rho^*)_n.
    \nonumber
  \end{align}
\end{proposition}
\begin{proof}
  Define restrictions of the Kerov's operators $\Df$ and $\Uf$ which act as follows:
  \begin{align*}
    \Df_{n,n-1}&:=\Df|_{{}_{\scriptstyle\fun(\Gb_n)}}\colon
    \fun(\Gb_n)\to\fun(\Gb_{n-1}),\\
    \Uf_{n,n+1}&:=\Uf|_{{}_{\scriptstyle\fun(\Gb_n)}}\colon
    \fun(\Gb_n)\to\fun(\Gb_{n+1}).
  \end{align*}
  
  From (\ref{T_n_def}) and (\ref{Kerov_UfDf_def}) it follows that the operator $T_n$ is conjugate to the operator $\frac1{(n+1)(n+\zz)}\Df_{n+1,n}\Uf_{n,n+1}$ which acts in $\fun(\Gb_n)$. Indeed, observe that
  \begin{equation}\label{}
    p^{\ua}_{n,n+1}(\la,\nu)p^\da_{n+1,n}(\nu,\tilde\la)
    =
    \frac1{(n+1)(n+\zz)}
    \frac{\dim\tilde\la}{\dim\la}
    \ufunc(\nu/\la)^2\km(\la,\nu)\km(\tilde\la,\nu)
  \end{equation}
  for all $\la,\tilde\la\in\Gb_n$ and $\nu\in\Gb_{n+1}$. On the other hand, for any $f\in\ellf^2(\Gb)$ and $\la\in\Gb$ one has
  \begin{align*}
    (\Df\Uf f)(\la)&=
    \sum\nolimits_{\nu\colon\nu\searrow\la}
    \km(\la,\nu)\ufunc(\nu/\la)(\Uf f)(\nu)
    \\&=\sum\nolimits_{\nu\colon\nu\searrow\la}
    \sum\nolimits_{\tilde\la\colon\tilde\la\nearrow\nu}
    \km(\la,\nu)\km(\tilde\la,\nu)\ufunc(\nu/\la)\ufunc(\nu/\tilde\la)
    f(\tilde\la).
  \end{align*}
  Thus, we get the following equality of operators in $\fun(\Gb_n)$:
  \begin{equation*}
    T_n=\frac1{(n+1)(n+\zz)}
    {\mathsf{d}}_{n}^{-1}\Df_{n+1,n}\Uf_{n,n+1}{\mathsf{d}}_n,
  \end{equation*}
  where $({\mathsf{d}}_ng)(\la):=\Big(\dim\la\cdot \prod_{\square\in\la}\ufunc(\square)\Big)g(\la)$ is a diagonal operator in $\fun(\Gb_n)$. Next, we can restrict the isometry $\istry_\xi$ (\ref{istry}) to $\fun(\Gb_n)$, and it is clear that for any $g\in\fun(\Gb_n)$ one has $({\mathsf{d}}_n\istry_\xi^{-1}|_{{}_{\scriptstyle\fun(\Gb_n)}} g)(\la)=(1-\xi)^{\zz/2}\xi^{-n/2}n!\cdot g(\la)$, which means that ${\mathsf{d}}_n\istry_\xi^{-1}|_{{}_{\scriptstyle\fun(\Gb_n)}}$ is a scalar operator. Thus, the action of $T_n$ on $\Pc_\mu^*$ has the form
  \begin{align*}
    T_n(\Pc_\mu^*)_{n}&=
    \frac1{(n+1)(n+\zz)}
    {\mathsf{d}}_{n}^{-1}\Df_{n+1,n}\Uf_{n,n+1}{\mathsf{d}}_n((\istry_\xi^{-1}\istry_\xi)\Pc_\mu^*)_{n}\\&=
    \frac1{(n+1)(n+\zz)}
    \istry_\xi^{-1}|_{{}_{\scriptstyle\fun(\Gb_n)}}
    \Df_{n+1,n}\Uf_{n,n+1}(\istry_\xi\Pc_\mu^*)_{n}.
  \end{align*}
  By Theorem \ref{thm:action_of_KO} and because $\po|_{{}_{\scriptstyle\fun(\Gb_k)}}\equiv k$ for every $k$, we obtain
  \begin{align*}
    \Uf_{n,n+1}(\istry_\xi\Pc_\mu^*)_{n}&=\xi^{-\frac12}(n+1-|\mu|)(\istry_\xi\Pc_\mu^*)_{n+1},
    \\
    \Df_{n+1,n}(\istry_\xi\Pc_\mu^*)_{n+1}
    &=\xi^{\frac12}\big(n+|\mu|+\zz\big)(\istry_\xi\Pc_{\mu}^*)_{n}
    +\xi^{\frac12}
    \sum_{\rho\colon\rho\nearrow\mu}\km(\rho,\mu)\ufunc(\mu/\rho)^2
    (\istry_\xi\Pc_\rho^*)_{n}.
  \end{align*}
  Putting all together, we see that the claim holds.
\end{proof}

\begin{remark}
  Formula of Proposition \ref{prop:T_n_action} (together with formulas similar to Theorem \ref{thm:action_of_KO}) first appeared for the Young graph in \cite[Lemma 5.2]{Borodin2007} as one of the key ingredients in construction of infinite-dimensional diffusions on the Thoma simplex (\ref{Thoma_simplex}). Similar formulas were obtained in \cite{Petrov2007}, \cite{petrov2009eng} for other branching graphs and were also used to construct infinite-di\-men\-sional diffusions.
\end{remark}

% subsection up_down_markov_chains_and_kerov_s_operators (end)

\subsection{Spectral structure of up/down Markov chains} % (fold)
\label{sub:spectral_structure_of_up_down_markov_chains}

Here we discuss the spectral structure of the operators $T_n$ in $\fun(\Gb_n)$. It is possible to obtain the results of this subsection using a method similar to \cite[Thm.\,4.1 and 4.3]{Fulman2007} (see also \cite[\S4]{stanley1988differential}, \cite[Thm.\,1.6.5]{fomin1994duality}) which is solely based on the commutation relations for the operators $\Uf$ and $\Df$. Instead, we make use of the algebraic expression for $T_n$ (Proposition \ref{prop:T_n_action}).

\begin{proposition}\label{prop:T_n_spectrum}
  The generator $T_n-\mathbf{1}$ of the $n$th up/down Markov chain on $\Gb_n$ has eigenvalues $\left(-\frac{j(j-1+\zz)}{(n+1)(n+\zz)}\right)$, $j=0,\ldots,n$. The $j$th eigenvalue has multiplicity $\#\Gb_j-\#\Gb_{j-1}$ (by agreement, $\#\Gb_{-1}=0$). In particular, since $\Gb_{0}=\{\varnothing\}$, the eigenvalue $0$ is simple, and the chain $T_n$ possesses a spectral gap.
\end{proposition}
\begin{proof}
  The operator $T_n-\mathbf{1}$ acts in the space $\fun(\Gb_n)$ of dimension $\#\Gb_n$, and its action is completely determined by (\ref{T_n_action}). By the interpolation property (\S\ref{sub:definition_of_relative_dimension_functions}), the functions $\{(\Pc_\mu^*)_n\}_{\mu\in\Gb_n}$ are linearly independent in $\fun(\Gb_n)$, and therefore form a basis of that space. 

  Let us denote by $V_{n-1}\subset\fun(\Gb_n)$ the linear span of the functions $\{(\Pc_\varkappa^*)_n\}_{\varkappa\in\Gb_{n-1}}$. Clearly, $\dim V_{n-1}=\#\Gb_{n-1}$. Let $W_n$ be any subspace such that $\fun(\Gb_n)=W_n\oplus V_{n-1}$ (direct sum of linear spaces). From (\ref{T_n_action}) it follows that the action of $T_n-\mathbf{1}$ in $\fun(\Gb_n)$ is triangular with respect to that decomposition. This implies that $T_n-\mathbf{1}$ has eigenvalue $\left(-\frac{n(n-1+\zz)}{(n+1)(n+\zz)}\right)$ with multiplicity at least $\#\Gb_n-\#\Gb_{n-1}$. The rest of the proposition is verified by induction inside $V_{n-1}$. This concludes the proof. 
\end{proof}

This proposition does not provide an explicit construction of eigenfunctions of $T_n-\mathbf{1}$. It is worth noting that for the second type of Markov dynamics which we consider (\S\ref{sec:markov_jump_dynamics}) such an explicit diagonalization of the generator is accessible.

% subsection spectral_structure_of_up_down_markov_chains (end)

\subsection{Up/down Markov chains on Pascal triangle} % (fold)
\label{sub:up_down_markov_chains_on_pascal_triangle}

Consider the $n$th up/down Markov chain for the Pascal triangle corresponding to the distinguished coherent system $\{M_{n}^{\zz_1,\zz_2}\}$ (\S\ref{sub:application_to_pascal_triangle}). By Proposition \ref{prop:T_n_spectrum}, the generator $T_n-\mathbf{1}$ has simple eigenvalues of the form $\left(-\frac{j(j-1+\zz_1+\zz_2)}{(n+1)(n+\zz_1+\zz_2)}\right)$, $j=0,\ldots,n$, because for this example $\#\PT_j=j+1$.

For fixed $\zz_1,\zz_2>0$ the up/down chains on the Pascal triangle converge as $n\to\infty$ (under suitable space and time scalings, see \cite[\S2]{Borodin2007}) to a diffusion on the boundary $\partial\PT=[0,1]$. This diffusion is reversible with respect to the corresponding Beta distribution on $[0,1]$ (\S\ref{sub:application_to_pascal_triangle}). In fact, this is the well-known Wright-Fisher diffusion, e.g., see \cite{Ewens1979} or \cite{Ethier1986}. Similar convergence leading to diffusions on the boundary holds for other branching graphs \cite{Borodin2007}, \cite{Petrov2007}, \cite{Olshanski2009}, \cite{petrov2009eng}. Note that in these papers the boundaries are infinite-dimensional.

% subsection up_down_markov_chains_on_pascal_triangle (end)

% section up_down_markov_chains (end)

\section{Markov jump dynamics. Diagonalization of generator} % (fold)
\label{sec:markov_jump_dynamics}

\subsection{Markov processes on $\Gb$} % (fold)
\label{sub:markov_processes_on_}

Throughout the whole section we fix a graph of ideals $\Gb=(J(P),\km)$ with a UD-self-dual multiplicity function and a triplet of Kerov's operators $(\Uf,\Df,\Hf)$. As always, the corresponding coherent system $\{M_n\}$ is assumed to be nondegenerate and such that $\zz=(\Hf\un\varnothing,\un\varnothing)$ is positive. For any $\xi\in(0,1)$ we define a continuous-time Markov jump dynamics $\labf_{\xi}$ on the graph $\Gb$ preserving the corresponding mixed measure $M_\xi$ (\S\ref{sub:mixing_of_measures}). The definition of the process $\labf_{\xi}$ follows the construction of Borodin--Olshanski \cite{Borodin2006}.
 
One starts with a birth and death process on $\Z_{\ge0}$ with the following jump rates ($n=0,1,\ldots$):
\begin{equation}\label{NegBinom_jump_rates}
  q_{n,n-1}:=\frac{n}{1-\xi},\qquad
  q_{n,n+1}:=\frac{\xi(n+\zz)}{1-\xi},\qquad
  q_{n,n}:=-\frac{n+\xi(n+\zz)}{1-\xi}
\end{equation}
All other jump rates are zero. The Markov process with jump rates (\ref{NegBinom_jump_rates}) can start from any point and any probability distribution. This process preserves the negative binomial distribution $\pi_{\zz,\xi}$ (\ref{NegBinom}) and is reversible with respect to it. By $\n_{\zz,\xi}$ denote the equilibrium version of the process, i.e., the process starting from the invariant distribution. For more detail on $\n_{\zz,\xi}$ see \cite[\S4.3]{Borodin2006}, and also \cite{KMG57BDClassif}, \cite{KMG58Linear} for a general treatment of birth and death processes.

Following and generalizing the constructions of \cite{Borodin2006}, \cite{Petrov2010Pfaffian}, we define a Markov jump process $\labf_{\xi}$ on the whole graph $\Gb$. The evolution of this process goes as follows. The level $|\labf_{\xi}|$ of a random vertex evolves according to the birth and death process $\n_{\zz,\xi}$. If at some moment the process $\n_{\zz,\xi}$ chooses to go one step up (it cannot jump twice at a single moment of time), say, from $n$ to $n+1$, then $\labf_{\xi}$ also goes one floor up in the graph $\Gb$. The vertex to which $\labf_{\xi}$ jumps is chosen according to the up transition probabilities $p^\ua_{n,n+1}(\labf_{\xi},\cdot)$ (\S\ref{sub:coherent_systems}). Similar thing happens when the process $\n_{\zz,\xi}$ decides to go down. Then the vertex in $\Gb$ to which $\labf_{\xi}$ jumps is chosen according to the down transition probabilities $p^\da_{n,n-1}(\labf_{\xi},\cdot)$. More formally, the jump rates of $\labf_{\xi}$ are as follows (here $\la\in\Gb_n$, $n=0,1,\ldots$):
\begin{align*}
  Q_{\la\mu}&=
  (1-\xi)^{-1}{n}\cdot p^\da_{n,n-1}(\la,\mu),&\mu\nearrow\la
  \\
  Q_{\la\nu}&=
  (1-\xi)^{-1}\xi(n+\zz)\cdot p^\ua_{n,n+1}(\la,\nu),&\nu\searrow\la
  \\
  Q_{\la\la}&=-(1-\xi)^{-1}\big(n+\xi(n+\zz)\big).
\end{align*}
All other jump rates are zero. The process with jump rates $\{Q_{\la\mu}\}$ preserves the measure $M_{\xi}$ on $\Gb$ and is reversible with respect to it. By agreement, let $\labf_{\xi}$ denote the equilibrium version of the process. For more detail about the process $\labf_{\xi}$ in the case of the Young graph see \cite[\S4.4]{Borodin2006}.

\begin{remark}
  The $n$th up/down Markov chain on $\Gb_n$ (\S\ref{sec:up_down_markov_chains}) can be reconstructed from $\labf_{\xi}$ as follows. Condition the process $\labf_{\xi}$ to stay in the set $\Gb_n\times\Gb_{n+1}$. Take its embedded Markov chain, that is, consider the process only at the times of jumps. This yields a Markov chain on $\Gb_n\times \Gb_{n+1}$ which belongs to $\Gb_n$ at, say, even discrete time moments. Taking this chain at the even moments, we reconstruct back the up/down Markov chain on $\Gb_n$with the transition operator $T_n$ (\ref{T_n_def}).
\end{remark}

% subsection markov_processes_on_ (end)

\subsection{Generator of the Markov jump process} % (fold)
\label{sub:generator_of_the_markov_jump_process}

Let ${\mathsf{A}}_{\xi}$ be the following operator acting on the space of finitely supported functions $\ellf^{2}(\Gb,M_{\xi})\subset\ell^{2}(\Gb,M_{\xi})$:
\begin{equation*}
  ({\mathsf{A}}_{\xi}f)(\la):=
  \sum\nolimits_{\rho\in\Gb}Q_{\la\rho}f(\rho).
\end{equation*}
The closure $\bar {\mathsf{A}}_{\xi}$ of this operator (closability follows from Proposition \ref{prop:Bxi_closability} below) in $\ell^{2}(\Gb,M_{\xi})$ is the Markov generator of the process $\labf_{\xi}$ on $\Gb$. 

Using the isometry $\istry_\xi$ (\ref{istry}), we can consider the corresponding operator ${\mathsf{B}}_{\xi}:=\istry_\xi {\mathsf{A}}_{\xi}\istry_\xi^{-1}$ which acts in $\ellf^2(\Gb)$. 

\begin{proposition}\label{prop:Bxi_action}
  The operator ${\mathsf{B}}_{\xi}$ is expressed through the Kerov's operators as follows:
  \begin{equation*}
    {\mathsf{B}}_{\xi}=
    \frac{\sqrt\xi}{1-\xi}(\Uf+\Df)-
    \frac12\frac{1+\xi}{1-\xi}\Hf+
    \frac\zz2\mathbf{1}.
  \end{equation*}
\end{proposition}
\begin{proof}
  This is established by a straightforward computation as in \cite[\S9]{Petrov2010Pfaffian}.
\end{proof}

Using the above proposition and Theorem \ref{thm:action_of_KO}, we get the following:
\begin{corollary}\label{corollary:action_of_Axi_on_Pmu}
  The generator $\bar {\mathsf{A}}_{\xi}$ of the process $\labf_\xi$ acts on the relative dimension functions $\Pc_\mu^*\in\ell^2(\Gb,M_{\xi})$ (\S\ref{sec:relative_dimension_functions}) in the following way:
  \begin{equation}\label{action_of_A_xi}
    \bar{\mathsf{A}}_{\xi}\Pc_\mu^*=-|\mu|\Pc_\mu^*+\frac{\xi}{1-\xi}
    \sum_{\rho\colon\rho\nearrow\mu}
    \km(\rho,\mu)\ufunc(\mu/\rho)^2\Pc_\rho^*,
    \quad\mbox{for every $\mu\in\Gb$}.
  \end{equation}
\end{corollary}

Thus, one can describe the spectrum of the generator $\bar{\mathsf{A}}_{\xi}$ in $\ell^2(\Gb,M_{\xi})$:
\begin{proposition}\label{prop:spectrum_of_A}
  The generator $\bar{\mathsf{A}}_{\xi}$ (acting in $\ell^2(\Gb,M_{\xi})$) of the Markov jump dynamics $\labf_{\xi}$  has eigenvalues $\{-n\colon n=0,1,\ldots\}$. The multiplicity of the $n$th eigenvalue is equal to $\#\Gb_n$.
\end{proposition}
\begin{proof}
  Indeed, recall (\S\ref{sub:space_A}) that the dense subspace $\A\subset\ell^2(\Gb,M_\xi)$ spanned by the relative dimension functions $\Pc_\mu^*$ possesses a filtration. By (\ref{action_of_A_xi}), the action of the operator $\bar{\mathsf{A}}_\xi$ in $\A$ respects this filtration, that is, $\bar{\mathsf{A}}_\xi$ acts triangularly in the basis $\{\Pc_\mu^*\}_{\mu\in\Gb}$ of $\A$. This concludes the proof.
\end{proof}

% subsection generator_of_the_markov_jump_process (end)

\subsection{Lifting of $\slf(2)$ representation and functions $\Ff_\la$} % (fold)
\label{sub:lifting_of_representation}

% Our next step is to provide an explicit description of functions which diagonalize the operator $\bar{\mathsf{A}}_\xi$ in $\ell^2(\Gb,M_\xi)$. We use the $\slf(2,\C)$ structure of the operator ${\mathsf{B}}_\xi$ (isometric to ${\mathsf{A}}_\xi$) given in Proposition \ref{prop:Bxi_action}.

Let us look closer at the formula for the operator ${\mathsf{B}}_\xi$ of Proposition \ref{prop:Bxi_action}. Consider the $2\times2$ matrix
\begin{equation}
  \frac12\frac{1+\xi}{1-\xi}H-\frac{\sqrt\xi}{1-\xi}(U+D)=
  \begin{bmatrix}
    \frac12\frac{1+\xi}{1-\xi}&-\frac{\sqrt\xi}{1-\xi}\\
    \rule{0pt}{12pt}
    \frac{\sqrt\xi}{1-\xi}&-\frac12\frac{1+\xi}{1-\xi}
  \end{bmatrix}=
  \frac12G_{\xi}HG_{\xi}^{-1}
  \in\slf(2,\C)
  \label{matrix_comput}
\end{equation}
(see (\ref{UDH_matrices})), where $G_{\xi}$ is the following element of the real form $SU(1,1)\subset SL(2,\C)$:
\begin{equation*}
  G_{\xi}:=\begin{bmatrix}
    \frac1{\sqrt{1-\xi}}&\frac{\sqrt\xi}{\sqrt{1-\xi}}\\
    \frac{\sqrt\xi}{\sqrt{1-\xi}}&\frac1{\sqrt{1-\xi}}
    \rule{0pt}{12pt}
  \end{bmatrix}.
\end{equation*}

Now recall that the Kerov's operators $(\Uf,\Df,\Hf)$ define a representation of the complex Lie algebra $\slf(2,\C)$ in the (complex) pre-Hilbert space $\ellf^2(\Gb)$. The subalgebra $\suf(1,1)\subset \slf(2,\C)$ is spanned by the matrices $U-D$, $i(U+D)$, and $iH$ ($i=\sqrt{-1}$) Note that the corresponding operators $\Uf-\Df$, $i(\Uf+\Df)$, and $i\Hf$ act skew-symmetrically in $\ellf^2(\Gb)$. 

\begin{proposition}
  \label{prop:integrability}
  All vectors of the space $\ell_{\mathrm{fin}}^2(\Gb)$ are analytic (see \cite{nelson1959analytic}) for the described above action of the Lie algebra $\mathfrak{su}(1,1)$ in $\ell_{\mathrm{fin}}^2(\Gb)$. Thus, this action of $\mathfrak{su}(1,1)$ gives rise to a unitary representation of the universal covering group $SU(1,1)^\sim$ in the Hilbert space $\ell^2(\Gb)$.
\end{proposition}
\begin{proof}
  This fact stems from \cite{Olshanski-fockone} and is established (for the Schur graph) in \cite[\S4.2]{Petrov2010Pfaffian}. The same proof works in our abstract setting.
\end{proof}

Observe that $\{G_{\xi}\}_{\xi\in[0,1)}$ is a continuous curve in $SU(1,1)$ starting at the unity. Let $\{\tilde G_{\xi}\}_{\xi\in[0,1)}$ denote the lifting of that curve to $SU(1,1)^{\sim}$, and by $\G_{\xi}$ for every $\xi\in[0,1)$ denote the corresponding unitary operator in the representation of $SU(1,1)^{\sim}$ in $\ell^2(\Gb)$. It follows from analyticity in the above proposition that the operator $\Hf$ with domain $\ellf^2(\Gb)$ is essentially self-adjoint in $\ell^2(\Gb)$. Let $\bar \Hf$ denote its closure.
\begin{proposition}\label{prop:Bxi_closability}
  The operator ${\mathsf{B}}_{\xi}$ of Proposition \ref{prop:Bxi_action} is essentially self-adjoint in the space $\ellf^2(\Gb)$, and its closure looks as
  \begin{equation}\label{Bxi_Gxi}
    \bar{\mathsf{B}}_{\xi}=
    -\frac12\G_{\xi}\bar\Hf\G_{\xi}^{-1}+\frac\zz2\mathbf{1}.
  \end{equation}
\end{proposition}
\begin{proof}
  This follows from matrix computation (\ref{matrix_comput}) and Proposition \ref{prop:integrability}.
\end{proof}

Now we are in a position to describe a diagonalization of the operator $\bar{\mathsf{B}}_\xi$:
\begin{proposition}
  The functions $\Ff_{\la}:=\G_\xi\un\la\in\ell^2(\Gb)$, $\la\in\Gb$, form an orthonormal basis of eigenfunctions of the operator $\bar {\mathsf{B}}_{\xi}$:
  \begin{equation}\label{action_Bxi_Fla}
    \bar{\mathsf{B}}_{\xi}\Ff_\la=-|\la|\Ff_\la,\qquad\la\in\Gb.
  \end{equation}
\end{proposition}
\begin{proof}
  The fact that $\{\Ff_\la\}_{\la\in\Gb}$ is an orthonormal basis of $\ell^2(\Gb)$ follows from the unitarity of $\G_\xi$, and (\ref{action_Bxi_Fla}) directly follows from (\ref{Bxi_Gxi}) and the action of $\Hf$ (Proposition \ref{prop:action_H}).
\end{proof}

% subsection lifting_of_representation (end)

\subsection{Remark: averages with respect to the process $\labf_{\xi}$} % (fold)
\label{sub:remark_averages_with_respect_to_the_process_}

The unitary operator $\G_\xi$ in $\ell^2(\Gb)$ can be used to describe averages with respect to the mixed measure $M_\xi$ on $\Gb$ and, more generally, to the finite-dimensional distributions of the process $\labf_\xi$ (recall that it starts from the invariant distribution).

Let $f$ be a bounded function on $\Gb$, and also by $f$ let us mean the operator of multiplication by $f$ in $\ell^2(\Gb)$. Then it can be shown (cf. \cite[\S4.2]{Petrov2010Pfaffian}) that
\begin{equation}\label{static_average}
  \langle f,M_\xi\rangle:=
  \sum_{\la\in\Gb}f(\la)M_\xi(\la)=
  \big(
  \G_{\xi}^{-1}f\G_\xi\un\varnothing,\un\varnothing
  \big).
\end{equation}

For dynamical averages, fix time moments $t_1\le\ldots\le t_n$, and by $M_{t_1,\ldots,t_n}$ denote the corresponding finite-dimensional distribution of the process $\labf_\xi$. The measure $M_{t_1,\ldots,t_n}$ lives on $\Gb^n=\Gb\times \ldots\times\Gb$. Let $f_i$ ($i=1,\ldots,n$) be bounded functions on $\Gb$. By $f_1\otimes \ldots\otimes f_n$ denote the corresponding product function on $\Gb^n$. Finally, let $\{V(t)\}_{t\ge0}$ be the semigroup in $\ell^2(\Gb)$ corresponding to the operator $\bar{\mathsf{B}}_\xi$.\footnote{One could first consider the Markov semigroup of the process $\labf_\xi$ (generated by $\bar{\mathsf{A}}_\xi$) which acts in $\ell^2(\Gb,M_\xi)$, and then translate it to the non-weighted space $\ell^2(\Gb)$ by means of the isometry $\istry_\xi$ (\ref{istry}). Informally, $V(t)=\exp(\bar{\mathsf{B}}_\xi t)$.} Then 
\begin{equation}\label{dynamic_average}
  \langle
    f_1\otimes \ldots\otimes f_n,
    M_{t_1,\ldots,t_n}
  \rangle=
  \left(
  \G_{\xi}^{-1}
  f_1V(t_2-t_1)f_2 \ldots V(t_n-t_{n-1})f_n
  \G_\xi\un\varnothing,\un\varnothing
  \right).
\end{equation}
Here $\langle\cdot,\cdot\rangle$ is the pairing of a function with a measure as in (\ref{static_average}).

In \cite{Olshanski-fockone} and \cite{Petrov2010Pfaffian} the above formulas for averages and (\ref{dynamic_average}) were used to establish determinantal/Pfaffian structure of dynamical correlation functions of the processes $\labf_\xi$ for the Young and Schur graphs, respectively. The computations substantially involved Fock space structure of those graphs. One cannot expect such nice structures to appear in a general situation.

% subsection remark_averages_with_respect_to_the_process_ (end)

\subsection{Explicit formula for $\Ff_\la$, definition of $\Mf_\la$} % (fold)
\label{sub:explicit_formula_for_}

Here we express the functions $\Ff_\la\in\ell^2(\Gb)$ which diagonalize the operator $\bar{\mathsf{B}}_\xi$ in terms of the relative dimension functions $\Pc_\mu^*$ on the graph $\Gb$. 
\begin{lemma}\label{lemma:e_xi_U_Pmu}
  For any $\xi\in[0,1)$ and any fixed $\mu\in\Gb$, one has
  \begin{equation*}
    e^{\sqrt\xi\Uf}\un\mu=(1-\xi)^{-\zz/2}\xi^{-|\mu|/2}
    \Big(\prod_{\square\in\mu}\ufunc(\square)^{-1}\Big)(\istry_\xi\Pc_\mu^*).
  \end{equation*}
\end{lemma}
\begin{proof}
  We need to show that for any $\rho\in\Gb$:
  \begin{equation}\label{e_xiU}
    (e^{\sqrt\xi\Uf}\un\mu)(\rho)=(1-\xi)^{-\zz/2}\xi^{-|\mu|/2}
    \Big(\prod_{\square\in\mu}\ufunc(\square)^{-1}\Big)(M_{\xi}(\rho))^{1/2}\Pc_\mu^*(\rho).
  \end{equation}
  Observe that both sides are analytic in $\xi\in\C$, $|\xi|<1$, so it suffices to establish (\ref{e_xiU}) for small $\xi$. By Proposition \ref{prop:integrability}, the vector $\un\mu\in\ellf^2(\Gb)$ is analytic for the action of $\suf(1,1)$. Thus on this vector the representation of $SU(1,1)^{\sim}$ of Proposition \ref{prop:integrability} can be extended to a representation of the local complexification of the group $SU(1,1)^{\sim}$ (see, e.g., the beginning of \S7 in \cite{nelson1959analytic}). For small $\xi$, the $2\times 2$ matrix $e^{\sqrt\xi U}$ is close to the unity of the group $SL(2,\C)$. Therefore, for small $\xi$ the action of the operator $e^{\sqrt\xi\Uf}$ on $\un\mu$ is given by the corresponding series for the exponent.

  Assume that $\mu\subseteq\rho$, otherwise $(e^{\sqrt\xi\Uf}\un\mu)(\rho)=0$. With the above explanations, we can write
  \begin{align*}
    (e^{\sqrt\xi\Uf}\un\mu)(\rho)&=
    \frac{\xi^{\frac{|\rho|-|\mu|}2}}{(|\rho|-|\mu|)!}\dim(\mu,\rho)
    \prod_{\square\in\rho/\mu}\ufunc(\square)
    \\&=(M_\xi(\rho))^{\frac12}\cdot(1-\xi)^{-\zz/2}\xi^{-|\mu|/2}
    \Big(\prod_{\square\in\mu}\ufunc(\square)^{-1}\Big)\Pc_{\mu}^{*}(\rho).
  \end{align*}
  This concludes the proof.
\end{proof}

Now we are in a position to write an explicit formula for the functions $\Ff_\la$:
\begin{proposition}\label{prop:F_la}
  Let $\la\in\Gb$, $|\la|=n$. The function $\Ff_\la$ has the following form:
  \begin{align*}
    \Ff_\la&=
    \sum_{\mu\colon\mu\subseteq\la}
    \frac{(-1)^{n-m}\xi^{n/2-m}(1-\xi)^{m}}{(n-m)!}
    \dim(\mu,\la)
    \times\\&\qquad \qquad \qquad \qquad\times
    \Big(\prod_{\square\in\la/\mu}\ufunc(\square)^{2}\Big)
    \Big(\prod_{\square\in\la}\ufunc(\square)^{-1}\Big)
    (\istry_\xi\Pc_\mu^*),
  \end{align*}
  where in the sum we have denoted $m:=|\mu|$.
\end{proposition}
\begin{proof}
  We argue as in the previous lemma. It suffices to establish the identity for small $\xi$ because both parts are analytic in $\xi$, $|\xi|<1$. For small $\xi$ the operator $\G_\xi$ (which is close to the unity of the group $SU(1,1)^{\sim}$) can be written as
  \begin{equation*}
    \G_{\xi}=\exp(\sqrt\xi\Uf)\exp\Big(-\frac{\sqrt\xi}{1-\xi}\Df\Big)
    (1-\xi)^{\Hf/2}
  \end{equation*}
  (mirroring the corresponding relation between $2\times 2$ matrices). Since $\Ff_\la=\G_\xi\un\la$ and the vector $\un\la$ is analytic for the action of $\suf(1,1)$ in $\ellf^2(\Gb)$, we can use the above formula to compute $\Ff_\la$. Set $\eta:=-\frac{\sqrt\xi}{1-\xi}$. Let us write (recall that $|\la|=n$):
  \begin{align*}
    \Ff_\la&=(1-\xi)^{n+\zz/2}
    e^{\sqrt\xi \Uf}e^{\eta\Df}\un\la\\&=
    (1-\xi)^{n+\zz/2}
    \sum_{m=0}^{n}\sum_{\mu\subseteq\la,\,|\mu|=m}
    \frac{\eta^{n-m}}{(n-m)!}
    \dim(\mu,\la)
    \Big(
    \prod_{\square\in\la/\mu}\ufunc(\square)
    \Big)e^{\sqrt\xi \Uf}\un\mu.
  \end{align*}
  Plugging in the formula of Lemma \ref{lemma:e_xi_U_Pmu} and simplifying, we conclude the proof.
\end{proof}

The functions $\Ff_\la$ form an orthonormal basis of $\ell^2(\Gb)$ and diagonalize the operator $\bar{\mathsf{B}}_\xi$. Using the isometry $\istry_\xi$ (\ref{istry}), we see that the functions $\istry^{-1}_\xi\Ff_\la$ form an orthonormal basis in $\ell^2(\Gb,M_{\xi})$, and diagonalize the Markov generator $\bar{\mathsf{A}}_\xi$ of the dynamics $\labf_\xi$. From Proposition \ref{prop:F_la} we see that $\istry^{-1}_\xi\Ff_\la$'s belong to the space $\A$ (\S\ref{sub:space_A}). Moreover, with respect to the filtration of $\A$ these functions have the following decomposition:
\begin{equation*}
  \istry^{-1}_\xi\Ff_\la=\left(\frac{\sqrt\xi}{1-\xi}\right)^{-|\la|}
  \Big(\prod_{\square\in\la}\ufunc(\square)^{-1}\Big)\Pc_\la^*+\text{lower degree terms}.
\end{equation*}
Let us consider ``monic'' versions of the above functions, i.e., their multiples that have coefficient $1$ by the top degree term $\Pc_\la^*$:
\begin{definition}\label{def:Mf_la}
  For any $\la\in\Gb$, let $\Mf_\la\in\A\subset\ell^2(\Gb,M_\xi)$ be the following function:
  \begin{equation*}
    \Mf_\la:=
    \sum_{\mu\subseteq\la}
    \left(\frac\xi{\xi-1}\right)^{|\la|-|\mu|}
    \frac{\dim(\mu,\la)}{(|\la|-|\mu|)!}
    \Big(\prod_{\square\in\la/\mu}\ufunc(\square)^{2}\Big)\Pc_\mu^*.
  \end{equation*}
\end{definition}

\begin{remark}
  The functions $\Mf_\la$ essentially arise as matrix elements of the $\slf(2,\C)$-module spanned by the Kerov's operators. When the graph $\Gb$ is a chain (and the Kerov's operators span an irreducible lowest weight module, see \S \ref{sub:lowest_weight_slf_2_c_modules}), the functions $\Mf_\la$ reduce to the classical Meixner orthogonal polynomials $\meix_n^{\zz,\xi}$ (see \S \ref{sub:birth_and_death_processes_related_to_the_meixner_polynomials}). The fact that the Meixner polynomials can be constructed from irreducible lowest weight $\slf(2,\C)$-modules is known \cite{Koornwinder1982Krawtchouk}, \cite{Vilenkin-Klimyk-DAN_UKR_1988}, \cite{Vilenkin-Klimyk-ITOGI1995-en}. Thus, our definition of the functions $\Mf_\la$ for any graph of ideals generalizes this approach when an irreducible module is replaced by the one spanned by the Kerov's operators. 
\end{remark}

% subsection explicit_formula_for_ (end)

\subsection{Properties of the functions $\Mf_\la$} % (fold)
\label{sub:properties_of_the_functions_}

Here we summarize properties of the functions $\Mf_\la$.

{\bf1.\/} One has $\Mf_\varnothing\equiv1$.

{\bf2.\/} The functions $\{\Mf_\la\}_{\la\in\Gb}$ diagonalize the generator $\bar{\mathsf{A}}_\xi$ of the jump dynamics $\labf_\xi$ on the graph $\Gb$ (\S\ref{sub:generator_of_the_markov_jump_process}):
\begin{equation}\label{Mf_diag}
  \bar{\mathsf{A}}_\xi\Mf_\la=-|\la|\Mf_\la,\qquad\la\in\Gb.
\end{equation}

{\bf3.\/} The functions $\{\Mf_\la\}_{\la\in\Gb}$ form an orthogonal basis in $\ell^2(\Gb,M_\xi)$:
\begin{equation}\label{Mf_orthog}
  \left(\Mf_\la,\Mf_\mu\right)_{M_{\xi}}=\delta_{\la,\mu}
  \cdot
  \frac{\xi^{|\la|}}{(1-\xi)^{2|\la|}}\prod_{\square\in\la}\ufunc(\square)^{2}.
\end{equation}

\begin{remark}
  The first property (\ref{Mf_diag}) can also be verified by a straightforward computation using Definition \ref{def:Mf_la} and Corollary \ref{corollary:action_of_Axi_on_Pmu}. However, on a abstract level it seems necessary to use the unitary operator $\G_\xi$ to establish the second property (\ref{Mf_orthog}). For the Young graph, the orthogonality (\ref{Mf_orthog}) is shown in \cite{Olshanski2011Meixner} by a completely different argument involving an analytic continuation.
\end{remark}

{\bf4.\/} (Characterization of the eigenfunctions $\Mf_\la$) Let $\{F_\la\}_{\la\in\Gb}$ be a family of functions in $\A\subset\ell^2(\Gb,M_\xi)$ such that 

\par\noindent
(1) $F_\la=\Pc_\la^*+\text{lower degree terms}$ (with respect to the filtration of $\A$, see \S\ref{sub:space_A}).

\par\noindent
(2) $\bar{\mathsf{A}}_\xi F_\la=-|\la|F_\la$, $\la\in\Gb$.

Then $F_\la=\Mf_\la$ for all $\la\in\Gb$.

\begin{proof}
  Let $|\la|=n$, then $F_\la,\Mf_\la\in\A_n$. The operator $\bar{\mathsf{A}}_{\xi}$ preserves the (finite-dimensional) space $\A_n$ (Corollary \ref{corollary:action_of_Axi_on_Pmu}); moreover, the eigenstructure of this operator in $\A_n$ is $\{-m\colon m=0,1,2,\dots,n\}$ (Proposition \ref{prop:spectrum_of_A}). Consider the function $F_\la-\Mf_\la$. Because the top terms of $F_\la$ and $\Mf_\la$ coincide, the difference $F_\la-\Mf_\la$ belongs to $\A_{n-1}$. On the other hand, $F_\la-\Mf_\la$ is an eigenfunction of $\bar{\mathsf{A}}_\xi$ with eigenvalue $(-n)$. Since there is no such eigenvalue of $\bar{\mathsf{A}}_\xi$ in $\A_{n-1}$, it follows that $F_\la=\Mf_\la$.
\end{proof}

{\bf5.\/} (Autoduality; cf. \cite[Prop. 4.28]{Olshanski2011Meixner}) 
Let the functions $\Mf_\la'$ be defined by
\begin{equation*}
  \Mf_\la=
  (-1)^{|\la|}
  \left(\frac{\xi}{1-\xi}\right)^{|\la|}
  \frac{\dim\la}{|\la|!}
  \left(\prod_{\square\in\la}\ufunc(\square)^2\right)
  \Mf_\la',\qquad\la\in\Gb.
\end{equation*}
Then the following autoduality property holds:
\begin{equation*}
  \Mf_\la'(\rho)=\Mf_\rho'(\la),\qquad\mbox{for all $\la,\rho\in\Gb$}.
\end{equation*}

{\bf6.\/} (Moment functional; cf. \cite[Prop. 5.2]{Olshanski2011Meixner})
Let $\phi\colon\A\to\C$, $\phi(f):=\langle f,M_\xi\rangle$ be the so-called moment functional on the space $\A$. Clearly, 
\begin{equation}\label{phi_M_la}
  \phi(\Mf_\la)=(\Mf_\la,\Mf_\varnothing)_{M_\xi}=
  \delta_{\la\varnothing},\qquad\la\in\Gb.
\end{equation}
We can compute the action of $\phi$ on the relative dimension functions $\Pc_\la^*$:
\begin{equation}\label{phi_P_la}
  \phi(\Pc_\la^*)=\left(\frac{\xi}{1-\xi}\right)^{|\la|}
  \left(\prod_{\square\in\la}\ufunc(\square)^2\right)
  \frac{\dim\la}{|\la|!},\qquad \la\in\Gb.
\end{equation}
Note that the functional $\phi$ on $\A$ is completely determined by either of the conditions (\ref{phi_M_la}) and (\ref{phi_P_la}).
\begin{proof}
  The statement clearly holds for $\la=\varnothing$, because $\Pc_\varnothing^*\equiv1$.
  
  Now let $\la\ne\varnothing$. Since $\phi$ vanishes at the range of $\bar{\mathsf{A}}_\xi$, we have
  \begin{equation*}
    \phi(\Pc_\la^*)=\frac1{|\la|}
    \frac{\xi}{1-\xi}
    \sum_{\rho\colon\rho\nearrow\la}
    \km(\rho,\la)\ufunc(\la/\rho)^2\phi(\Pc_\rho^*).
  \end{equation*}
  Iterating this identity, we see that the statement holds.
\end{proof}

% subsection properties_of_the_functions_ (end)

\subsection{Pascal triangle and birth and death processes related to the Meixner polynomials} % (fold)
\label{sub:birth_and_death_processes_related_to_the_meixner_polynomials}

For our usual example, the Pascal triangle $\PT$, the dynamics factorizes into two independent stochastic processes each living on one copy of $J(\Z_{>0})$ in accordance with Proposition \ref{prop:KO_disjoint}. 

It is enough to consider the case of the chain $\Gb=\Z_{\ge0}=(J(\Z_{>0}),1)$. The jump dynamics for the chain is exactly the process $\n_{\zz,\xi}$ described in \S \ref{sub:markov_processes_on_}. The functions $\{\Mf_n\}_{n\in\Z_{\ge0}}$ (Definition \ref{def:Mf_la}) become
\begin{equation*}
  \Mf_n(x)=\sum_{m=0}^{n}
  \left(\frac\xi{\xi-1}\right)^{n-m}\binom{n}{m}
  \frac{\Gamma(n+\zz)}{\Gamma(m+\zz)}x^{\da m},\qquad x\in\Z_{\ge0}.
\end{equation*}
Let us denote the above functions by $\meix_{n}^{\zz,\xi}(x)$. These are the classical monic (i.e., having the coefficient $1$ by the top degree term $x^n$) Meixner orthogonal polynomials \cite[\S1.9]{Koekoek1996}. The orthogonality measure for the Meixner polynomials is precisely the negative binomial distribution $\pi_{\zz,\xi}$. 

For the Pascal triangle, the orthogonal eigenfunctions of the dynamics will have the form $\Mf_{(k,l)}(x,y)=\meix_k^{\zz_1,\xi}(x)\meix_l^{\zz_2,\xi}(y)$, where $(k,l),(x,y)\in\PT$.

\begin{remark}
  The case of a finite chain $\Gb=\{0,1,\ldots,N\}$ (which has only one triplet of Kerov's operators, see \S \ref{sub:lowest_weight_slf_2_c_modules}) leads to the Krawtchouk orthogonal polynomials \cite[\S1.10]{Koekoek1996}.
\end{remark}

Let us briefly discuss a limiting behavior of the process $\n_{\zz,\xi}$. As $\xi\to1$, under the space scaling $\Z\mapsto (1-\xi)\Z\subset\R$, the process $\n_{\zz,\xi}$ converges to the one-dimensional diffusion processes with the generator
\begin{equation}\label{1dim_Laguerre}
  r\frac{d^2}{dr^2}+(\zz-r)\frac{d}{dr},\qquad r\in\R_{\ge0}.
\end{equation}
This diffusion preserves the gamma distribution with density ${r^{\zz-1}e^{-r}}{(\Gamma(\zz))^{-1}}dr$ (which in turn is the limit of the negative binomial distributions $\pi_{\zz,\xi}$ as $\xi\to1$), and is reversible with respect to it. The above generator is diagonalized in the Laguerre polynomials \cite[\S1.11]{Koekoek1996}. About this process (which is closely related to a squared Bessel process) see, e.g., \cite{Eie83}.

% subsection birth_and_death_processes_related_to_the_meixner_polynomials (end)

% section markov_jump_dynamics (end)

\section{Remark about Heisenberg algebra operators and Plancherel measures} % (fold)
\label{sec:heisenberg_operators}

\subsection{General setting} % (fold)

There exists a variation of the setting of the present paper when the algebra $\slf(2,\C)$ is replaced by the Heisenberg algebra. Namely, assume that we are given a branching graph $\Gb$ (not necessary of ideals) which is $r$-self-dual in the sense of \cite{fomin1994duality}, where $r>0$ is some real constant. This means that the operators $\Uf^\circ$ and $\Df^\circ$ on the graph defined in Remark \ref{rmk:Plancherel_operators} (they are constructed using only edge multiplicities) satisfy the commutation relation $[\Df^\circ,\Uf^\circ]=r\mathbf{1}$. Clearly, the operators $(\Uf^\circ,\Df^\circ,r\mathbf{1})$ define a representation of the 3-dimensional Heisenberg algebra in the space $\ellf^2(\Gb)$. In many examples of branching graphs considered in \S \ref{sec:examples}, the operators $(\Uf^\circ,\Df^\circ,r\mathbf{1})$ arise as degenerations of the Kerov's operators on $\Gb$.  Let us indicate which probabilistic models arise from the Heisenberg algebra structure. The proofs here essentially repeat the ones of the previous subsections.

Having operators $(\Uf^\circ,\Df^\circ,r\mathbf{1})$ as above, we can define the \emph{Plancherel measures} on the floors of the graph:
\begin{equation*}
  \Pl_n(\la):=\frac1{Z_n}
  \big((\Uf^{\circ})^{n}\uno\varnothing,\uno\la\big)
  \big((\Df^{\circ})^{n}\uno\la,\uno\varnothing\big)
  =\frac{(\dim\la)^2}{r^nn!},\qquad \la\in\Gb_n.
\end{equation*}
The measures $\{\Pl_n\}$ form a coherent system on $\Gb$.

As in \S \ref{sub:mixing_of_measures}, we mix the Plancherel measures $\Pl_n$, but now as a mixing distribution we choose the Poisson distribution $\{e^{-\ga r}\frac{(\ga r)^{n}}{n!}\}_{n\in\Z_{\ge0}}$ with parameter $\ga r$ (here $\ga>0$ is a new parameter). We arrive at the \emph{poissonized Plancherel measures}
\begin{equation*}
  \Pl_\ga(\la):=e^{-\ga r}{\ga^{|\la|}}
  \left(\frac{\dim\la}{|\la|!}\right)^{2},\qquad\la\in\Gb.
\end{equation*}

It is possible to define dynamics on the whole graph $\Gb$ preserving the measure $\Pl_\ga$. This dynamics for the Young graph appeared in \cite{PhahoferSpohn2002}, \cite{borodin2006stochastic}. The Markov process here is defined exactly as in \S \ref{sub:markov_processes_on_}, but the underlying process $\n_{\zz,\xi}$ should be replaced by the process with jump rates ($n=0,1,\ldots$):
\begin{equation*}
  q_{n,n-1}=n,\qquad
  q_{n,n+1}=\ga r,\qquad
  q_{n,n}=-n-\ga r.
\end{equation*}
All other jump rates are zero. 

The generator $\bar {\mathsf{A}}_\ga^{\mathrm{Poiss}}$ of the dynamics on the graph $\Gb$ acts on the relative dimension functions (\S \ref{sec:relative_dimension_functions}) as follows:
\begin{equation*}
  \bar {\mathsf{A}}_\ga^{\mathrm{Poiss}}\Pc_\mu^*=-|\mu|\Pc_\mu^*+\ga r\sum\nolimits_{\varkappa\colon
  \varkappa\nearrow\mu}\km(\varkappa,\mu)\Pc_\varkappa^*,\qquad \mu\in\Gb.
\end{equation*}

\begin{definition}
  Consider the following functions on the graph $\Gb$ which belong to the space $\A$ (\S \ref{sub:space_A}):
  \begin{equation}\label{Cf_la}
    \Cf_\la:=
    \sum_{\mu\subseteq\la}
    \left(-\ga r\right)^{|\la|-|\mu|}
    \frac{\dim(\mu,\la)}{(|\la|-|\mu|)!}\Pc_\mu^*,\qquad\la\in\Gb.
  \end{equation}
\end{definition}

\begin{proposition}
  {\rm\bf1.} The functions $\Cf_\la$ form an orthogonal basis in the Hilbert space $\ell^2(\Gb,\Pl_\ga)$, and $(\Cf_\la,\Cf_\mu)_{\Pl_\ga}=\delta_{\mu,\la}(\ga r)^{|\la|}$ for all $\la,\mu\in\Gb$.

  {\rm\bf2.} The functions $\Cf_\la$ are eigenfunctions of the operator $\bar {\mathsf{A}}_\ga^{\mathrm{Poiss}}$:
  \begin{equation*}
    \bar {\mathsf{A}}_\ga^{\mathrm{Poiss}}\Cf_\la=-|\la|\Cf_\la,\qquad\la\in\Gb.
  \end{equation*}
\end{proposition}

We note that the averages of the relative dimension functions with respect to the poissonized Plancherel measure $\Pl_\ga$ are given by 
\begin{equation*}
  \phi^{\mathrm{Poiss}}(\Pc_\mu^*):=
  \langle\Pc_\mu^*,\Pl_\ga\rangle=(\ga r)^{|\mu|}\frac{\dim\mu}{|\mu|},\qquad\mu\in\Gb.
\end{equation*}

\subsection{Charlier orthogonal polynomials} % (fold)

In the case of the chain, one must take the multiplicity function $\km(n,n+1)=\sqrt{n+1}$ (where $n=0,1,\dots$) in order for the operators $\Uf^\circ$ and $\Df^\circ$ to satisfy $[\Df^\circ,\Uf^\circ]=\mathbf{1}$ (with $r=1$). Of course, this multiplicity function is equivalent to the trivial one. For the chain $\Gb=(J(\Z_{>0}),\km)$, the relative dimension functions are $\Pc_k^*(x)=x^{\da k}/\sqrt{k!}$, and the functions $\Cf_n$ ($n\in\Z_{\ge0}$) are multiples of the classical Charlier orthogonal polynomials \cite[\S1.12]{Koekoek1996}:
\begin{equation*}
  \Cf_n(x)=\frac1{\sqrt{n!}}
  \sum_{m=0}^{n}
  \left(-\ga \right)^{|\la|-|\mu|}
  \binom nm x^{\da m},\qquad n=0,1,\ldots.
\end{equation*}
These polynomials $\Cf_n$ are orthogonal with respect to the Poisson measure on $\Z_{\ge0}$ with parameter $\ga>0$. Note that they are monic not in the usual sense, but with respect to the relative dimension functions $\Pc_k^*$.

\subsection{References}
The Plancherel measures on partitions (i.e., when $\Gb$ is the Young graph) were studied in several papers including \cite{baik1999distribution}, \cite{okounkov2000random}, \cite{Borodin2000b}, \cite{okounkov2001infinite}, \cite{Okounkov2002}. Plancherel measures on other graphs were studied in, e.g., \cite{matsumoto2008jack} (Plan\-che\-rel measures on partitions with Jack parameter for the graph described in \S \ref{sub:measures_with_jack_parameter}; they are related to random matrix $\beta$-ensembles), \cite{Tracy2004}, \cite{Matsumoto2005}, \cite{Petrov2010} (the Schur graph, see \S \ref{sub:schur_graph}).

% section heisenberg_operators (end)

\section{Applications of general formalism to concrete branching graphs} % (fold)
\label{sec:examples}

In this section we discuss various concrete examples of branching graphs to which the formalism of the present paper is applicable. By $\Yb$ we will denote both the set of all partitions (identified with the Young diagrams \cite[Ch. I, \S1]{Macdonald1995}), and the Young graph $J(\Z_{>0}^{2})$ with simple edges. We start with branching graphs having $\Yb$ as the set of vertices, and after that consider more examples.

\subsection{Macdonald $(q,t)$ edge multiplicities} % (fold)
\label{sub:macdonald}

Let $\La$ be the algebra of symmetric functions \cite[Ch. I, \S2]{Macdonald1995}. A distinguished linear basis of $\La$ is formed by the Schur functions $s_\la$ \cite[Ch. I, \S3]{Macdonald1995} which are indexed by all partitions $\la\in\Yb$. I.G. Macdonald introduced a two-parameter deformation of the Schur functions --- the Macdonald symmetric functions $P_\la(x;q,t)$ \cite[Ch. VI, \S4]{Macdonald1995}. For most values of the parameters $(q,t)$ these functions also form a linear basis of $\La$.

There is a \emph{Pieri rule} which describes the branching of the Macdonald symmetric functions \cite[Ch. VI, (6.24.iv)]{Macdonald1995}:
\begin{equation}\label{qt_branching}
  p_1(x)P_\la(x;q,t)=\sum\nolimits_{\nu\colon\nu\searrow\la}\km_{q,t}(\la,\nu)
  P_\nu(x;q,t)\qquad \mbox{for all $\la\in\Yb$}.
\end{equation}
Here $p_1(x)=\sum_{i=1}^{\infty}x_i$ is the first Newton power sum, and $\km_{q,t}(\la,\nu)$ are rational functions in $q,t$ defined as
\begin{equation}\label{kqt}
  \km_{q,t}(\la,\nu)=\prod_{\text{$\square$ above $\nu/\la$}}
  F_{q,t}(a(\square),l(\square)),
  \quad
  F_{q,t}(a,l):=
  \frac{(1-q^{a}t^{l+2})
  (1-q^{a+1}t^{l})}
  {(1-q^{a}t^{l+1})(1-q^{a+1}t^{l+1})}.
\end{equation}
Here the product is taken over all boxes $\square$ in the $\j$th column of the Young diagram $\la$ containing the new box $\nu/\la$. The number $a(\square)=a(\i,\j)=\la_\i-\j$ is the \emph{arm length}, and $l(\square)=l(\i,\j)=\la_\j'-\i$ is the \emph{leg length} of a box $\square=(\i,\j)$ with row and column numbers $\i$ and $\j$, respectively. In other words, $a(\square)$ and $l(\square)$ are the horizontal and the vertical distances from $\square$ to the boundary of the smaller Young diagram $\la$, respectively:
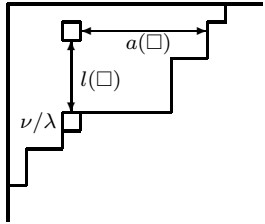
\begin{figure}[htpb]
  \begin{center}
    \setlength{\unitlength}{3000sp}%
    \begingroup\makeatletter\ifx\SetFigFont\undefined%
    \gdef\SetFigFont#1#2#3#4#5{%
      \reset@font\fontsize{#1}{#2pt}%
      \fontfamily{#3}\fontseries{#4}\fontshape{#5}%
      \selectfont}%
    \fi\endgroup%
    \begin{picture}(2144,1444)(579,-1583)
    \put(1201,-736){\makebox(0,0)[lb]{\smash{{\SetFigFont{12}{14.4}{\rmdefault}{\mddefault}{\itdefault}{\color[rgb]{0,0,0}$\scriptstyle{}l(\square)$}%
    }}}}
    \thicklines
    {\color[rgb]{0,0,0}\put(601,-1561){\line( 1, 0){150}}
    \put(751,-1561){\line( 0, 1){300}}
    \put(751,-1261){\line( 1, 0){300}}
    \put(1051,-1261){\line( 0, 1){300}}
    \put(1051,-961){\line( 1, 0){900}}
    \put(1951,-961){\line( 0, 1){450}}
    \put(1951,-511){\line( 1, 0){300}}
    \put(2251,-511){\line( 0, 1){300}}
    \put(2251,-211){\line( 1, 0){150}}
    \put(2401,-211){\line( 0, 1){150}}
    }%
    {\color[rgb]{0,0,0}\put(1201,-961){\line( 0,-1){150}}
    \put(1201,-1111){\line(-1, 0){150}}
    \put(1051,-1111){\line( 0, 1){150}}
    \put(1051,-961){\line( 1, 0){150}}
    }%
    {\color[rgb]{0,0,0}\put(1051,-211){\line( 1, 0){150}}
    \put(1201,-211){\line( 0,-1){150}}
    \put(1201,-361){\line(-1, 0){150}}
    \put(1051,-361){\line( 0, 1){150}}
    }%
    \thinlines
    {\color[rgb]{0,0,0}\put(1201,-286){\vector(-1, 0){  0}}
    \put(1201,-286){\vector( 1, 0){1050}}
    }%
    {\color[rgb]{0,0,0}\put(1126,-361){\vector( 0, 1){  0}}
    \put(1126,-361){\vector( 0,-1){600}}
    }%
    \put(1576,-436){\makebox(0,0)[lb]{\smash{{\SetFigFont{12}{14.4}{\rmdefault}{\mddefault}{\itdefault}{\color[rgb]{0,0,0}$\scriptstyle{}a(\square)$}%
    }}}}
    \put(701,-1086){\makebox(0,0)[lb]{\smash{{\SetFigFont{12}{14.4}{\rmdefault}{\mddefault}{\itdefault}{\color[rgb]{0,0,0}$\scriptstyle{}\nu/\la$}%
    }}}}
    \thicklines
    {\color[rgb]{0,0,0}\put(2701,-61){\line(-1, 0){2100}}
    \put(601,-61){\line( 0,-1){1800}}
    }%
    \end{picture}%    
    \label{fig:arm_leg}
  \end{center}  
\end{figure}

For $-1<q,t<1$, the numbers $\km_{q,t}(\la,\nu)$ are positive for any $\la\nearrow\nu$, and thus they define an edge multiplicity function on the Young graph. In this way one arrives at a branching graph of ideals $(\Yb,\km_{q,t})$. This graph was considered by Kerov \cite{Kerov-book}. It is also a multiplicative graph in the sense of \cite{Borodin2000} and \cite{VK1984Kfunctor}, \cite{Kerov1990} corresponding to the algebra $\La$ with the basis of Macdonald symmetric functions.

% subsection macdonald (end)

\subsection{Duality and $(q,t)$-Plancherel measures} % (fold)
\label{sub:duality_of_qt}

We introduce another edge multiplicities equivalent to $\km_{q,t}$ (see Definition \ref{def:equivalent_multilpicities}). Let $\la\nearrow \nu$ be two Young diagrams. By definition, put
\begin{equation}\label{kqt_star}
  \km_{q,t}^{*}(\la,\nu):=
  \prod_{\text{$\square$ below $\nu/\la$}}
  F_{q,t}(a(\square),l(\square)),
\end{equation}
where $F_{q,t}$ is given in (\ref{kqt}), the product is taken over all boxes in the complement $\Z_{>0}^{2}\setminus \nu$ lying in the same column as the new box $\nu/\la$, and $a(\square)$ and $l(\square)$ are the horizontal and the vertical distances from $\square$ to the boundary of the larger Young diagram $\nu$, respectively (see Figure \ref{fig:arm_leg_star}). Though the product in (\ref{kqt_star}) is infinite, one readily sees that it telescopes, and therefore $\km^*_{q,t}(\la,\nu)$ actually is also a rational function in $q$ and $t$:
\begin{figure}[htpb]
  \begin{center}
    \setlength{\unitlength}{3000sp}%
    \begingroup\makeatletter\ifx\SetFigFont\undefined%
    \gdef\SetFigFont#1#2#3#4#5{%
      \reset@font\fontsize{#1}{#2pt}%
      \fontfamily{#3}\fontseries{#4}\fontshape{#5}%
      \selectfont}%
    \fi\endgroup%
    \begin{picture}(2144,1444)(579,-1583)
    \put(700,-1711){\makebox(0,0)[lb]{\smash{{\SetFigFont{12}{14.4}{\rmdefault}{\mddefault}{\itdefault}{\color[rgb]{0,0,0}$\scriptstyle{}a(\square)$}%
    }}}}
    \thicklines
    {\color[rgb]{0,0,0}\put(601,-1561){\line( 1, 0){150}}
    \put(751,-1561){\line( 0, 1){300}}
    \put(751,-1261){\line( 1, 0){300}}
    \put(1051,-1261){\line( 0, 1){300}}
    \put(1051,-961){\line( 1, 0){900}}
    \put(1951,-961){\line( 0, 1){450}}
    \put(1951,-511){\line( 1, 0){300}}
    \put(2251,-511){\line( 0, 1){300}}
    \put(2251,-211){\line( 1, 0){150}}
    \put(2401,-211){\line( 0, 1){150}}
    }%
    {\color[rgb]{0,0,0}\put(1201,-961){\line( 0,-1){150}}
    \put(1201,-1111){\line(-1, 0){150}}
    \put(1051,-1111){\line( 0, 1){150}}
    \put(1051,-961){\line( 1, 0){150}}
    }%
    \thinlines
    {\color[rgb]{0,0,0}\put(1126,-1111){\vector( 0, 1){  0}}
    \put(1126,-1111){\vector( 0,-1){300}}
    }%
    {\color[rgb]{0,0,0}\put(1051,-1486){\vector( 1, 0){  0}}
    \put(1051,-1486){\vector(-1, 0){300}}
    }%
    \thicklines
    {\color[rgb]{0,0,0}\put(1051,-1561){\framebox(150,150){}}
    }%
    \put(701,-1066){\makebox(0,0)[lb]{\smash{{\SetFigFont{12}{14.4}{\rmdefault}{\mddefault}{\itdefault}{\color[rgb]{0,0,0}$\scriptstyle{}\nu/\la$}%
    }}}}
    \put(1201,-1336){\makebox(0,0)[lb]{\smash{{\SetFigFont{12}{14.4}{\rmdefault}{\mddefault}{\itdefault}{\color[rgb]{0,0,0}$\scriptstyle{}l(\square)$}%
    }}}}
    {\color[rgb]{0,0,0}\put(2701,-61){\line(-1, 0){2100}}
    \put(601,-61){\line( 0,-1){1800}}
    }%
    \end{picture}%
    \label{fig:arm_leg_star}
  \end{center}  
\end{figure}
\begin{proposition}
  The edge multiplicity functions $\km_{q,t}$ (\ref{kqt}) and $\km_{q,t}^{*}$ (\ref{kqt_star}) are UD-dual in the sense of Definition \ref{def:UD-dual}.
\end{proposition}
\begin{proof}
  A straightforward verification.
\end{proof}

Thus, by Proposition \ref{prop:duality_and_gauge}, the multiplicity functions $\km_{q,t}$ and $\km_{q,t}^{*}$ are equivalent to each other (Definition \ref{def:equivalent_multilpicities}). We define the UD-self-dual multiplicity function $\tilde \km_{q,t}$ equivalent to both $\km_{q,t}$ and $\km_{q,t}^{*}$ as follows:
\begin{equation*}
  \tilde\km_{q,t}(\la,\nu):=\sqrt{
  \km_{q,t}(\la,\nu)\km^*_{q,t}(\la,\nu)},
  \qquad \mbox{for all $\la\nearrow\nu$}.
\end{equation*}

Let $\Uf^{\circ}$ and $\Df^{\circ}$ be the operators constructed from the UD-self-dual multiplicity function $\tilde\km_{q,t}$ as in Remark \ref{rmk:Plancherel_operators}. The self-duality means that the commutator $[\Df^{\circ},\Uf^{\circ}]$ acts diagonally. In fact, much more can be said:
\begin{theorem}\label{thm:qt_Plancherel}
  One has
  \begin{equation}\label{qt_commut}
    [\Df^{\circ},\Uf^{\circ}]=\frac{1-q}{1-t}\mathbf{1}.
  \end{equation}
\end{theorem}
\begin{proof}
  First, it can be checked that the multiplicity function $\tilde\km_{q,t}$ can be written as $\tilde\km_{q,t}(\la,\nu)=\km_{q,t}\sqrt{\frac{b_\la(q,t)}{b_\nu(q,t)}}$, where $b_\mu(q,t)$, $\mu\in\Yb$, is defined in \cite[Ch. VI, (6.19)]{Macdonald1995}. Let $\qf{\cdot,\cdot}=\qf{\cdot,\cdot}_{q,t}$ denote the scalar product in the algebra $\La$ which is associated with the Macdonald symmetric functions \cite[Ch. VI, \S1]{Macdonald1995}. Let $\tilde P_\mu(x;q,t):=\sqrt{b_\mu(q,t)}P_\mu(x;q,t)$. It follows from \cite[Ch. VI, \S4]{Macdonald1995} that these functions form a self-dual basis in $\La$, i.e.,$\langle{\tilde P_\mu,\tilde P_\la}\rangle=\delta_{\mu,\la}$ for all $\mu,\la\in\Yb$. Moreover, the operator of multiplication by $p_1$ has the form (see (\ref{qt_branching})):
  \begin{equation}\label{p1_tildePmu}
    p_1(x)\tilde P_\mu (x;q,t)=\sum\nolimits_{\la\colon\la\searrow\mu}\tilde\km_{q,t}(\mu,\la)
    \tilde P_\la(x;q,t)\qquad \mbox{for all $\mu\in\Yb$}.
  \end{equation}
  Let $p_1^{\perp}$ denote the operator which is adjoint (with respect to $\qf{\cdot,\cdot}$) to the operator of multiplication by $p_1$. Clearly,
  \begin{equation}\label{p1_perp_tildePmu}
    p_1^{\perp}\tilde P_\mu(x;q,t)=
    \sum\nolimits_{\varkappa\colon\varkappa\nearrow\mu}\tilde\km_{q,t}(\varkappa,\mu)
    \tilde P_\varkappa(x;q,t)\qquad \mbox{for all $\mu\in\Yb$}.
  \end{equation}
  Moreover, it can be shown similarly to \cite[Ch.~I, \S5, Ex.~3(c)]{Macdonald1995} that this operator acts as the formal differential operator
  \begin{equation*}
    p_1^{\perp}=\frac{1-q}{1-t}\frac{\partial}{\partial p_1}
  \end{equation*}
  in the polynomial algebra $\La=\R[p_1,p_2,p_3,\dots]$. Thus, $[p_1^{\perp},p_1]=p_1^{\perp}p_1-p_1p_1^{\perp}=\frac{1-q}{1-t}\mathbf{1}$. Comparing this with (\ref{p1_tildePmu}) and (\ref{p1_perp_tildePmu}), we readily see that the claim follows.
\end{proof}

We see that the fact that $[\Df^\circ,\Uf^\circ]$ is a diagonal operator is an easy combinatorial exercise, but to show that the operator $[\Df^\circ,\Uf^\circ]$ is actually scalar requires to use the technique of Macdonald symmetric functions. It seems that there is no direct combinatorial proof of that fact.

Theorem \ref{thm:qt_Plancherel} allows to apply the formalism of \S \ref{sec:heisenberg_operators} and define the \emph{$(q,t)$-Plancherel measures} $\Pl_n^{q,t}$ on the graph $(\Yb,\tilde\km_{q,t})$ which form a coherent system. One can also consider the poissonization $\Pl_\ga^{q,t}$ of these measures and introduce a Markov jump dynamics on partitions preserving $\Pl_\ga^{q,t}$. The generator of this dynamics acts diagonally on an orthogonal basis in $\ell^2(\Yb,\Pl_\ga^{q,t})$ formed by functions $\{\Cf^{q,t}_\la\}_{\la\in\Yb}$ given by formula (\ref{Cf_la}) (with quantities $\dim(\mu,\la)$ and the relative dimension functions $\Pc_\mu^*$ specialized for our graph $(\Yb,\tilde\km_{q,t})$). These functions $\Cf_\la^{q,t}$ provide a two-parameter extension of the Charlier symmetric functions of \cite{Olshanski2011Meixner}. Note that one cannot view the functions $\Cf_\la^{q,t}$ as elements of the algebra of symmetric functions (see also \S\ref{ssub:te-regular} below).

Unfortunately, for general parameters $(q,t)$ the graph $(\Yb,\tilde\km_{q,t}))$ does not carry any Kerov's operators: the set of solutions of (\ref{q_condition})--(\ref{q_condition_0}) in this case is empty. In the next subsections we consider various degenerations of the general $(q,t)$ edge multiplicities. For some of these degenerations the Kerov's operators exist, and we discuss the measures and dynamics arising from the formalism of the present paper. 

% subsection duality_of_qt (end)

\subsection{Degenerations of $(q,t)$ edge multiplicities} % (fold)
\label{sub:degenerations_of_qt}

There are various degenerations of the two-parameter Macdonald symmetric functions:
\begin{enumerate}[{\bf1.}]
  \item ($q=0$) Hall-Littlewood symmetric functions \cite[Ch. III]{Macdonald1995}.
  \item ($t=q^\theta$ and $q\to1$, where $\te\ge0$ is a new parameter\footnote{Macdonald's $\al$ parameter is related to $\theta$ as $\al=1/\theta$. We follow the notation of \cite{Kerov1998}, \cite{Olshanski2009} and use $\te$ instead of $\al$.}) Jack's symmetric functions \cite[Ch. VI, \S10]{Macdonald1995}.
  \item ($q=t$) Schur functions.
  \item ($q=0$, $t=1$) Monomial symmetric functions \cite[Ch. I, \S2]{Macdonald1995}.
  \item ($q=0$, $t=-1$) Schur's Q-functions \cite[Ch. III, \S8]{Macdonald1995}.
\end{enumerate}
All these families of symmetric functions except for the last one are indexed by all partitions and form a linear basis of $\La$. The Schur's Q-functions are indexed by all strict partitions (i.e., partitions in which all nonzero parts are distinct) and span a proper subalgebra of so-called \emph{doubly symmetric functions} $\Gamma\subset\La$.

Thus, in the first four cases one gets branching graphs with the underlying poset $\Z^{2}_{>0}$ and various edge multiplicity functions. In particular, the Young graph with simple edges corresponds to the basis of Schur functions. In the case of Schur's Q-functions one arrives at the Schur graph. This is the branching graph of ideals with the underlying poset $\{(\i,\j)\in\Z_{>0}^{2}\colon \j\ge \i\}$ and simple edges. The vertices of this graph are the so-called shifted shapes which are identified with strict partitions \cite[Ch. I, \S1, Ex. 9]{Macdonald1995}.

It turns out that all branching graphs arising in the above degenerations except for the first one (with Hall-Littlewood multiplicities) carry Kerov's operators. In the next three sections we discuss measures and processes arising in the cases 2--5.

% subsection degenerations_of_qt (end)

\subsection{$z$-measures with Jack parameter} % (fold)
\label{sub:measures_with_jack_parameter}

The Pieri rule for the Jack's symmetric functions $P_\la(x;\te)$ \cite[Ch. VI, \S10]{Macdonald1995} gives rise to the edge multiplicity function $\km_\te$ on $\Yb$ which is defined exactly as in (\ref{kqt}) with $F_{q,t}$ replaced by its limit $F_{\te}(a,l)=\frac{(a+\te(l+2))(a+1+\te l)}{(a+\te(l+1))(a+1+\te(l+1))}$. We are assuming that $\te>0$ is fixed. The edge multiplicities $\km_\te$ were considered in \cite{Kerov1998}, \cite{Kerov2000}, \cite{Borodin2000}, \cite{Borodin2005b}, \cite{Olshanski2009}, and other papers. For $\te=1$ one has $\km_\te\equiv1$, and one recovers the Young graph. The results in that particular case were discussed in Introduction. 

\subsubsection{Kerov's operators} % (fold)

Following \S\S \ref{sub:macdonald}--\ref{sub:duality_of_qt}, we define the UD-self-dual edge multiplicities $\tilde \km_\te$ equivalent to $\km_\te$.
\begin{theorem}\label{thm:Kerov_Jack}
  Triplets of Kerov's operators for the Young graph with Jack edge multiplicities $\tilde\km_\te$ are parametrized by two complex parameters $z,z'$. Each triplet is defined by (\ref{Kerov_UfDf_def}) with the corresponding function $\ufunc(\i,\j)$ on the underlying poset $\Z_{>0}^{2}$ given by
  \begin{equation}\label{Jack_ufunc}
    \ufunc_{zz'\te}(\i,\j)=\big\{\big(z+(\j-1)-\te(\i-1)\big)\big(z'+(\j-1)-\te(\i-1)\big)\big\}^{\frac12},
  \end{equation}
  where $\square=(\i,\j)\in\Z_{>0}^{2}$. The parameter $\zz=(\Hf\un\varnothing,\un\varnothing)$ is equal to $zz'/\te$.
\end{theorem}
\begin{proof}
  Solving (\ref{q_condition}) for $\la=\varnothing$, $\la=(n)$ and $\la=(1^{n})$ (where $n=1,2,\ldots$), one arrives at the desired expressions (\ref{Jack_ufunc}) for $(\i,\j)=(1,n)$ and $(n,1)$ for all $n\ge1$. Then, considering all possible rectangular Young diagrams, one checks that (\ref{Jack_ufunc}) must hold for all $(\i,\j)\in\Z_{>0}^{2}$. 

  Now it remains to show that (\ref{q_condition}) is true for any $\la\in\Yb$, i.e., to check the following identity (here $c_\te(\i,\j):=(\j-1)-\te(\i-1)$ is the $\te$-content of a box $\square=(\i,\j)$):
  \begin{align}&
    \label{te_jack_commut}
    \sum\nolimits_{\nu\colon\nu\searrow\la}
    \tilde\km_{\te}(\la,\nu)^{2}\big(z+c_\te(\nu/\la)\big)
    \big(z'+c_\te(\nu/\la)\big)\\&\qquad \qquad-
    \sum\nolimits_{\mu\colon\mu\nearrow\la}
    \tilde\km_{\te}(\mu,\la)^{2}
    \big(z+c_\te(\la/\mu)\big)
    \big(z'+c_\te(\la/\mu)\big)
    =2|\la|+{zz'}/{\te}.\nonumber
  \end{align}
  This can be checked using the results of \cite{Olshanski2009}. Namely, in \cite[Thm. 7.1]{Olshanski2009} certain operators $U$ and $D$ are written as formal differential operators (in the algebra of $\te$-regular functions on Young diagrams). Arguing by analogy with \cite[Thm. 8.2]{Olshanski2009}, it we readily see that for those operators one has $[U,D]=2\te\mathbf{h}_2+\te zz'$ (in the notation of \cite{Olshanski2009}). Using the definitions of $U$ and $D$ in \cite[\S7.1]{Olshanski2009}, one concludes that this commutation relation is equivalent to (\ref{te_jack_commut}).
\end{proof}

\subsubsection{Measures $M_n^{zz'\te}$}

The measures $M_n^{zz'\te}$ on $\Yb_n$ arising from the Kerov's operators via (\ref{Kerov_coherent}) are exactly the $z$-measures on partitions with Jack parameter $\te>0$. They were defined in \cite{Kerov2000}, \cite{Borodin2000}, see also \cite{Borodin2005b}. In the latter paper conditions under which $M_n^{zz'\te}(\la)>0$ for all $\la\in\Yb_n$ are given.

Thus, Theorem \ref{thm:Kerov_Jack} provides a new characterization of the $z$-measures on partitions with Jack parameter together with all their degenerations. See, e.g., \cite{Borodin2005b} for a description of some degenerations for a general $\te$, and also \cite{borodin2006meixner} for discussion of all degenerate series of the $z$-measures on the Young graph (case $\te=1$) when they live on Young diagrams with bounded number of rows, or columns, or both (i.e., on Young diagrams inside a fixed rectangular shape).

Note that if one divides both sides of (\ref{te_jack_commut}) by $zz'$ and lets $z,z'\to\infty$, one recovers the commutation relation $[\Df^\circ,\Uf^\circ]=1/\te$ for $(\Yb,\tilde\km_\te)$. Moreover, as $z,z'\to\infty$, the measures $M_n^{zz'\te}$ converge to the Plancherel measure $\Pl_n^{\te}$ on the $n$th floor of the graph $(\Yb,\tilde\km_\te)$, see \S \ref{sec:heisenberg_operators}.

The paper \cite{Olshanski2009} is devoted to the study of the up/down Markov chains preserving $M_n^{zz'\te}$, and their asymptotics. In a scaling limit  as $n\to\infty$, the up/down chains converge to infinite-dimensional diffusions on the boundary of $(\Yb,\tilde\km_\te)$ which can be identified with the Thoma simplex $\Omega$ (\ref{Thoma_simplex}).

\subsubsection{Dynamics $\labf_\xi^{zz'\te}$}

Let $M_{\xi}^{zz'\te}$ denote the mixed measures for $(\Yb,\tilde\km_\te)$ as in \S \ref{sub:mixing_of_measures}. For $\te=1$, they form a determinantal random point process \cite{Borodin2000a}. Apart from $\te=1$, there are two other special cases, namely, $\te=2$ and $\te=\frac12$, when $M_{\xi}^{zz'\te}$ can be interpreted as a Pfaffian point process \cite{strahov2009z}, \cite{Strahov2009}. 

Let $\labf_{\xi}^{zz'\te}(t)$ be the Markov jump dynamics on the set of all partitions which preserves the measure $M_{\xi}^{zz'\te}$ (\S \ref{sub:markov_processes_on_}). The Markov generator of $\labf_{\xi}^{zz'\te}$ acts diagonally in the orthogonal basis $\{\Mf_\la^{zz'\te\xi}\}_{\la\in\Yb}$ of the Hilbert space $\ell^2(\Yb,M_\xi^{zz'\te})$. The functions $\Mf_\la^{zz'\te\xi}$ are given by Definition \ref{def:Mf_la} with $\dim=\dim_\te$, $\Pc_\mu^*$, and $\ufunc=\ufunc_{zz'\te}$ suitably specialized for our graph $(\Yb,\tilde\km_\te)$.

\subsubsection{Algebraic structure of the functions $\Mf_\la^{zz'\te\xi}$}
\label{ssub:te-regular} 

The functions $\Mf_\la^{zz'\te\xi}$ are expressed through the relative dimension functions $\Pc_\mu^*$ on the graph $(\Yb,\tilde\km_\te)$. Let us discuss the algebraic nature of the $\Pc_\mu^*$'s. From \cite{OkounkovOlshanskiJack1996} (see also \cite{Kerov1998}) it follows that these functions have the form $\Pc_\mu^*(\la)=({b_\mu^{(1/\te)}})^{\frac12}P_\mu^*(\la;\te)$ for all $\mu,\la\in\Yb$, where ${b_\mu^{(1/\te)}}$ is given by \cite[Ch. VI, (10.16)]{Macdonald1995}.

This implies that the space $\A$ for $(\Yb,\tilde\km_\te)$ is the same as the algebra of $\te$-regular functions on Young diagrams (e.g., see \cite[\S6]{Olshanski2009}) which is a filtered algebra under pointwise multiplication. The graded algebra associated with $\A$ is (regardless of the value of $\te>0$) isomorphic to the algebra of symmetric functions $\La$. For $\te\ne1$ there is no natural way of identifying $\A$ with $\La$. This identification for $\te=1$ is possible my means of the Frobenius-Schur functions on Young diagrams. This construction is explained in, e.g., \cite{OlshRegVer2003}.

Thus, the functions $\Mf_\la^{zz'\te\xi}$ which diagonalize the generator of the dynamics $\labf_\xi^{zz'\te}$ form a linear basis of the algebra of $\te$-regular functions on Young diagrams. The functions $\Mf_\la^{zz',\te=1,\xi}$ are the Meixner symmetric functions introduced in \cite{Olshanski2010LaguerreMeixner}, \cite{Olshanski2011Meixner}. We see that our technique provides a Jack extension of the Meixner symmetric functions of those papers which are, however, cannot longer be seen as symmetric functions.

\subsubsection{Limit as $\xi\to1$ and Laguerre symmetric functions with Jack parameter}
\label{ssub:limit_Laguerre}

Let us briefly discuss the scaling limit as $\xi\to1$ of the dynamics and measures for the graph $(\Yb,\tilde\km_\te)$. This limit is related to the one described in \S \ref{sub:birth_and_death_processes_related_to_the_meixner_polynomials}.

There are embeddings $\phi_n\colon \Yb_n\hookrightarrow\Omega$ for all $n\in\Z_{\ge0}$ (denoted by $\nu\mapsto \omega_\nu(\te)$ in \cite{Kerov1998}) under which the up/down Markov chains for the graph $(\Yb,\tilde\km_\te)$ converge to diffusions on $\Omega$ \cite{Olshanski2009}. Let $\tilde\Omega= \Omega\times\R_{>0}$ be the cone over $\Omega$ (see \cite{Olshanski2010LaguerreMeixner}, \cite{Olshanski2011Meixner} for more detail). Consider embeddings
\begin{equation*}
  \phi_\xi\colon \Yb\hookrightarrow \tilde\Omega,\qquad
  \phi_\xi\colon\la\mapsto (\phi_{|\la|}(\la),(1-\xi)|\la|)\in\Omega\times\R_{>0}.
\end{equation*}
Under these embeddings, the process $\labf_\xi^{zz'\te}(t)$ converges to a diffusion process $X^{zz'\te}(t)$ on $\tilde\Omega$. This fact for $\te=1$ is a subject of Borodin--Olshanski's future paper \cite{BorodinOlsh2011Prep}, and for all $\te>0$ it can be established in a similar manner.

Consider the following functions of $(\omega,r)\in\tilde\Omega$ indexed by $\la\in\Yb$:
\begin{equation}\label{Laguerre_Jack}
  \Lf^{zz'\te}_\la(\omega,r)=\sum_{\mu\subseteq\la}
  (-1)^{|\la|-|\mu|}\frac{\dim_{\te}(\mu,\la)}{(|\la|-|\mu|)!}
  \Big(\prod_{\square\in\la/\mu}
  \ufunc_{zz'\te}(\square)^{2}\Big)(b_\mu^{(1/\te)})^{\frac12}P_\mu((\omega,r);\te).
\end{equation}
Here $\ufunc_{zz'\te}$ is given in Theorem \ref{thm:Kerov_Jack}, and $P_\mu$ is the Macdonald symmetric function. The value $P_\mu((\omega,r);\te):=\psi_{(\omega,r)}(P_\mu)$ is determined via the algebra homomorphism (specialization) $\psi_{(\omega,r)}\colon\La\to\fun(\tilde\Omega)$ defined on Newton power sums as 
\begin{equation}\label{Jack_spec}
  \psi_{(\omega,r)}(p_k)=\begin{cases}
    r,&k=1,\\
    r^k\Big(\sum\nolimits_{i=1}^{\infty}\al_i^{k}+(-\te)^{k-1}
    \sum\nolimits_{i=1}^{\infty}\be_i^{k}\Big),&k\ge2.
  \end{cases}
\end{equation}
The $\Lf^{zz'\te}_\la$'s are eigenfunctions of the Markov generator $\bar{\mathsf{A}}$ of the limiting diffusion $X^{zz'\te}$, and $\bar{\mathsf{A}}\Lf^{zz'\te}_\la=-|\la|\Lf^{zz'\te}_\la$ for all $\la\in\Yb$. 
\begin{remark}\label{rmk:separation_of_var}
  Similarly to \cite[\S9]{Olshanski2010LaguerreMeixner}, one can separate the variable $r$ in the generator $\bar{\mathsf{A}}$ and view the diffusion $X^{zz'\te}(t)$ (at least on an algebraic level) as a skew product of the radial part (which is the one-dimensional diffusion with generator (\ref{1dim_Laguerre})) and the diffusion in the simplex $\Omega$ introduced in \cite{Olshanski2009}. 
\end{remark}

For $\te=1$, the functions (\ref{Laguerre_Jack}) become the Laguerre symmetric functions of \cite{Olshanski2010LaguerreMeixner}, \cite{Olshanski2011Meixner}. For a general $\te$, forgetting about the specialization $\psi_{(\omega,r)}$ in (\ref{Laguerre_Jack}), we can view $\Lf^{zz'\te}_\la$'s also as elements of the algebra of symmetric functions $\La$. Moreover, they form a linear basis in $\La$. The functions $\Lf^{zz'\te}_\la$ are the Laguerre symmetric functions with Jack parameter introduced in \cite{Desrosiers2011Laguerre}.

Let us denote by $\Mc^{zz'\te}$ the scaling limit as $\xi\to1$ (under embeddings $\phi_\xi$) of the measures $M_\xi^{zz'\te}$. Thus, $\Mc^{zz'\te}$ is a measure on $\tilde\Omega$. The Laguerre symmetric functions $\{\Lf^{zz'\te}_\la\}_{\la\in\Yb}$ viewed as functions on $\tilde\Omega$ form an orthogonal basis in $L^2(\tilde\Omega,\Mc^{zz'\te})$, and 
\begin{equation*}
  \left\langle\Lf^{zz'\te}_\la,\Lf^{zz'\te}_\mu\right\rangle_{\Mc^{zz'\te}}=
  \delta_{\la\mu}\prod_{\square\in\la}\ufunc_{zz'\te}(\square)^{2},\qquad
  \la,\mu\in\Yb.
\end{equation*}

The functions $\Lf^{zz'\te}_\la$ can be obtained in a scaling limit as $\xi\to1$ from the functions $\Mf_\la^{zz'\te\xi}$ on Young diagrams. For $\te=1$ this is explained in \cite[\S4.10]{Olshanski2011Meixner}, and for general $\te$ this can be done in a similar way, but in the limit $\te$-regular symmetric functions (elements of $\A$) become symmetric functions (elements of $\La$). See \cite{Kerov1998} and \cite{Olshanski2009} for more detail.

% subsection measures_with_jack_parameter (end)

\subsection{Kingman graph and Poisson-Dirichlet distributions} % (fold)
\label{sub:kingman_graph_and_poisson_dirichlet}

\subsubsection{Kerov's operators and measures $M_n^{\al\tau}$}

Now let us consider another multiplicity function on the lattice of Young diagrams which defines the Kingman graph corresponding to branching of set partitions. As is explained in \cite{Kerov1998}, the Kingman multiplicities can be understood as $\te\to0$ limits of the $\km_\te$'s from the previous subsection:
\begin{equation}\label{km_0}
  \lim_{\te\to0}\km_\te(\la,\mu)=\km_0(\la,\nu):=r_{k}(\nu),\qquad \la\nearrow\nu,
\end{equation}
where $k$ is the length of the row in $\nu$ which contains the box $\nu/\la$, and $r_k(\nu)$ is the number of rows of length $k$ in $\nu$. The dual multiplicities $\km_\te^*$ do not have a limit as $\te\to0$, but if one multiplies them by $\te$ (which leads to an equivalent multiplicity function), then
\begin{equation}\label{km_0_star}
  \lim_{\te\to0}\te\cdot\km_{\te}^*(\la,\nu)=: \km_0^*(\la,\nu)=r_{k-1}(\la),\qquad\la\nearrow\nu
\end{equation}
(we use the same notation as in (\ref{km_0})). It was noted in \cite{Kerov1989} that the multiplicities $\km_0$ and $\km_0^{*}$ are equivalent. It can be also shown that they are UD-dual. As above, define the UD-self-dual multiplicy function by $\tilde\km_0:=\sqrt {\km_0\km_0^*}$.

\begin{theorem}\label{thm:KO_Kingman}
  Triplets of Kerov's operators on the Kingman graph $(\Yb,\tilde\km_0)$ are parametrized by two complex parameters $\al,\tau$. Each triplet is defined by (\ref{Kerov_UfDf_def}) with the function $\ufunc(\i,\j)$ on the underlying poset $\Z_{>0}^{2}$ given by
  \begin{equation*}
    \ufunc_{\al\tau}(\i,\j)=\begin{cases}
      \sqrt{\tau+(\i-1)\al},&\ \mbox{if $\j=1$};\\
      \sqrt{\j(\j-1-\al)},&\ \mbox{if $\j\ge2$}.
    \end{cases}
  \end{equation*}
  One has $\zz=(\Hf\un\varnothing,\un\varnothing)=\tau$.
\end{theorem}
\begin{proof}
  As for Theorem \ref{thm:Kerov_Jack}, one first verifies that all Kerov's operators must have the desired form. Then it is not hard to show (it was actually done in \cite{Petrov2007}, the arXiv ``v1'' version) that (\ref{q_condition}) holds for all Young diagrams $\la$.
\end{proof}

The measures $\{M_n^{\al\tau}\}$ arising from these Kerov's operators via (\ref{Kerov_coherent}) coincide with the Ewens-Pitman's partition structures, see \cite{Ewens1979}, \cite{Pitman1992}, \cite{Pitman1997}, and also \cite[\S1.3]{Petrov2007}.\footnote{We denote by $\tau$ the parameter which is usually denoted by $\te$ in the literature to avoid notational conflict with the Jack parameter $\te$.} These measures $\{M_n^{\al\tau}\}$ satisfy $M_n^{\al\tau}(\la)\ge0$ for all $\la\in\Yb_n$ in one of the following three cases:

(1) $0\le\al<1$ and $\tau>-\al$. Then each measure $M_n^{\al\tau}$ is concentrated on the whole $\Yb_n$.

(2) $\al=-\be<0$ and $\tau=N\be$ for some $N=1,2,\ldots$. Then $M_n^{\al\tau}(\la)>0$ iff $\la$ has no more than $N$ rows.
  
(3) $\al=1$ and $\tau>-1$. Then for each $n$, the measure $M_n^{\al\tau}$ is concentrated on a single one-column Young diagram $\la=(1^{n})$.

We see that our approach provides a new characterization of the two-parameter Ewens-Pitman's partition structures together with all their degenerations. Moreover (\cite[Rmk. 9.12]{Olshanski2009}), these measures on partitions can also be obtained from the $z$-measures with Jack parameter in a suitable limit. 

The third (the trivial) case above actually corresponds to the Heisenberg algebra structure on the graph $(\Yb,\tilde\km_0)$. We see that the highly nontrivial Plancherel measures with Jack parameter $\te>0$ become a trivial object for $\te=0$.

The scaling limit of up/down Markov chains leads to infinite-dimensional diffusions in the boundary of the Kingman graph $\overline\nabla_\infty:=\big\{(x_1\ge x_2\ge \ldots\ge 0),\ \sum_{i=1}^{\infty}x_i\le 1\big\}$. These diffusions preserve the remarkable two-parameter Poisson-Dirichlet measures (e.g., see \cite{Pitman1997} or \cite{Feng2010book}). In the case $\al=0$ these diffusions were constructed in \cite{Ethier1981}, the two-parameter case was investigated in \cite{Petrov2007}.
    
\subsubsection{Relative dimension functions}

It is worth noting that the nondegenerate measures $M_n^{\al\tau}$ can correspond to the negative parameter $\zz=\tau$. However, our considerations of the dynamics in \S \ref{sec:markov_jump_dynamics} are only valid for $\zz>0$, so from now on we work under this assumption. Let $M_\xi^{\al\tau}$ denote the mixture of the measures $M_n^{\al\tau}$ (\S \ref{sub:mixing_of_measures}), and $\labf_\xi^{\al\tau}$ be the Markov jump dynamics on partitions which preserves the measures $M_\xi^{\al\tau}$. %Our aim is to investigate the eigenfunctions of the generator of $\labf_\xi^{\al\tau}$.

Let $m_\la$ be the monomial symmetric functions
\begin{equation}\label{m_la}
  m_\la(x_1,x_2,x_3,\ldots)=\sum x_{i_1}^{\la_1}\ldots x_{i_{\ell}}^{\la_\ell},
  \qquad \la=(\la_1,\ldots,\la_\ell)\in\Yb,
\end{equation}
where the sum is taken over all distinct summands. Let $m_\la^{*}$ be the factorial monomial symmetric function which is given by the same sum as (\ref{m_la}) with each power $x^k$ replaced by the falling factorial power $x^{\da k}$. Note that each $m_\la^*$ is also a symmetric function (in constrast to \S\ref{ssub:te-regular}). We have
\begin{equation*}
  \Pc_\mu^*=r_\mu^{1/2}m_\mu^*,\qquad \mu\in\Yb.
\end{equation*}
Here and below $r_\mu:=r_1(\mu)!r_2(\mu)!\ldots$. This fact can be checked in a straightforward way using the recurrence for the relative dimension.

\subsubsection{Degenerate parameters}

Let $\{\Mf_{\la}^{\al\tau\xi}\}_{\la\in\Yb}$ be the eigenfunctions of the generator of the Markov jump dynamics $\labf_\xi^{\al\tau}$. They are given by (suitably specialized) formula of Definition \ref{def:Mf_la}. However, these functions can be also characterized in a way similar to \cite[Prop. 4.22]{Olshanski2011Meixner}. One first needs to consider these functions in the case of degenerate parameters $(\al=-\be,\tau=N\be)$. In this case $\Mf_{\la}^{\al\tau\xi}$'s become functions in $N$ variables which are indexed by partitions with $\le N$ rows. Denote $\Yb(N):=\{\la\in\Yb\colon\ell(\la)\le N\}$.

\begin{proposition}\label{prop:Mf_la_degen}
  For degenerate parameters $(\al=-\be,\tau=N\be)$, the functions $\{\Mf_\la^{\al\tau\xi}\}_{\la\in\Yb(N)}$ have the form
  \begin{equation}\label{Mf_la_degen}
    \Mf^{\al\tau\xi}_\la(\nu)=r_\la^{-1/2}
    \sum_{i_1,\dots,i_{\ell(\la)}}
    \meix^{\be,\xi}_{\la_1}(\nu_{i_1})\ldots\meix^{\be,\xi}_{\la_{\ell(\la)}}(\nu_{i_{\ell(\la)}})
  \end{equation}
  where $\la, \nu\in\Yb(N)$, and the sum is taken over all pairwise distinct indices $i_1,\ldots,i_{\ell(\la)}$ from $1$ to $N$. Here $\meix^{\be,\xi}_n$ are the monic Meixner orthogonal polynomials (\S \ref{sub:birth_and_death_processes_related_to_the_meixner_polynomials}).
\end{proposition}
It can be readily seen that the top degree term (in the sense of the filtration of $\A$, see \S \ref{sub:space_A}) of $\Mf^{\al\tau\xi}_\la$ (\ref{prop:Mf_la_degen}) is equal to $\Pc_\la^*=r_\la^{1/2}m_\la^*$, as it should be: the factor $r_\la^{1/2}$ appears due to the summation in (\ref{Mf_la_degen}) over all distrinct indices and not over distinct summands as in $m_\la^{*}$, see (\ref{m_la}).
\begin{proof}
  Let us represent partitions $\varkappa\in\Yb(N)$ as $\varkappa=(\varkappa_1,\ldots,\varkappa_N)$ (with possible zeroes in the end). Let $r_i(\varkappa)$, $i\ge1$, be as above, and $r_0(\varkappa)$ denote the number of zeroes in $\varkappa$. Denote $\tilde r_\varkappa:=r_0(\varkappa)!{r}_\varkappa$. One can write
  \begin{equation}\label{m^*_expansion}
    m_\varkappa^{*}(x_1,x_2,\ldots,x_N)=
    \tilde {r}_\varkappa^{-1}
    \sum\nolimits_{\si\in \Sym(N)}
    x_{\si(1)}^{\da\varkappa_1}\ldots x_{\si(N)}^{\da\varkappa_N}
  \end{equation}
  ($\Sym(N)$ is the symmetric group). Next, we have
  $\prod\limits_{\square\in\la}\ufunc_{\al\tau}(\square)^{2}=N^{\da\ell(\la)}\prod\limits_{i=1}^{\ell(\la)}\la_i!(\be)_{\la_i}$. Denote $\la!:=\la_1!\ldots\la_{\ell(\la)}!$, and $(\be)_{\la}:=(\be)_{\la_1}\ldots(\be)_{\la_{\ell(\la)}}$. Finally, it can be established (e.g., see \cite{Kerov1989}) that $\dim\la=r_\la^{-1/2}{|\la|!}/\la!$.

  With all these preparations, we can write $\Mf_\la^{\al\tau\xi}$ (Definition \ref{def:Mf_la}), where $\nu=(\nu_1,\ldots,\nu_N)\in\Yb(N)$:
  \begin{align*}
    \Mf_\la^{\al\tau\xi}(\nu)&=
    \sum_{\mu\subseteq\la}
    \left(\frac\xi{\xi-1}\right)^{|\la|-|\mu|}
    \frac{\dim_{0}(\mu,\la)}{(|\la|-|\mu|)!}
    \Big(\prod_{\square\in\la/\mu}\ufunc_{\al\tau}(\square)^{2}\Big)
    {r}_\mu^{1/2}
    m_\mu^*(\nu)\\&=
    \sum_{\mu\subseteq\la}
    \left(\frac\xi{\xi-1}\right)^{|\la|-|\mu|}
    \frac{({r}_\la)^{-\frac12}(\be)_{\la}}
    {(N-\ell(\la))!\mu!(\be)_{\mu}}
    \tilde{r}_\mu
    m_\mu^*(\nu)m_\mu^*(\la).
  \end{align*}
  Now observe that the above sum is over $\mu_1\ge\ldots\ge\mu_N$ such that $0\le \mu_i\le \la_i$ ($i=1,\ldots,N$). This sum is equal to the same sum over unordered $\mu_1,\ldots,\mu_N$ divided by the multinomial coefficient $N!/\tilde{r}_\mu$. Thus, we can write (expanding the factorial monomial functions as in (\ref{m^*_expansion})):
  \begin{align*}
    \Mf_\la^{\al\tau\xi}(\nu)&=
    \frac{{r}_\la^{1/2}}{(N-\ell(\la))!}
    \sum_{\si,\tau\in \Sym(N)}\frac1{N!}
    \prod_{i=1}^{N}
    \sum_{\mu_i=0}^{\la_{\si(i)}}
    \left(\frac{\xi}{\xi-1}\right)^{\la_{\si(i)}-\mu_i}
    \frac{\la_{\si(i)}^{\da\mu_i}}{\mu_i!}
    \frac{(\be)_{\la_{\si(i)}}}{(\be)_{\mu_i}}\nu_{\tau(i)}^{\da\mu_i}
    \\&=
    \frac{{r}_\la^{1/2}}{(N-\ell(\la))!}
    \sum_{\rho\in \Sym(N)}\meix^{\be,\xi}_{\la_1}(\nu_{\rho(1)})\ldots
    \meix^{\be,\xi}_{\la_N}(\nu_{\rho(N)}).
  \end{align*}
  Here we have used the definition of the Meixner polynomials (\S \ref{sub:birth_and_death_processes_related_to_the_meixner_polynomials}), and the fact that each summand depends only on the difference $\rho=\tau\si^{-1}$, which makes the factor $N!$ disappear. Now passing from the summation over $\rho\in \Sym(N)$ to the summation over $i_1,\ldots,i_{\ell(\la)}$ as in (\ref{Mf_la_degen}), we conclude the proof.
\end{proof}

\begin{remark}
  It can be readily verified that the measure $M_\xi^{\al\tau}$ for the degenerate parameters $(\al=-\be,\tau=N\be)$ has the form 
  $M_\xi^{\al\tau}(\la)={N!}\tilde{r}_\la^{-1}\prod_{i=1}^{N}\pi_{\be,\xi}(\la_i)$. This means that the components of a random partition $\la=(\la_1\ge \ldots\ge\la_N)$ with distribution $M_\xi^{\al\tau}$ can be identified with the decreasing order statistics of $N$ independent random variables on $\Z_{\ge0}$ with the negative binomial distribution $\pi_{\be,\xi}$ (\ref{NegBinom}).

  The same structure appears in the dynamical picture. Namely, if we consider the evolution of decreasing order statistics of $N$ independent birth and death processes $\n_{\be,\xi}(t)$ on $\Z_{\ge0}$ (\S \ref{sub:markov_processes_on_}), then we get the dynamics $\labf_{\xi}^{\al\tau}(t)$. This can be deduced from Proposition \ref{prop:Mf_la_degen} by comparing actions of two generators (of $\labf_{\xi}^{\al\tau}$ and of the dynamics of decreasing order statistics) on the functions $\Mf_\la^{\al\tau\xi}$.
\end{remark}

\subsubsection{Characterization of $\Mf_\la^{\al\tau\xi}$}

Denote the right-hand side of (\ref{Mf_la_degen}) by $\meix_{\la\mid N,\be,\xi}$. These functions of $N$ variables can be viewed as ``bosonic'' Meixner symmetric polynomials as opposed to the ``fermionic'' ones arising for the graph $(\Yb,1)$ \cite{Olshanski2011Meixner} (see also references there to earlier work).
\begin{proposition}[{cf. \cite[Prop. 4.22]{Olshanski2011Meixner}}]
\label{prop:Kingman_degen}
  For any $\la\in\Yb$, the function $\Mf_\la$ is characterized as the unique element of the algebra
  \begin{equation*}
    \La[\al,\tau,\xi,(1-\xi)^{-1}]=\La\otimes\C[\al,\tau,\xi,(1-\xi)^{-1}]
  \end{equation*}
  such that for any $N=1,2,\ldots$, $\be>0$, and $\xi\in(0,1)$, one has
  \begin{equation*}
    \Mf_\la|_{{}_{\scriptstyle\al=-\be,\;\tau=N\be}}=
    \begin{cases}
      \meix_{\la\mid N,\be,\xi},&\mbox{if $\ell(\la)\le N$},
      \\
      0,&\mbox{if $\ell(\la)>N$}.
    \end{cases}
  \end{equation*} 
\end{proposition}
That is, the symmetric functions $\Mf_\la$ for general $(\al,\tau)$ can be viewed as analytic continuations of the Meixner symmetric polynomials $\meix_{\la\mid N,\be,\xi}$ of Proposition \ref{prop:Mf_la_degen} with respect to the dimension $N$ and the parameter $\be$.

\subsubsection{Limit $\xi\to1$ and Laguerre-type functions}

The discussion of \S \ref{ssub:limit_Laguerre} about the limit behavior can be essentially repeated for the Kingman graph. Let $\widetilde{\overline\nabla}_\infty=\overline\nabla_\infty\times\R_{>0}$ be the cone over the boundary of $(\Yb,\tilde\km_0)$. As $\xi\to1$, under the embeddings
\begin{equation*}
  \Yb\ni\la\mapsto \left(\Big(\frac{\la_1}{|\la|},\ldots,\frac{\la_{\ell(\la)}}{|\la|},0,0,\ldots\Big);(1-\xi)|\la|\right)\in \widetilde{\overline\nabla}_\infty,
\end{equation*}
the dynamics $\labf_\xi^{\al\tau\xi}$ converges to a diffusion $X^{\al\tau}(t)$ process in $\widetilde{\overline\nabla}_\infty$. The generator of this diffusion has eigenfunctions which have the form
\begin{equation*}
  \Lf_\la^{\al\tau}=
  \sum_{\mu\subseteq\la}
  (-1)^{|\la|-|\mu|}\frac{\dim_{0}(\mu,\la)}{(|\la|-|\mu|)!}
  \Big(\prod_{\square\in\la/\mu}\ufunc_{\al\tau}(\square)^{2}\Big)
  {r}_\mu^{1/2}m_\mu,\qquad\la\in\Yb.
\end{equation*}
To view these symmetric functions as functions on $\widetilde{\overline\nabla}_\infty$, one must consider the specializations of the algebra of symmetric functions under which (cf. (\ref{Jack_spec}))
\begin{equation*}
  p_1\mapsto r,\qquad
  p_k\mapsto r^k\sum_{i=1}^{\infty}x_i^{k},\qquad k\ge2.
\end{equation*}

In the same way as for the Meixner-type functions on the Kingman graph (Propositions \ref{prop:Mf_la_degen} and \ref{prop:Kingman_degen}), the Laguerre-type functions are analytic continuations of the following ``bosonic'' Laguerre symmetric polynomials in $N$ variables:
\begin{equation*}
  \mathrm{L}_{\la|N,\be}(y_1,\ldots,y_N)=
  {r}_\la^{1/2}
  \sum_{i_1,\dots,i_{\ell(\la)}}
  \mathrm{L}^{\be}_{\la_1}(y_{i_1})\ldots\mathrm{L}^{\be}_{\la_{\ell(\la)}}(y_{i_{\ell(\la)}}).
\end{equation*}
Here $\mathrm{L}^{\be}_n(y)$ are the monic Laguerre polynomials \cite[\S1.11]{Koekoek1996}, and the sum is taken over all pairwise distinct indices $i_1,\ldots,i_{\ell(\la)}$ from 1 to $N$. 

The diffusion process arising for the multidimensional Pascal triangle (see \S\S \ref{sub:kerov_s_operators_for_pascal_triangle}, \ref{sub:birth_and_death_processes_related_to_the_meixner_polynomials}) are direct products of copies of the Laguerre one-dimensional diffusion with generator (\ref{1dim_Laguerre}). The process $X^{\al\tau}$ for $\la=-\be$, $\tau=N\be$ can be viewed as dynamics of order statistics of a diffusion coming from the $N$-dimensional Pascal triangle.

\begin{remark}
  The diffusion process $X^{\al=0,\tau}$ on $\widetilde{\overline\nabla}_\infty$ coincides with the dynamics of ranked atoms of the measure-valued Jirina process (a special case of Dawson-Watanabe superprocesses), see, e.g., \cite{Jirina1964}, \cite{Dawson1991}, \cite{Etheridge:2000fk}. For $0<\al<1$ and $\tau>0$ we thus get an extension of the diffusion $X^{\al=0,\tau}$ on $\widetilde{\overline\nabla}_\infty$. However, it seems that there is no measure-valued process for $\al\ne0$ corresponding to $X^{\al\tau}$ (see also comments after (14) in \cite{Ruggiero2009a}).

  The generator of the process $X^{\al\tau}$ admits a separation of variables as in Remark \ref{rmk:separation_of_var}. This reflects the fact that conditioning the Jirina process to have total mass 1, one arrives at the Fleming-Viot measure-valued diffusion \cite{etheridge1991note}.
\end{remark}

% subsection kingman_graph_and_poisson_dirichlet (end)

\subsection{Schur graph} % (fold)
\label{sub:schur_graph}

The Schur graph $\Sb$ is usually defined as the lattice of all strict partitions $\la=(\la_1>\la_2>\ldots>\la_\ell>0)$ (where $\la\in\Z$) ordered by inclusion. Strict partitions are represented by shifted Young diagrams as in \cite[Ch. I, \S1, Ex. 9]{Macdonald1995}. Two shifted diagrams are connected by an edge iff one is obtained from another by adding a box. Thus defined, Schur graph can be identified with the branching graph of ideals $(J(P),1)$ with the trivial multiplicity function, where $P=((\i,\j)\in\Z_{>0}^{2}\colon \j\ge\i)$. Note that $P\subset\Z_{>0}^{2}$ is not a subposet, so the Schur graph is not reduced to a sublattice of the Young graph. We choose another edge multiplicity function on $\Sb$ which is equivalent to the trivial one:
\begin{equation*}
  \km_{\Sb}(\la,\nu):=\begin{cases}
    1,&\mbox{if $\la\nearrow\nu$ and $\ell(\nu)=\ell(\la)+1$},\\
    \sqrt 2,&\mbox{if $\la\nearrow\nu$ and $\ell(\nu)=\ell(\la)$}.
  \end{cases}
\end{equation*}
The reason is that $\km_{\Sb}$ is an UD-self-dual multiplicity function.

The Kerov's operators and Markov jump dynamics on the Schur graph were studied in detail in the paper \cite{Petrov2010Pfaffian}. We refer to it for the characterization of the Kerov's operators. They depend on one parameter $a$ which is assumed to be real positive for the corresponding measures $M_n^{a}$ to be nonnegative. One has $\zz=\frac a2$. The measures $\{M_n^{a}\}$ on the Schur graph were introduced in \cite{Borodin1997}. We denote the parameter by $a$ instead of $\al$ as in \cite{petrov2009eng}, \cite{Petrov2010}, \cite{Petrov2010Pfaffian} to avoid a notational conflict with \S \ref{sub:kingman_graph_and_poisson_dirichlet}. In \cite{petrov2009eng} a limit behavior of the up/down Markov chains on $\Sb$ was studied, it leads to infinite-dimensional diffusion processes on the boundary of the Schur graph which is $\Omega_+:=\big\{\al\in\R^{\infty}\colon \al_1\ge\al_2\ge\dots\ge0,\ \sum\nolimits_{i=1}^\infty\al_i\le1\big\}$.

By $M_\xi^{a}$ and $\labf_\xi^{a}$ we denote the mixed measure and the Markov jump process on $\Sb$ constructed from the Kerov's operators. The measures $M_\xi^{a}$ is a determinantal point process on $\Z_{\ge0}$ \cite{Petrov2010}, and the dynamics $\labf_\xi^{a}$ is Pfaffian \cite{Petrov2010Pfaffian}. The technique of the present paper allows to construct eigenfunctions for the Markov generator of the dynamics $\labf_\xi^{a}$.

The relative dimension functions on the Schur graph essentially coincide with the factorial Schur's Q-functions $Q_\mu^*$ introduced in \cite{IvanovNewYork3517-3530}. These functions are symmetric (there is no effect of \S \ref{ssub:te-regular}), and they span a proper subalgebra $\Ga\subset\La$ of the algebra of symmetric functions, see \cite[Ch. III, \S8]{Macdonald1995}. Elements of $\Ga$ can be called doubly symmetric functions (the name is taken from \cite{Berele2009}, see also \cite{petrov2009eng}). They can be characterized as follows \cite{Stembridge1985}. A symmetric function $f(x_1,x_2,\ldots)$ is doubly symmetric iff under the substitution $x_i=u$, $x_j=-u$ (where $i<j$, and $u$ is another formal variable) the function $f(x_1,\ldots,u,\ldots,-u,\ldots)$ does not depend on $u$. An equivalent description is that $\Gamma\subset\La$ is the subalgebra generated by the Newton power sums with odd numbers. 

The eigenfunctions of the generator of the jump dynamics $\labf_\xi^a$ are indexed by strict partitions $\la\in\Sb$ and look as
\begin{equation*}
  \Mf_\la^{a\xi}=
  \sum_{\mu\subseteq\la}
  \left(\frac{\xi}{\xi-1}\right)^{|\la|-|\mu|}
  \frac{\dim_{\Sb}(\mu,\la)}{(|\la|-|\mu|)!}2^{-\ell(\mu)/2}
  \Big(\prod_{\square\in\la/\mu}\ufunc_a(\square)\Big)Q_\mu^*.
\end{equation*}
Here $\ufunc_a(\i,\j):=\{\frac12((\j-\i+1)(\j-\i)+a)\}^{\frac12}$. The functions $\{\Mf_\la^{a\xi}\}_{\la\in\Sb}$ form an orthogonal basis in the Hilbert space $\ell^2(\Sb,M_\xi^{{a}})$, and also a linear basis in the algebra $\Ga$ of doubly symmetric functions.

\begin{remark}
  As for the case of the $z$-measures and the Ewens-Pitman's partition structures, for the measures $M_\xi^{a}$ there exist degenerate values of the parameter $a$. These are $a=-N(N+1)$, where $N=1,2,\ldots$. The degenerate measure $M_\xi^{a}$ lives on shifted diagrams which are inside a staircase shape $(N,N-1,\ldots,1,0)$. Thus, one can define the functions $\Mf_\la^{a\xi}$ for these degenerate parameters. It would be interesting to find an expression (similar to \cite[Prop. 4.22]{Olshanski2011Meixner} and Proposition \ref{prop:Kingman_degen}) for these functions $\Mf_\la^{a\xi}$ on the Schur graph in terms of the one-row functions $\Mf_{(k)}^{a\xi}(x)$.
\end{remark}

As $\xi\to1$, the dynamics $\labf_\xi^{a}$ converges to a diffusion process on the cone $\tilde\Omega_+$ over $\Omega_+$. The limiting diffusion also has Pfaffian dynamical correlation functions \cite{Petrov2010Pfaffian}. The situation here is parallel to the discussion of \S \ref{ssub:limit_Laguerre}, and leads to the following doubly symmetric Laguerre-type functions:
\begin{equation*}
  \Lf_\la^{a}=
  \sum_{\mu\subseteq\la}
  \left(-1\right)^{|\la|-|\mu|}
  \frac{\dim_{\Sb}(\mu,\la)}{(|\la|-|\mu|)!}2^{-\ell(\mu)/2}
  \Big(\prod_{\square\in\la/\mu}\ufunc_a(\square)\Big)Q_\mu,\qquad \la\in\Sb.
\end{equation*}
Here $Q_\mu$ are the usual Schur's Q-functions \cite[Ch. III, \S8]{Macdonald1995}. The functions $\{\Lf_\la^a\}_{\la\in\Sb}$ form a basis of the algebra $\Ga$. There also exists a measure on $\tilde \Omega_+$ (a scaling limit of the measures $M_\xi^{a}$) which is an orthogonality measure for the $\Lf_\la^a$'s. 

% subsection schur_graph (end)

\subsection{Rim-hook lattices} % (fold)
\label{sub:rim_hook_lattices}

Let us briefly discuss the $\slf(2,\C)$ operators for the rim-hook lattices which appeared in \cite{Okounkov2001a} along with the Kerov's operators for the Young graph.

Let us fix a natural number $r\ge1$. A rim-hook of a Young diagram $\la$ is a skew diagram $\la/\mu$ which is connected (the squares of it are connected by common edges) and lies in the rim of $\la$. A diagram $\la$ is called $r$-decomposable if one can consecutively remove rim-hooks of length $r$ ($r$-rim-hooks) from $\la$ and obtain an empty diagram in the end. A diagram $\la$ is said to be an $r$-core if one cannot remove any $r$-rim-hooks from it. The operation of adding/removing an $r$-rim-hook defines a branching of $r$-decomposable diagrams. Denote the branching graph thus arising by $\Yb^{(r)}$. It can be shown (e.g., see \cite{fomin1997rim}) that $\Yb^{(r)}$ is a lattice which is isomorphic to the direct product $\Yb\times \ldots\times \Yb$ ($r$ times). In other words, $\Yb^{(r)}$ is isomorphic to a branching graph of ideals whose underlying poset is $\Z_{>0}^{2}\sqcup \ldots\sqcup\Z_{>0}^{2}$ ($r$ times) and the multiplicity function $\km\equiv1$ is trivial.

\begin{remark}
  The whole set of Young diagrams with the branching defined by removing $r$-rim-hooks (for fixed $r$) is isomorphic to a disjoint union $\bigsqcup_{\text{$r$-cores}}\Yb^{(r)}$. For the study of the Kerov's operators it is enough to consider only one component consisting of $r$-decomposable diagrams.
\end{remark}

Using the isomorphism $\Yb^{(r)}\cong\Yb\times \ldots\times \Yb$ (we say that $\la\in\Yb^{(r)}$ corresponds to an $r$-tuple $(\la^{1},\ldots,\la^{r})$), one can readily characterize all Kerov's operators for the lattice $\Yb^{(r)}$. The Kerov's operators $(\Uf_r^{\mathbf{z}},\Df_r^{\mathbf{z}},\Hf_r^{\mathbf{z}})$ for $\Yb^{(r)}$ depend on $2r$ parameters $\mathbf{z}:=(z_1,z_1',\ldots,z_r,z_r')\in\C^{2r}$ and are defined with the help of Proposition \ref{prop:KO_disjoint} by taking the Kerov's operators for the Young graph with parameters $z_i,z_i'$ at the $i$th copy of $\Yb$ (i.e., acting on $\la^{i}$). On the lattice $\Yb^{(r)}$, the operators $\Uf_r^{\mathbf{z}}$ and $\Df_r^{\mathbf{z}}$ act by adding or removing an $r$-rim-hook, respectively. One can write formulas for their action, but we omit them (see (\ref{KO_rimhook}) below for a particular case). The Kerov's operator $\Hf_r^{\mathbf{z}}$ acts on $\un\la$, $\la\in\Yb^{(r)}$ as (clearly, $r^{-1}|\la|=|\la^{1}|+\ldots+|\la^{r}|$)
\begin{equation*}
  \Hf_r^{\mathbf{z}}\un\la=\left(2r^{-1}{|\la|}+z_1z_1'+\ldots+z_rz_r'\right)\un\la.
\end{equation*}

Now consider the $r$-rim-hook Kerov's operators from \cite[\S3.5]{Okounkov2001a} depending on two parameters $z,z'$:
\begin{equation}\label{KO_rimhook}
  \begin{array}{ll}
    U_r\un\la&\displaystyle=\sum_{\nu\colon\nu=\la+\text{$r$-rim-hook}}
    \Big(z+\frac1{r^2}\sum_{\square=(\i,\j)\in\text{rim-hook $\nu/\la$}}(\j-\i)\Big)\un\nu,\\
    D_r\un\la&\displaystyle=\sum_{\mu\colon\mu=\la-\text{$r$-rim-hook}}
    \Big(z'+\frac1{r^2}\sum_{\square=(\i,\j)\in\text{rim-hook $\la/\mu$}}(\j-\i)\Big)\un\mu,
  \end{array}
\end{equation}
where the row and column coordinates $(\i,\j)$ of a box are defined with respect to the Young diagram $\la\in\Yb^{(r)}$. We have removed the factors $(-1)^{\mathrm{ht}+1}$ which are present in \cite[\S3.5]{Okounkov2001a} because otherwise Proposition \ref{prop:KO_rim} below fails. However, these factors could only change signs in certain formulas, and this sign disappears in some of them (see also \cite[\S3.5]{Okounkov2001a}).

One can establish that
\begin{equation*}
  H_r\un\la=[D_r,U_r]\un\la=\left(2r^{-1}|\la|+rzz'+\frac{r^2-1}{12r}\right)\un\la,\qquad\la\in\Yb^{(r)}. 
\end{equation*}
\begin{proposition}\label{prop:KO_rim}
  The operators (\ref{KO_rimhook}) coincide (up to a gauge transformation of \S \ref{sub:definition_KO}) with the Kerov's operators $\Uf_r^{\mathbf{z}}$ and $\Df_r^{\mathbf{z}}$ constructed using the isomorphism $\Yb^{(r)}\cong\Yb\times \ldots\times \Yb$, where the parameters $\mathbf{z}$ and $z,z'$ are related as
  \begin{equation*}
    z_i=z+\frac{r+1-2i}{2r},\qquad
    z_i'=z'+\frac{r+1-2i}{2r},\qquad
    i=1,\ldots,r.
  \end{equation*}
\end{proposition}
The Kerov's operators on $\Yb^{(r)}$ depending on a parameter $\mathbf{z}$ define mixed measures $M_\xi^{\mathbf{z}}$ and jump dynamics $\labf_\xi^{\mathbf{z}}$ on $\Yb^{(r)}$. At each jump the process $\labf_\xi^{\mathbf{z}}$ adds or removes an $r$-rim-hook. This process can be viewed as a direct product of $r$ processes on the Young graph corresponding to the $z$-measures (see \S \ref{ssub:Stochastic dynamics associated with the z-measures} and \S\ref{sub:measures_with_jack_parameter}). All properties of $\labf_\xi^{\mathbf{z}}$ can be obtained using that fact. For example, the dynamical correlation functions of $\labf_\xi^{\mathbf{z}}$ are determinantal, and the correlation kernel has a block form as in \cite[\S3.5]{Okounkov2001a}. The eigenfunctions of the Markov generator of the jump dynamics look as 
\begin{equation*}
  \Mf_\la^{\mathbf{z},\xi}(\nu)=
  \Mf_{\la^{1}}^{z_1z_1'\xi}(\nu^{1})\ldots
  \Mf_{\la^{r}}^{z_rz_r'\xi}(\nu^{r}),\qquad 
  \la,\nu\in\Yb^{(r)},
\end{equation*}
where $\Mf^{zz'\xi}_{mu}$, $\mu\in\Yb$, are the Meixner symmetric functions of \cite{Olshanski2011Meixner} (see also \S \ref{sub:measures_with_jack_parameter}).

\begin{remark}
  There exists a similar lattice of shifted rim-hook shapes, and it also can be decomposed as a direct product of several copies of the Young and the Schur graphs \cite[Thm. 3.7]{fomin1997rim}. This implies that our constructions also work for those branching graphs.
\end{remark}

\begin{remark}[Slow graphs]
  In \cite{VershikNikitin2011} a graph describing branching of representations of symmetric inverse semigroups was described. This is the so-called slow Young graph which is obtained from the original Young graph by replacing the underlying poset $P=\Z_{>0}^{2}$ by $P\sqcup \Z_{>0}$. (This procedure can be applied to any branching graph, see \cite{VershikNikitin2011}.) For the graph $J(\Z_{>0}^{2}\sqcup \Z_{>0})$ one can readily describe the Kerov's operators and the corresponding coherent systems of measures. They depend on three parameters: two of them come from the $z$-measures on the Young graph, and there is one additional parameter corresponding to the Kerov's operators on the chain $J(\Z_{>0})$ (see Proposition \ref{prop:KO_disjoint}). It might be of interest to find an explicit representation-theoretic construction of the distinguished coherent systems on the slow Young graph in the spirit of \cite{Kerov1993}, \cite{Kerov2004}.
\end{remark}

% subsection rim_hook_lattices (end)

\subsection{Rooted unlabeled trees} % (fold)
\label{sub:rooted_trees}

Fulman \cite{Fulman2009Trees} introduced and studied down/up Markov chains on unlabeled rooted trees. The branching here does not define a graph of ideals. However, the model of \cite{Fulman2009Trees} possesses a $\slf(2,\C)$ structure similar to that of the Kerov's operators.

Let $\Tc_n$ denote the set of all unlabeled rooted trees with $n$ vertices. For example, these are all the trees with $4$ vertices (picture taken from \cite{Fulman2009Trees}):

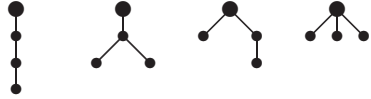
\begin{figure}[h]
\setlength{\unitlength}{.14in}
\begin{picture}(10,0)
\put(0,0){\circle*{.6}}
\put(0,0){\line(0,-1){1}}
\put(0,-1){\circle*{.4}}
\put(0,-1){\line(0,-1){1}}
\put(0,-2){\circle*{.4}}
\put(0,-2){\line(0,-1){1}}
\put(0,-3){\circle*{.4}}
\put(4,0){\circle*{.6}}
\put(4,0){\line(0,-1){1}}
\put(4,-1){\circle*{.4}}
\put(4,-1){\line(1,-1){1}}
\put(4,-1){\line(-1,-1){1}}
\put(5,-2){\circle*{.4}}
\put(3,-2){\circle*{.4}}
\put(8,0){\circle*{.6}}
\put(8,0){\line(1,-1){1}}
\put(8,0){\line(-1,-1){1}}
\put(9,-1){\circle*{.4}}
\put(7,-1){\circle*{.4}}
\put(9,-1){\line(0,-1){1}}
\put(9,-2){\circle*{.4}}
\put(12,0){\circle*{.6}}
\put(12,0){\line(-1,-1){1}}
\put(12,0){\line(0,-1){1}}
\put(12,0){\line(1,-1){1}}
\put(11,-1){\circle*{.4}}
\put(12,-1){\circle*{.4}}
\put(13,-1){\circle*{.4}}
\end{picture}
\end{figure}

\vspace{0.8cm}

It is more convenient for us to denote $\Tb_n:=\Tc_{n+1}$, $n=0,1,\ldots$, i.e., to count the number of edges $|t|$ of a tree $t$. The union $\Tb=\bigsqcup_{n=0}^{\infty}\Tb_n$ can be equipped with a structure of a branching graph. If a tree $t\in\Tb_{n}$ is obtained from $t'\in\Tb_{n+1}$ by removing a single terminal vertex together with an edge, we say that $t$ and $t'$ are connected in $\Tb$ and write as usual, $t\nearrow t'$. Now we describe multiplicities of edges in $\Tb$. Let for $t\nearrow t'$ denote
\begin{align*}
  n(t,t') &:= \#\{ 
  \mbox{vertices of $t$ to which a new edge can be added to get $t'$}\},\\
  m(t,t') &:= \#\{\mbox{edges of $t'$ which when removed give $t$}\}. 
\end{align*}
Set for $t\nearrow t'$, $\km_{\Tb}(t,t'):=\sqrt {n(t,t')m(t,t')}$. We consider the branching graph $(\Tb,\km_{\Tb})$. Though it is not a graph of ideals, it turns out that there are operators on $\Tb$ which span an $\slf(2,\C)$-module similarly to the Kerov's operators. The difference is that there is only a single triplet of such operators on $\Tb$ and there is no dependence on any parameters.

By \cite[Prop. 2.2]{Hoffman2003TreesHopf}, the multiplicity function $\km_\Tb$ is UD-self-dual. Define the following operators in $\ellf^2(\Tb)$ with the standard basis $\{\un t\}_{t\in\Tb}$:
\begin{align*}
  \Uf_{\Tb}\un t:=\sum_{t'\colon t'\searrow t}\sqrt2 \km_{\Tb}(t,t')\un{t'},
  \qquad
  \Df_{\Tb}\un t:=\sum_{t''\colon t''\nearrow t}\sqrt2 \km_{\Tb}(t'',t)\un{t''}.
\end{align*}
Define $\Hf_{\Tb}:=[\Df_{\Tb},\Uf_{\Tb}]$. We have (see \cite[Prop. 2.2]{Hoffman2003TreesHopf} and \cite[(8)]{Fulman2009Trees}):
\begin{equation*}
  \Hf_{\Tb}\un t=(2|t|+2)\un t,\qquad t\in\Tb.
\end{equation*}
Therefore, the operators $(\Uf_{\Tb},\Df_{\Tb},\Hf_{\Tb})$ span an $\slf(2,\C)$-module in $\ellf^2(\Tb)$. The parameter $\zz=(\Hf_{\Tb}\un\varnothing,\un\varnothing)=2$ (where $\varnothing$ denotes the tree with only one vertex and no edges). The measures considered in \cite{Fulman2009Trees} arise from our construction of \S \ref{sec:coherent_systems_of_measures}:
\begin{equation*}
  M_n(t)=\frac1{(2)_nn!}(\Uf_{\Tb}^{n}\un\varnothing,\un t)(\Df_{\Tb}^{n}\un t,\un\varnothing),\qquad t\in\Tb_n.
\end{equation*}
In the notation of \cite[Def. 1]{Fulman2009Trees}, $M_n=\pi_{n+1}$.

The Markov chains studied in \cite{Fulman2009Trees} are the down/up chains on $\Tb_n$. The spectral structure of the closely related up/down Markov chains on $\Tb_n$ can be read from \S \ref{sec:up_down_markov_chains} (see Proposition \ref{prop:T_n_spectrum}).

It is possible to define the mixture $M_\xi$ of the measures $M_n$ and introduce a Markov jump dynamics on rooted trees. Since $\zz=2$, the mixing distribution (depending on a new parameter $\xi\in(0,1)$) looks as
\begin{equation*}
  (1-\xi)^{2}(n+1){\xi}^{n},\qquad n=0,1,\ldots. 
\end{equation*}
The Markov jump dynamics on all unlabeled rooted trees is defined as in \S \ref{sec:markov_jump_dynamics}. The Markov generator of this dynamics has eigenfunctions which are explicitly described (Definition \ref{def:Mf_la}). These eigenfunctions form an orthogonal basis in $\ell^2(\Tb,M_\xi)$.

% subsection rooted_trees (end)

\subsection{A remark on plane partitions} % (fold)
\label{sub:plane_partitions}

Let us mention one important example of a branching graph of ideals to which our technique is not applicable. This is the graph of 3D Young diagrams $(J(\Z_{>0}^{3}),1)$ which are also called plane partitions (e.g., see \cite{Stanley1999}, \cite{okounkov2003correlation}). The branching here corresponds to adding $1\times 1\times 1$ boxes to 3D diagrams. There are no Kerov's operators on it this graph: the equations (\ref{q_condition})--(\ref{q_condition_0}) on $\ufunc$ have an empty set of solutions. There is also no Heisenberg algebra structure (\S \ref{sec:heisenberg_operators}) on this graph.  

However, there exists a related higher-dimensional structure. Namely, the set of all 3D Young diagrams contained inside a box of dimensions $a\times b\times c$ parametrizes a basis of an irreducible representation of the Lie algebra $\slf(a+b,\C)$ (Irreducible finite-dimensional $\slf(2,\C)$-modules in \S \ref{sub:lowest_weight_slf_2_c_modules} arise in a particular case $a=b=1$.) About this structure see \cite{Kuperberg1994SCPP}, and also \cite{Proctor1984Bruhat}, \cite{Stembridge1994minuscule}. It could be of interest to see whether this structure provides any probabilistic models like the measures and Markov dynamics considered in the present paper.

% subsection plane_partitions (end)

% section examples (end)

\providecommand{\bysame}{\leavevmode\hbox to3em{\hrulefill}\thinspace}
\providecommand{\MR}{\relax\ifhmode\unskip\space\fi MR }
% \MRhref is called by the amsart/book/proc definition of \MR.
\providecommand{\MRhref}[2]{%
  \href{http://www.ams.org/mathscinet-getitem?mr=#1}{#2}
}
\providecommand{\href}[2]{#2}


\begin{thebibliography}{Oko01b}

\bibitem[AESW51]{AESW51}
M.~Aissen, A.~Edrei, I.~J. Schoenberg, and A.~Whitney, \emph{On the generating
  functions of totally positive sequences}, Proc. Nat. Acad. Sci. U. S. A.
  \textbf{37} (1951), 303--307.

\bibitem[Akh65]{Akhiezer1965Moment}
N.~I. Akhiezer, \emph{The classical moment problem and some related questions
  in analysis}, Translated by N. Kemmer, Hafner Publishing Co., New York, 1965.

\bibitem[BDJ99]{baik1999distribution}
J.~Baik, P.~Deift, and K.~Johansson, \emph{{On the distribution of the length
  of the longest increasing subsequence of random permutations}}, Journal of
  the American Mathematical Society \textbf{12} (1999), no.~4, 1119--1178,
  arXiv:math/9810105 [math.CO].

\bibitem[BO98]{Borodin1998}
A.~Borodin and G.~Olshanski, \emph{Point processes and the infinite symmetric
  group}, Math. Res. Lett. \textbf{5} (1998), 799--816, arXiv:math/9810015
  [math.RT].

\bibitem[BO00a]{Borodin2000a}
\bysame, \emph{Distributions on partitions, point processes, and the
  hypergeometric kernel}, Commun. Math. Phys. \textbf{211} (2000), no.~2,
  335--358, arXiv:math/9904010 [math.RT].

\bibitem[BO00b]{Borodin2000}
\bysame, \emph{Harmonic functions on multiplicative graphs and interpolation
  polynomials}, Electronic Journal of Combinatorics \textbf{7} (2000), R28,
  arXiv:math/9912124 [math.CO].

\bibitem[BO01]{Borodin1999RSK}
\bysame, \emph{{Z-Measures on partitions, Robinson-Schensted-Knuth
  correspondence, and $\beta=2$ random matrix ensembles}}, Random matrix models
  and their applications (P.M.Bleher and R.A.Its, eds.), Math. Sci. Res. Inst.
  Publ, vol.~40, Cambridge Univ. Press, Cambridge, 2001, arXiv:math/9905189
  [math.CO], pp.~71--94.

\bibitem[BO05a]{Borodin2005}
\bysame, \emph{{Random partitions and the Gamma kernel}}, Adv. Math.
  \textbf{194} (2005), no.~1, 141--202, arXiv:math-ph/0305043.

\bibitem[BO05b]{Borodin2005b}
\bysame, \emph{Z-measures on partitions and their scaling limits}, European J.
  Combin. \textbf{26} (2005), no.~6, 795--834, arXiv:math-ph/0210048.

\bibitem[BO06a]{Borodin2006}
\bysame, \emph{Markov processes on partitions}, Probab. Theory Related Fields
  \textbf{135} (2006), no.~1, 84--152, arXiv:math-ph/0409075.

\bibitem[BO06b]{borodin2006meixner}
\bysame, \emph{{Meixner polynomials and random partitions}}, Moscow
  Mathematical Journal \textbf{6} (2006), no.~4, 629--655, arXiv:math/0609806
  [math.PR].

\bibitem[BO06c]{borodin2006stochastic}
\bysame, \emph{{Stochastic dynamics related to Plancherel measure on
  partitions}}, Representation Theory, Dynamical Systems, and Asymptotic
  Combinatorics (V.~Kaimanovich and A.~Lodkin, eds.), 2, vol. 217, Transl. AMS,
  2006, pp.~9--22, arXiv:math--ph/0402064.

\bibitem[BO09]{Borodin2007}
\bysame, \emph{Infinite-dimensional diffusions as limits of random walks on
  partitions}, Prob. Theor. Rel. Fields \textbf{144} (2009), no.~1, 281--318,
  arXiv:0706.1034 [math.PR].

\bibitem[BO11a]{BorodinOlsh2011Prep}
\bysame, \emph{paper in preparation}.

\bibitem[BO11b]{BorodinOlsh2011Bouquet}
\bysame, \emph{{The Young bouquet and its boundary}}, arXiv:1110.4458
  [math.RT].

\bibitem[BOO00]{Borodin2000b}
A.~Borodin, A.~Okounkov, and G.~Olshanski, \emph{{Asymptotics of Plancherel
  measures for symmetric groups}}, J. Amer. Math. Soc. \textbf{13} (2000),
  no.~3, 481--515, arXiv:math/9905032 [math.CO].

\bibitem[Bor99]{Borodin1997}
A.~Borodin, \emph{Multiplicative central measures on the {S}chur graph}, Jour.
  Math. Sci. (New York) \textbf{96} (1999), no.~5, 3472--3477, in Russian: Zap.
  Nauchn. Sem. POMI {\bf{}240\/} (1997), 44--52, 290--291.

\bibitem[Bor01]{borodin2000harmonic}
A.M. Borodin, \emph{{Harmonic analysis on the infinite symmetric group, and the
  Whittaker kernel}}, Saint-Petersburg Math. J. \textbf{12} (2001), no.~5,
  733--759.

\bibitem[Boy83]{Boyer1983}
R.~Boyer, \emph{Infinite traces of {AF-algebras} and characters of
  {$U(\infty)$}}, J. Operator Theory \textbf{9} (1983), 205--236.

\bibitem[BT09]{Berele2009}
A.~Berele and B.E. Tenner, \emph{Doubly symmetric functions}, arXiv:0903.5306
  [math.CO].

\bibitem[Daw93]{Dawson1991}
D.~Dawson, \emph{Measure-valued {M}arkov processes}, \'{E}cole d'\'{E}t{\'e} de
  {P}robabilit{\'e}s de {S}aint-{F}lour {XXI}---1991, Lecture Notes in Math.,
  vol. 1541, Springer, Berlin, 1993, pp.~1--260.

\bibitem[DH11]{Desrosiers2011Laguerre}
P.~Desrosiers and M.~Hallnas, \emph{{Hermite and Laguerre symmetric functions
  associated with operators of Calogero-Moser-Sutherland type}},
  arXiv:1103.4593 [math.QA].

\bibitem[Edr52]{Edrei1952}
A.~Edrei, \emph{On the generating functions of totally positive sequences.
  {II}}, J. Analyse Math. \textbf{2} (1952), 104--109.

\bibitem[Eie83]{Eie83}
B.~Eie, \emph{The generalized {B}essel process corresponding to an
  {O}rnstein-{U}hlenbeck process}, Scand. J. Statist. \textbf{10} (1983),
  no.~3, 247--250.

\bibitem[EK81]{Ethier1981}
S.N. Ethier and T.G. Kurtz, \emph{The {I}nfinitely-{M}any-{N}eutral-{A}lleles
  {D}iffusion {M}odel}, Advances in Applied Probability \textbf{13} (1981),
  no.~3, 429--452.

\bibitem[EK86]{Ethier1986}
\bysame, \emph{Markov processes: {C}haracterization and convergence},
  Wiley-Interscience, New York, 1986.

\bibitem[EM91]{etheridge1991note}
A.~Etheridge and P.~March, \emph{A note on superprocesses}, Probab. Theory
  Related Fields \textbf{89} (1991), no.~2, 141--147.

\bibitem[Erd53]{Erdelyi1953}
A.~Erd{\'e}lyi (ed.), \emph{{Higher transcendental functions}}, McGraw--Hill,
  1953.

\bibitem[Eth00]{Etheridge:2000fk}
A.~Etheridge, \emph{An introduction to superprocesses}, University Lecture
  Series, vol.~20, American Mathematical Society, Providence, RI, 2000.

\bibitem[Ewe79]{Ewens1979}
W.~J. Ewens, \emph{{M}athematical {P}opulation {G}enetics}, Springer-Verlag,
  Berlin, 1979.

\bibitem[Fen10]{Feng2010book}
S.~Feng, \emph{{The Poisson-Dirichlet distributions and related topics: Models
  and asymptotic behaviours}}, Springer, 2010.

\bibitem[Fom79]{fomin1979thesis}
S.~Fomin, \emph{Two-dimensional growth in dedekind lattices}, Master's thesis,
  Leningrad State University, 1979.

\bibitem[Fom94]{fomin1994duality}
\bysame, \emph{{Duality of graded graphs}}, Journal of Algebraic Combinatorics
  \textbf{3} (1994), no.~4, 357--404.

\bibitem[FS98]{fomin1997rim}
S.~Fomin and D.~Stanton, \emph{{Rim hook lattices}}, St. Petersburg
  Mathematical Journal \textbf{9} (1998), no.~5, 1007--1016, Translated from
  Algebra i Analiz, {\bf{}9\/} (1997), no. 5, 140--150.

\bibitem[Ful05]{Fulman2005}
J.~Fulman, \emph{Stein's method and {P}lancherel measure of the symmetric
  group}, Trans. Amer. Math. Soc. \textbf{357} (2005), no.~2, 555--570,
  arXiv:math/0305423 [math.RT].

\bibitem[Ful09a]{Fulman2007}
\bysame, \emph{Commutation relations and {M}arkov chains}, Prob. Theory Rel.
  Fields \textbf{144} (2009), no.~1, 99--136, arXiv:0712.1375 [math.PR].

\bibitem[Ful09b]{Fulman2009Trees}
\bysame, \emph{Mixing time for a random walk on rooted trees}, Electronic
  Journal of Combinatorics \textbf{16} (2009), R139, arXiv:0908.1141 [math.CO].

\bibitem[Hof03]{Hoffman2003TreesHopf}
M.E. Hoffman, \emph{Combinatorics of rooted trees and {H}opf algebras}, Trans.
  Amer. Math. Soc. \textbf{355} (2003), no.~9, 3795--3811 (electronic),
  arXiv:math/0201253 [math.CO].

\bibitem[Iva99]{IvanovNewYork3517-3530}
V.~Ivanov, \emph{The {D}imension of {S}kew {S}hifted {Y}oung {D}iagrams, and
  {P}rojective {C}haracters of the {I}nfinite {S}ymmetric {G}roup}, Jour. Math.
  Sci. (New York) \textbf{96} (1999), no.~5, 3517--3530, in Russian: Zap.
  Nauchn. Sem. POMI {\bf{}240\/} (1997), 115-135, arXiv:math/0303169 [math.CO].

\bibitem[Jir64]{Jirina1964}
M.~Jirina, \emph{Branching processes with measure-valued states}, Trans. Third
  Prague Conf. on Inf. Th., 1964, pp.~333--357.

\bibitem[Ker89]{Kerov1989}
S.~Kerov, \emph{{C}ombinatorial examples in the theory of {AF}-algebras},
  Zapiski Nauchn. Semin. LOMI \textbf{172} (1989), 55--67, English translation:
  J. Soviet Math., {\bf{}59\/} (1992), 1063-1071.

\bibitem[Ker00]{Kerov2000}
\bysame, \emph{Anisotropic {Y}oung diagrams and {J}ack symmetric functions},
  Functional Analysis and Its Applications \textbf{34} (2000), no.~1, 41--51,
  arXiv:math/9712267 [math.CO].

\bibitem[Ker03]{Kerov-book}
S.~Kerov, \emph{Asymptotic representation theory of the symmetric group and its
  applications in analysis}, vol. 219, AMS, Translations of Mathematical
  Monographs, 2003.

\bibitem[KM57]{KMG57BDClassif}
S.~Karlin and J.~McGregor, \emph{The classification of birth and death
  processes}, Trans. Amer. Math. Soc. \textbf{86} (1957), 366--400.

\bibitem[KM58]{KMG58Linear}
\bysame, \emph{Linear growth, birth and death processes}, J. Math. Mech.
  \textbf{7} (1958), 643--662.

\bibitem[K{\"o}n05]{Konig2005}
W.~K{\"o}nig, \emph{{Orthogonal polynomial ensembles in probability theory}},
  Probab. Surv. \textbf{2} (2005), 385--447, arXiv:math/0403090 [math.PR].

\bibitem[Koo82]{Koornwinder1982Krawtchouk}
T.H. Koornwinder, \emph{Krawtchouk polynomials, a unification of two different
  group theoretic interpretations}, SIAM J. Math. Anal. \textbf{13} (1982),
  no.~6, 1011--1023.

\bibitem[KOO98]{Kerov1998}
S.~Kerov, A.~Okounkov, and G.~Olshanski, \emph{{T}he boundary of {Y}oung graph
  with {J}ack edge multiplicities}, Intern. Math. Research Notices \textbf{4}
  (1998), 173--199, arXiv:q-alg/9703037.

\bibitem[KOV93]{Kerov1993}
S.~Kerov, G.~Olshanski, and A.~Vershik, \emph{Harmonic analysis on the infinite
  symmetric group. {A} deformation of the regular representation}, Comptes
  Rendus Acad. Sci. Paris Ser. I \textbf{316} (1993), 773--778.

\bibitem[KOV04]{Kerov2004}
\bysame, \emph{Harmonic analysis on the infinite symmetric group}, Invent.
  Math. \textbf{158} (2004), no.~3, 551--642, arXiv:math/0312270 [math.RT].

\bibitem[KS96]{Koekoek1996}
R.~Koekoek and R.F. Swarttouw, \emph{{The Askey-scheme of hypergeometric
  orthogonal polynomials and its q-analogue}}, Tech. report, Delft University
  of Technology and Free University of Amsterdam, 1996.

\bibitem[Kup94]{Kuperberg1994SCPP}
G.~Kuperberg, \emph{{Self-complementary plane partitions by Proctor's minuscule
  method}}, European J. Combin. \textbf{15} (1994), no.~6, 545--553,
  arXiv:math/9411239 [math.CO].

\bibitem[Mac95]{Macdonald1995}
I.G. Macdonald, \emph{Symmetric functions and {H}all polynomials}, 2nd ed.,
  Oxford University Press, 1995.

\bibitem[Mat05]{Matsumoto2005}
S.~Matsumoto, \emph{{Correlation functions of the shifted Schur measure}}, J.
  Math. Soc. Japan, vol. \textbf{57} (2005), no.~3, 619--637,
  arXiv:math/0312373 [math.CO].

\bibitem[Mat08]{matsumoto2008jack}
\bysame, \emph{{Jack deformations of Plancherel measures and traceless Gaussian
  random matrices}}, Electronic Journal of Combinatorics \textbf{15} (2008),
  no.~R149, 1, arXiv:0810.5619 [math.CO].

\bibitem[Nel59]{nelson1959analytic}
E.~Nelson, \emph{{Analytic vectors}}, Ann. Math. \textbf{2} (1959), no.~70,
  572--615.

\bibitem[Oko96]{Okounkov1996quantumImm}
A.~Okounkov, \emph{Quantum immanants and higher {C}apelli identities},
  Transform. Groups \textbf{1} (1996), no.~1-2, 99--126, arXiv:q-alg/9602028.

\bibitem[Oko00]{okounkov2000random}
\bysame, \emph{{Random matrices and random permutations}}, International
  Mathematics Research Notices \textbf{2000} (2000), no.~20, 1043--1095,
  arXiv:math/9903176 [math.CO].

\bibitem[Oko01a]{okounkov2001infinite}
\bysame, \emph{{Infinite wedge and random partitions}}, Selecta Mathematica,
  New Series \textbf{7} (2001), no.~1, 57--81, arXiv:math/9907127 [math.RT].

\bibitem[Oko01b]{Okounkov2001a}
\bysame, \emph{{S}{L}(2) and z-measures}, Random matrix models and their
  applications (P.~M. Bleher and A.~R. Its, eds.), Mathematical Sciences
  Research Institute Publications, vol. {\bf{}40\/}, pp.~407--420, Cambridge
  Univ. Press, 2001, arXiv:math/0002135 [math.RT].

\bibitem[Oko02]{Okounkov2002}
\bysame, \emph{Symmetric functions and random partitions}, Symmetric functions
  2001: Surveys of Developments and Perspectives (S.~Fomin, ed.), Kluwer
  Academic Publishers, 2002, arXiv:math/0309074 [math.CO].

\bibitem[Ols08]{Olshanski-fockone}
G.~Olshanski, \emph{{Fock Space and Time-dependent Determinantal Point
  Processes}}, unpublished work, 2008.

\bibitem[Ols10a]{Olshanski2009}
\bysame, \emph{Anisotropic {Y}oung diagrams and infinite-dimensional diffusion
  processes with the {J}ack parameter}, International Mathematics Research
  Notices \textbf{2010} (2010), no.~6, 1102--1166, arXiv:0902.3395 [math.PR].

\bibitem[Ols10b]{Olshanski2010LaguerreMeixner}
\bysame, \emph{Laguerre and {M}eixner symmetric functions, and
  infinite-dimensional diffusion processes}, Zap. Nauchn. Sem. S.-Peterburg.
  Otdel. Mat. Inst. Steklov. (POMI) \textbf{378} (2010), no.~Teoriya
  Predstavlenii, Dinamicheskie Sistemy, Kombinatornye Metody. XVIII, 81--110,
  230, arXiv:1009.2037 [math.CO].

\bibitem[Ols11]{Olshanski2011Meixner}
\bysame, \emph{{Laguerre and Meixner orthogonal bases in the algebra of
  symmetric functions}}, arXiv:1103.5848 [math.CO].

\bibitem[OO97a]{OkounkovOlshanskiJack1996}
A.~Okounkov and G.~Olshanski, \emph{Shifted {J}ack polynomials, binomial
  formula, and applications}, Math. Res. Lett. \textbf{4} (1997), no.~1,
  69--78, arXiv:q-alg/9608020.

\bibitem[OO97b]{OkounkovOlshanski1996ShiftSchur}
\bysame, \emph{Shifted {S}chur functions}, Algebra i Analiz \textbf{9} (1997),
  no.~2, 73--146, translation in St. Petersburg Math. J. {\bf9} (1998), no. 2,
  239---300, arXiv:q-alg/9605042.

\bibitem[OO98]{OkOl1998}
\bysame, \emph{{Asymptotics of Jack polynomials as the number of variables goes
  to infinity }}, Int. Math. Res. Notices \textbf{1998} (1998), no.~13,
  641--682, arXiv:q-alg/9709011.

\bibitem[OR03]{okounkov2003correlation}
A.~Okounkov and N.~Reshetikhin, \emph{{Correlation function of Schur process
  with application to local geometry of a random 3-dimensional Young diagram}},
  Journal of the American Mathematical Society \textbf{16} (2003), no.~3,
  581--603, arXiv:math/0107056 [math.CO].

\bibitem[ORV03]{OlshRegVer2003}
G.~Olshanski, A.~Regev, and A.~Vershik, \emph{{Frobenius--Schur functions}},
  {Studies in Memory of Issai Schur} (A.~Joseph, A.~Melnikov, and
  R.~Rentschler, eds.), Progress in Mathematics, vol. 210, Birkhauser, 2003,
  arXiv:math/0110077 [math.CO], pp.~251--300.

\bibitem[Pet09]{Petrov2007}
L.~Petrov, \emph{A two-parameter family of infinite-dimensional diffusions in
  the {K}ingman simplex}, Functional Analysis and Its Applications \textbf{43}
  (2009), no.~4, 279--296, arXiv:0708.1930 [math.PR].

\bibitem[Pet10a]{petrov2009eng}
\bysame, \emph{{Random walks on strict partitions}}, Journal of Mathematical
  Sciences \textbf{168} (2010), no.~3, 437--463, in Russian: Zap. Nauchn. Sem.
  POMI {\bf{}373\/} (2009), 226--272, arXiv:0904.1823 [math.PR].

\bibitem[Pet10b]{Petrov2010}
\bysame, \emph{{Random Strict Partitions and Determinantal Point Processes}},
  Electronic Communications in Probability \textbf{15} (2010), 162--175,
  arXiv:1002.2714 [math.PR].

\bibitem[Pet11]{Petrov2010Pfaffian}
\bysame, \emph{Pfaffian stochastic dynamics of strict partitions}, 2011,
  arXiv:1011.3329v2 [math.PR], to appear in Electronic Journal of Probability.

\bibitem[Pit92]{Pitman1992}
J.~Pitman, \emph{The two-parameter generalization of {E}wens' random partition
  structure}, Technical report 345, Dept. Statistics, U. C. Berkeley, 1992,
  http://www.stat.berkeley.edu/tech-reports/.

\bibitem[Pro82]{proctor1982solution}
R.A. Proctor, \emph{Solution of two difficult combinatorial problems with
  linear algebra}, The American Mathematical Monthly \textbf{89} (1982),
  no.~10, 721--734.

\bibitem[Pro84]{Proctor1984Bruhat}
\bysame, \emph{Bruhat lattices, plane partition generating functions, and
  minuscule representations}, European J. Combin. \textbf{5} (1984), no.~4,
  331--350.

\bibitem[PS02]{PhahoferSpohn2002}
M.~Pr{\"a}hofer and H.~Spohn, \emph{{Scale invariance of the PNG droplet and
  the Airy process}}, J. Stat. Phys. \textbf{108} (2002), 1071--1106, arXiv:
  math.PR/0105240.

\bibitem[Puk64]{Pukanszky_SL2_1964}
L.~Pukanszky, \emph{{The Plancherel formula for the universal covering group of
  $SL(2,\mathbb{R})$}}, Mathematische Annalen \textbf{156} (1964), no.~2,
  96--143.

\bibitem[PY97]{Pitman1997}
J.~Pitman and M.~Yor, \emph{Two-parameter {P}oisson-{D}irichlet distribution
  derived from a stable subordinator}, The Annals of Probability \textbf{25}
  (1997), no.~2, 855--900.

\bibitem[Rie23]{riesz1923probleme}
M.~Riesz, \emph{Sur le probleme des moments et le th{\'e}oreme de parseval
  correspondant}, Acta Litt. Acad. Sci. Szeged \textbf{1} (1923), 209--225.

\bibitem[Roz99]{rozhkovskaya1997multiplicative}
N.~Rozhkovskaya, \emph{{Multiplicative distributions on Young graph}}, Jour.
  Math. Sci. (New York) \textbf{96} (1999), no.~5, 3600--3608, in Russian: Zap.
  Nauchn. Sem. POMI {\bf{}240\/} (1997), 245--256.

\bibitem[RW09]{Ruggiero2009a}
M.~Ruggiero and S.G. Walker, \emph{{Countable representation for infinite
  dimensional diffusions derived from the two-parameter Poisson-Dirichlet
  process}}, Electronic Communications in Probability \textbf{14} (2009),
  501--517.

\bibitem[Sta88]{stanley1988differential}
R.~Stanley, \emph{Differential posets}, Journal of the American Mathematical
  Society \textbf{1} (1988), no.~4, 919--961.

\bibitem[Sta90]{stanley1990variations}
\bysame, \emph{Variations on differential posets}, Invariant theory and
  tableaux ({M}inneapolis, {MN}, 1988), IMA Vol. Math. Appl., vol.~19,
  Springer, New York, 1990, pp.~145--165.

\bibitem[Sta97]{Stanley1997}
\bysame, \emph{Enumerative {C}ombinatorics. {V}ol. 1}, Cambridge University
  Press, Cambridge, 1997, With a foreword by Gian-Carlo Rota, Corrected reprint
  of the 1986 original.

\bibitem[Sta99]{Stanley1999}
\bysame, \emph{Enumerative {C}ombinatorics. {V}ol. 2}, Cambridge University
  Press, Cambridge, 1999, With a foreword by Gian-Carlo Rota and appendix 1 by
  Sergey Fomin.

\bibitem[Ste85]{Stembridge1985}
J.~Stembridge, \emph{A characterization of supersymmetric polynomials}, J.
  Algebra \textbf{95} (1985), 439--444.

\bibitem[Ste94]{Stembridge1994minuscule}
\bysame, \emph{On minuscule representations, plane partitions and involutions
  in complex {L}ie groups}, Duke Math. J. \textbf{73} (1994), no.~2, 469--490.

\bibitem[Str10a]{Strahov2009}
E.~Strahov, \emph{{The z-measures on partitions, Pfaffian point processes, and
  the matrix hypergeometric kernel}}, Advances in mathematics \textbf{224}
  (2010), no.~1, 130--168, arXiv:0905.1994 [math-ph].

\bibitem[Str10b]{strahov2009z}
\bysame, \emph{{Z-measures on partitions related to the infinite Gelfand pair
  $(S (2\infty), H (\infty))$}}, Journal of Algebra \textbf{323} (2010), no.~2,
  349--370, arXiv:0904.1719 [math.RT].

\bibitem[Tho64]{Thoma1964}
E.~Thoma, \emph{Die unzerlegbaren, positive-definiten {K}lassenfunktionen der
  abz\"ahlbar unendlichen, symmetrischen {G}ruppe}, Math. Zeitschr \textbf{85}
  (1964), 40--61.

\bibitem[TW04]{Tracy2004}
C.A. Tracy and H.~Widom, \emph{{A Limit Theorem for Shifted Schur Measures}},
  Duke Mathematical Journal \textbf{123} (2004), 171--208, arXiv:math/0210255
  [math.PR].

\bibitem[Ver96]{Vershik1996StatMech}
A.~Vershik, \emph{Statistical mechanics of combinatorial partitions, and their
  limit shapes}, Funct. Anal. Appl. \textbf{30} (1996), 90--105.

\bibitem[VK81a]{VK81AsymptoticTheory}
A.~Vershik and S.~Kerov, \emph{Asymptotic theory of the characters of a
  symmetric group}, Funktsional. Anal. i Prilozhen. \textbf{15} (1981), no.~4,
  15--27, 96.

\bibitem[VK81b]{VK1981Characters}
\bysame, \emph{Characters and factor representations of the infinite symmetric
  group}, Dokl. Akad. Nauk SSSR \textbf{257} (1981), no.~5, 1037--1040.

\bibitem[VK82]{VK82CharactersU}
\bysame, \emph{Characters and factor-representations of the infinite unitary
  group}, Dokl. Akad. Nauk SSSR \textbf{267} (1982), no.~2, 272--276.

\bibitem[VK84]{VK1984Kfunctor}
\bysame, \emph{Characters, factor representations and {$K$}-functor of the
  infinite symmetric group}, Operator algebras and group representations,
  {V}ol. {II} ({N}eptun, 1980), Monogr. Stud. Math., vol.~18, Pitman, Boston,
  MA, 1984, pp.~23--32.

\bibitem[VK88]{Vilenkin-Klimyk-DAN_UKR_1988}
N.Y. Vilenkin and A.U. Klimyk, \emph{{Representations of the group $SU(1, 1)$,
  and the Krawtchouk-Meixner functions}}, Dokl. Akad. Nauk Ukrain. SSR Ser. A
  (1988), no.~6, 12--16.

\bibitem[VK90]{Kerov1990}
A.~Vershik and S.~Kerov, \emph{The {G}rothendieck {G}roup of the {I}nfinite
  {S}ymmetric {G}roup and {S}ymmetric {F}unctions with the {E}lements of the
  ${K}_0$-functor theory of {AF}-algebras}, Adv. Stud. Contemp. Math., Gordon
  and Breach, \textbf{7} (1990), 36--114.

\bibitem[VK95]{Vilenkin-Klimyk-ITOGI1995-en}
N.Y. Vilenkin and A.U. Klimyk, \emph{{Representations of Lie groups and special
  functions}}, {Representation Theory and Noncommutative Harmonic Analysis II}
  (A.A. Kirillov, ed.), Springer, 1995, (translation of VINITI vol. 59, 1990),
  pp.~137--259.

\bibitem[VN11]{VershikNikitin2011}
A.~Vershik and P.~Nikitin, \emph{Description of the characters and factor
  representations of the infinite symmetric inverse semigroup}, Functional
  Analysis and Its Applications \textbf{45} (2011), no.~1, 13--24,
  arXiv:1102.4425 [math.RT].

\bibitem[Voi76]{Voiculescu1976}
D.~Voiculescu, \emph{Representations factorielles de type {$II_1$} de
  {$U(\infty)$}}, J. Math. Pures Appl. \textbf{55} (1976), 1--20.

\end{thebibliography}
\end{document}